\documentclass[12pt,reqno]{amsart}
\usepackage{amssymb}
\usepackage{diagrams}
\usepackage{graphicx}
\usepackage{amsmath}
\usepackage{a4wide}
\usepackage{version}
\usepackage{mathrsfs}
\usepackage{tikz}
\usepackage{mathtools}
\usetikzlibrary{arrows,decorations.markings, matrix}
\usepackage{tikz-cd}
\usepackage{hyperref}
\usepackage{float}
\usepackage[all]{xy}

\numberwithin{equation}{section}

\newtheorem{thm}{Theorem}[subsection]
\newtheorem{lem}[thm]{Lemma}
\newtheorem{prop}[thm]{Proposition}
\newtheorem{cor}[thm]{Corollary}

\theoremstyle{definition}
\newtheorem{definition}[thm]{Definition}
\newtheorem{example}[thm]{Example}
\newtheorem{rmk}[thm]{Remark}
\newtheorem{conj}[thm]{Conjecture}

\newcommand{\Z}{\mathbb{Z}}
\renewcommand{\k}{\mathbf{k}}

\newcommand{\mf}{\operatorname{mf}}
\newcommand{\com}{\operatorname{com}}

\newcommand{\id}{\operatorname{id}}

\newcommand{\ov}{\overline}

\newcommand{\XX}{{\mathcal X}}
\newcommand{\YY}{{\mathcal Y}}
\newcommand{\ZZ}{{\mathcal Z}}
\newcommand{\HH}{{\mathcal H}}

\renewcommand{\SS}{{\mathcal S}}

\newcommand{\Hom}{\operatorname{Hom}}

\newcommand{\End}{\operatorname{End}}

\renewcommand{\a}{\alpha}
\renewcommand{\b}{\beta}
\newcommand{\om}{\omega}
\newcommand{\De}{\Delta}
\newcommand{\la}{\lambda}

\newcommand{\R}{{\mathbb R}}

\newcommand{\Ga}{\Gamma}

\newcommand{\wt}{\widetilde}
\newcommand{\ot}{\otimes}
\newcommand{\sing}{\operatorname{sing}}

\newcommand{\sub}{\subset}

\renewcommand{\k}{\mathbf{k}}

\newcommand{\bw}{{\mathbf w}}

\renewcommand{\mod}{\operatorname{mod}}

\newcommand{\und}{\underline}

\newcommand{\OO}{{\mathcal O}}

\newcommand{\Sym}{\operatorname{Sym}}

\newcommand{\Si}{\Sigma}

\newcommand{\coker}{\operatorname{coker}}

\newcommand{\BB}{{\mathcal B}}

\newcommand{\G}{{\mathbb G}}

\newcommand{\hra}{\hookrightarrow}
\newcommand{\lan}{\langle}
\newcommand{\ran}{\rangle}
\newcommand{\Coh}{\operatorname{Coh}}
\newcommand{\GG}{{\mathcal G}}

\newcommand{\UU}{{\mathcal U}}

\newcommand{\MF}{\operatorname{MF}}
\newcommand{\crit}{\operatorname{crit}}

\newcommand{\Spec}{\operatorname{Spec}}
\newcommand{\Sing}{\operatorname{Sing}}

\renewcommand{\P}{{\mathbb P}}

\newcommand{\si}{\sigma}

\newcommand{\ga}{\gamma}
\newcommand{\de}{\delta}

\renewcommand{\ker}{\operatorname{ker}}
\newcommand{\im}{\operatorname{im}}

\newcommand{\A}{{\mathbb A}}

\newcommand{\qcoh}{\operatorname{qcoh}}

\newcommand{\w}{{\mathbf w}}
\newcommand{\bbL}{{\mathbb L}}

\title[]{Homological mirror symmetry for the symmetric squares of punctured spheres}
\author{Yank\i\ Lekili}
\author{Alexander Polishchuk}

\address{Imperial College London}
\address{University of Oregon, National Research University Higher School of Economics, and
Korea Institute for Advanced Study}


\begin{document}

\begin{abstract}
        For an appropriate choice of a $\mathbb{Z}$-grading structure, we prove that the wrapped Fukaya category of the symmetric square of a $(k+3)$-punctured sphere, i.e. the Weinstein manifold given as the complement of $(k+3)$ generic lines in $\mathbb{C}P^2$ is quasi-equivalent to the derived category of coherent sheaves on a singular surface $\mathcal{Z}_{2,k}$ constructed as the boundary of a toric Landau-Ginzburg model $(\mathcal{X}_{2,k}, \w_{2,k})$. We do this by first constructing a quasi-equivalence between certain categorical resolutions of both sides and then localising. 
        We also provide a general homological mirror symmetry conjecture concerning all the higher symmetric powers of punctured spheres. The corresponding toric LG-models $(\mathcal{X}_{n,k},\w_{n,k})$ are constructed from the combinatorics of curves on the punctured sphere and are related to small toric resolutions of the singularity $x_1\ldots x_{n+1}=v_1\ldots v_k$.
  
\end{abstract}

\maketitle

\section{Introduction}

\subsection{Motivating examples}

We begin by considering the family \[ \mathcal{Y}_{1,1} \coloneqq \mathrm{Spec}\, \mathbb{C}[x,y,V]/(xy-V)\] over $\mathbb{A}^1_V$ as a ``toric degeneration'' of the general fiber to the singular variety $\mathcal{Z}_{1,1} \coloneqq \{ xy = 0\}$ which is a gluing of two toric varieties $\mathbb{A}^1_x$ and $\mathbb{A}^1_y$ at their toric boundary $0$. The symplectic mirror to this degeneration is well understood \cite{aurouxspec}. Namely, we consider the cotangent bundle of the circle $T^*\mathbb{S}^1$ as our symplectic manifold and $D_{1,1} \coloneqq \{ z \}$  a marked point on $T^*\mathbb{S}^1$ as a symplectic divisor. Then, one has an equivalence of pre-triangulated categories over $\mathbb{C}[V]$ 
\begin{align}\label{logfamily} \mathcal{W}(T^*\mathbb{S}^1, D_{1,1}) \simeq D^b \mathrm{Coh}(\mathcal{Y}_{1,1}) \end{align}
        where the left hand side is the wrapped Fukaya category of $T^*\mathbb{S}^1$ relative to the divisor $D_{1,1}$ (cf. \cite{seidelICM} \cite{LPe} \cite{sheridan2}) and the right hand side is the bounded derived category of $\mathcal{Y}_{1,1}$ over $\mathbb{A}^1_V$. 

The equivalence \eqref{logfamily} fits into the general framework of Gross-Siebert, Gross-Hacking-Keel \cite{GHK}, Abouzaid-Auroux-Katzarkov \cite{AAK}, where the general fiber is a log-Calabi-Yau variety and the degeneration given is a ``large complex structure limit'' degeneration of such a log-Calabi-Yau variety to a union of toric varieties. Even though in the case of this specific example both sides can be computed rather easily, in general this is a harder  problem. A viable strategy to establish this equivalence is to first look at the special fiber and prove the following equivalence
\begin{align} \label{node} \mathcal{W}(V_{1,1}) \simeq D^b \mathrm{Coh}(\mathcal{Z}_{1,1}) \end{align}
        for $V_{1,1}\coloneqq T^*\mathbb{S}^1 \setminus D_{1,1}$. Then one can try to deduce the equivalence \eqref{logfamily} via deformation theory (certain subcategories of those in equivalence \eqref{node} deform to give rise to the equivalence \eqref{logfamily}).

Note that $V_{1,1}$ is the pair-of-pants surface. To establish equivalences of the form \eqref{node}, we introduced a method via categorical partial compactifications in \cite{LP}, \cite{LPsymhms} 
(a similar idea was also applied in the log Calabi-Yau setting in \cite{keating}, \cite{hackingkeating}). On the A-side we introduce a stop $\Lambda^{\circ}$, that is a symplectic hypersurface on the contact boundary of $T^*\mathbb{S}^1 \setminus D_{1,1}$ and consider the partially wrapped Fukaya category, and on the B-side we partially compactify $\mathcal{Z}_{1,1}$ to $\overline{\mathcal{Z}}_{1,1}$ which is obtained by compactifying $\mathbb{A}^1_y$ to $\mathbb{P}^1_y$ by adding a point $y=\infty$. Thus, $\overline{\mathcal{Z}}_{1,1}$ is a union of $\mathbb{A}^1_x$ and $\mathbb{P}^1_{y}$ which intersect transversely at a point. We establish directly an easier equivalence \begin{align} \label{partcpct} 
        \mathcal{W}(V_{1,1}, \Lambda^{\circ}) \simeq D^b \mathrm{Coh} (\overline{\mathcal{Z}}_{1,1}) 
\end{align}
and then show that the equivalence of \eqref{node} is obtained from \eqref{partcpct} by localization.

The Liouville domain $V_{1,1}$ with the stop $\Lambda^\circ$ is given by the following picture.

\begin{figure}[ht!]
\centering
\begin{tikzpicture}

\begin{scope}[scale=0.8]

\tikzset{
  with arrows/.style={
    decoration={ markings,
      mark=at position #1 with {\arrow{>}}
    }, postaction={decorate}
  }, with arrows/.default=2mm,
}

\tikzset{vertex/.style = {style=circle,draw, fill,  minimum size = 2pt,inner        sep=1pt}}
\def \radius {1.5cm}

\foreach \s in {1,2,3,4,5,6,7,8,9,10} {
    \draw[thick] ([shift=({360/10*(\s)}:\radius-1.1cm)]-1,0) arc ({360/10 *(\s)}:{360/10*(\s+1)}:\radius-1.1cm);
    \draw[thick] ([shift=({360/10*(\s)}:\radius-1.1cm)]1,0) arc ({360/10 *(\s)}:{360/10*(\s+1)}:\radius-1.1cm);

   \draw[thick] ([shift=({360/10*(\s)}:\radius+1.5cm)]0,0) arc ({360/10 *(\s)}:{360/10*(\s+1)}:\radius+1.5cm);
}

\draw[with arrows] ([shift=({360/10*(7)}:\radius-1.1cm)]-1,0) arc ({360/10 *(7)}:{360/10*(5)}:\radius-1.1cm);
\draw[with arrows] ([shift=({360/10*(7)}:\radius-1.1cm)]1,0) arc ({360/10*(7)}:{360/10*(5)}:\radius-1.1cm);

\draw[with arrows]({360/10 * (3)}:\radius+1.5cm) arc ({360/10 *(3)}:{360/10*(4)}:\radius+1.5cm);

\draw [blue] ([shift=({360/10*(5)}:\radius-1.1cm)]-1,0) -- ([shift=({360/10*(5)}:\radius+1.5cm)]0,0);
\draw [blue] ([shift=({360/10*(10)}:\radius-1.1cm)]-1,0) --  ([shift=({360/10*(5)}:\radius-1.1cm)]1,0);

\node[blue] at ([shift=({360/10*(4)}:\radius-1.1cm)]-2,0.05) {\small $L_{0}$};

\node[blue] at (0,0.3) {\small $L_{1}$};

\node[vertex] at  ([shift=({360/10*(2.5)}:\radius+1.5cm)]0,0) {};

\draw[purple] (0.2, 3) to[in=0,out=260] (0,2.7);
\draw[purple] (-0.2, 3) to[in=180,out=280] (0,2.7);

 \node[purple] at ([shift=({360/10*(2.5)}:\radius+1.5cm)]0,-0.7) {\small $T_1$};

\end{scope}

\end{tikzpicture}
    \caption{Pair-of-pants}
    \label{pop}
\end{figure}
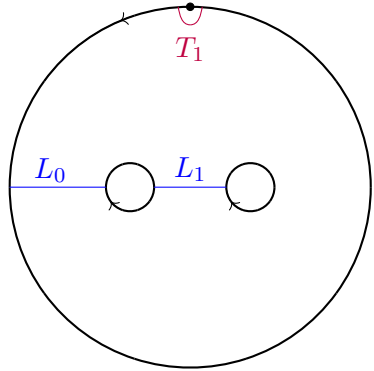

The equivalence \eqref{partcpct} is proved by matching the generators $L_0$ and $L_1$ with $\mathcal{O}_{\mathbb{P}^1_y}$ and $\mathcal{O}_{\mathbb{A}^1_{x}}$, respectively. To deduce the equivalence \eqref{node}, we localize at the subcategory generated by $T_1$ on the A-side and by the skyscraper sheaf $\mathcal{O}_{y=\infty}$ on the $B$-side. The reason that the equivalence \eqref{partcpct} is easier to establish is because the endomorphism algebra of our generators turns out to be formal which is not true before compactification.

The example that we considered above and the proof of the equivalences \eqref{logfamily}, \eqref{node},\eqref{partcpct} can be generalized in a straightforward way to the case of $(T^*\mathbb{S}^1, D_{1,k})$ where $D_{1,k}$ is now $k$ points (see \cite{LP} for a proof of \eqref{node}, \eqref{partcpct}) and the family $\mathcal{Y}_{1,k}$ is a $k$-dimensional family over $\mathbb{A}^k = \mathrm{Spec}\, \mathbb{C}[V_1,V_2,\ldots, V_k]$. The general fiber is again isomorphic to $\mathbb{C}^\times$ and the special fiber is a nodal chain given by the following 
figure: 

\begin{figure}[!h]
\begin{tikzpicture}
\draw (0,0) -- (1,1);
\draw (0.7,1) -- (1.7,0);
\draw (1.4,0) -- (2.4,1);
\draw (2.1,1) -- (3.1,0);
\draw (2.7,0) -- (3.7,1);
\draw (3.4,1) -- (4.4,0);

\node at (0,-0.3) {\tiny $\mathbb{A}^1_x$};
\node at (4.5,-0.3) {\tiny $\mathbb{A}^1_y$};
\end{tikzpicture}

        \caption{Mirror to $7$-punctured sphere, $V_{1,5}$}

\end{figure}
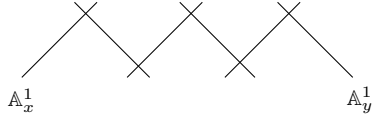

The two components at the ends are $\mathbb{A}^1_x$ and $\mathbb{A}^1_y$, and the ``interior components'' are isomorphic to $\mathbb{P}^1$.

Another way to represent this singular curve is via the image under the moment map induced by the toric action. This results in the following diagram where the rays on the left and right represent $\mathbb{A}^1_x$ and $\mathbb{A}^1_y$ and the interior intervals represent 4 components isomorphic to $\mathbb{P}^1$.

\begin{figure}[!h]
\begin{tikzpicture}
\draw (0,0) -- (6,0);
\fill (1,0) circle[radius=2pt];
\fill (2,0) circle[radius=2pt];
\fill (3,0) circle[radius=2pt];
\fill (4,0) circle[radius=2pt];
\fill (5,0) circle[radius=2pt];
\end{tikzpicture}

        \caption{Moment map picture of $\mathcal{Z}_{1,5}$ - the mirror to $V_{1,5}$}
\label{fg:intersection11}

\end{figure}
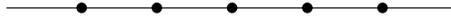

The $A$-side in the above example has a natural generalization to higher dimensions.
Namely, for $n,k\geq 0$, let $\{x_1,x_2, \ldots x_{n+1}\} \cup \{ v_1,\ldots, v_k\}$ be $(n+k+1)$ points in $\mathbb{P}^1_{\mathbb{C}}$ that are mutually distinct. We consider the exact symplectic manifold 
\[ V_{n,k} = \mathrm{Sym}^n (\mathbb{P}_{\mathbb{C}}^1 \setminus (\{x_1,x_2,\ldots x_{n+1}\} \cup \{v_1,v_2, \ldots, v_k\})) \] with the standard Liouville structure which can alternatively be described as the complement of $(n+k+1)$ generic hyperplanes in $\mathbb{P}_{\mathbb{C}}^n$. Note that $V_{n,0}$ can be symplectically identified with the cotangent bundle of the $n$-dimensional torus, 
\[ \mathrm{Sym}^n (\mathbb{P}_{\mathbb{C}}^1 \setminus (\{x_1,x_2,\ldots x_{n+1}\}) =  T^*\mathbb{T}^n,\] and we have $V_{n,k} = T^*\mathbb{T}^n \setminus D_{n,k}$ with
\[ D_{n,k} = \bigcup_{j=1}^k \{ v_j\} \times \mathrm{Sym}^{n-1} (\mathbb{P}_{\mathbb{C}}^1 \setminus (\{x_1,x_2,\ldots x_{n+1}\})\] 

In our previous work \cite{LPsymhms}, we studied (partially) wrapped Fukaya categories associated with $V_{n,k}$. In the case $k=1$, i.e., the case of $n$-dimensional pairs-of-pants, we proved a homological mirror symmetry statement that provides an equivalence of $\mathbb{Z}$-graded pre-triangulated categories over a regular commutative ring $\mathbf{k}$:
\begin{align} \mathcal{W}(V_{n,1}) \simeq \mathrm{D}^b \mathrm{Coh} (\mathcal{Z}_{n,1})  \simeq \MF_{\G_m}(\mathcal{Y}_{n,1} \times \mathbb{A}^1, \mathbf{w}_{n,1}) 
\end{align}
where $\mathcal{W}(V_{n,1})$ is the fully wrapped Fukaya category of the $n$-dimensional pair-of-pants $V_{n,1}$, the middle category is the derived category of coherent sheaves on the singular affine variety
\[ \mathcal{Z}_{n,1} = \mathrm{Spec}\, \mathbf{k}[x_1,\ldots, x_{n+1}]/(x_1x_2\ldots x_{n+1} )\]
and the third category is the $\mathbb{G}_m$-equivariant derived category of matrix factorizations on 
$\mathcal{Y}_{n,1}\times \mathbb{A}^1_{U_1}\simeq \mathbb{A}^{n+2}$, where 
\[ \mathcal{Y}_{n,1}\coloneqq \mathrm{Spec}\, \mathbb{C}[x_1,x_2,\ldots, x_{n+1},V_1]/(x_1x_2\ldots x_{n+1} - V_1)\simeq
\A^{n+1}\]  
for the potential $\mathbf{w}_{n,1}= U_1V_1$.

\subsection{Formulation of the main result}

In this paper, we give an explicit candidate for the $B$-side for any $n$ and $k$. To that end, inspired by the work of Abouzaid-Auroux-Katzarkov \cite{AAK}, we introduce a smooth toric variety $\mathcal{Y}_{n,k}$ over $\mathbb{A}^k = \mathrm{Spec}\, \k[V_1,V_2,\ldots, V_k]$ with the general fiber isomorphic to 
the torus $\G_m^n$ and the special fiber $\mathcal{Z}_{n,k}$ over $V_1=V_2=\ldots = V_k=0$ given by a union of toric varieties glued along toric strata
(see Sec.\ \ref{toric-constr-sec} for the construction). 
In fact, $\YY_{n,k}$ is a small toric resolution of the singularity $X_1\ldots X_{n+1}=V_1\ldots V_k$, in particular, it is Calabi-Yau
(see Sec.\ \ref{Toric-geom-sec}).
We also consider a Landau-Ginzburg version by considering the variety
\[ \mathcal{X}_{n,k} \coloneqq \mathcal{Y}_{n,k} \times \mathrm{Spec}\, \k[U_1,U_2,\ldots, U_k] \] 
with the potential 
\[ \mathbf{w}_{n,k} = U_1V_1+U_2V_2+\ldots +U_kV_k\]

There is a universal grading of the Fukaya category of $V_{n,k}$ with values in an abelian group 
$\bbL\simeq \Z^{n+k+1}$ which is compatible with a certain natural $\Z/2$-grading structure. In order to get a relation with the B-side we identify $\bbL$ with a natural grading group $\bbL_B$ on the B-side, which is a sublattice in
the character group of an algebraic torus acting on $\XX_{n,k}$ in a way that rescales
$\mathbf{w}_{n,k}$ according to the given character $l_0\in \bbL_B$.

The possible $\mathbb{Z}$-grading structures on $V_{n,k}$ is a torsor over $H^1(V_{n,k}) = \mathbb{Z}^{n+k}$. Among these there is a natural $\mathbb{Z}$-grading structure $\eta_0$ on $V_{n,k}$ which roughly speaking is determined by the property that the primitive chords around the $x_i$ have degree 0, and the primitive chords around $v_i$ have degree 2. This corresponds to the unique subgroup $\G_m$ in the above torus, such that all $x_i$ and $V_j$ have weight $0$ (and $U_j$ have weight $1$) with respect to this subgroup.

We formulate a conjectural homological mirror symmetry statement that generalizes the above result to arbitrary $n,k$.

\begin{conj} \label{mainconj} There exists a quasi-equivalence of $\mathbb{Z}$-graded pre-triangulated categories 
        \begin{align} \label{limitcat} \mathcal{W}(V_{n,k},\eta_0) \simeq D^b \mathrm{Coh} (\mathcal{Z}_{n,k}),
       \end{align}
       and a quasi-equivalence of $\bbL$-graded pre-triangulated categories
       \begin{align} \label{LG-limitcat} \mathcal{W}(V_{n,k})\simeq \BB_{n,k},\end{align}
               where $\BB_{n,k}$ is the category of $\bbL_B$-graded equivariant matrix factorizations of $\mathbf{w}_{n,k}$ on $\XX_{n,k}$.
Moreover,  by considering deformations over $\mathbb{A}^k$, there exists a quasi-equivalence
        \begin{align} \mathcal{W}(T^*\mathbb{T}^n, D_{n,k},\eta_0) \simeq D^b \mathrm{Coh}(\mathcal{Y}_{n,k}) \end{align} 
\end{conj}

We give a precise definition of the $\bbL$-graded dg-category of equivariant 
matrix factorizations $\BB_{n,k}$ in Sec.\ \ref{L-gr-G-mf}.
Note that \eqref{limitcat} follows from \eqref{LG-limitcat} by choosing a specific $\Z$-grading and applying localization theorems from \cite{auroux} on the A-side and from \cite{Isik} on the B-side.
Our main result is the proof of \eqref{limitcat} and \eqref{LG-limitcat} when $n=2$. Along the way, we also give a new proof for $n=1$.

\begin{thm}\label{main-thm} Equivalence  \eqref{limitcat} holds for $n=1$ and $n=2$.
\end{thm}

The equivalence \eqref{LG-limitcat} of $\bbL$-graded categories in the cases $n\le 2$ should also be within reach. We prove its 
$\Z$-graded version, as well as a version for the partially wrapped Fukaya category with one stop (see Theorem \ref{main-thm-bis} below). To deduce \eqref{LG-limitcat} one has to develop a setup for localization of $\bbL$-graded $A_\infty$ (or dg-)categories.

Figure \ref{toricdim2} describes the moment map picture of the singular variety $\mathcal{Z}_{2,k}$ for $k=4$. This is a two-dimensional analogue of Figure \ref{fg:intersection11}. 
The 3 components on the corners are all isomorphic to $\mathbb{A}^2$, 6 components isomorphic to $\mathbb{A}^1 \times \mathbb{P}^1$, 3 components isomorphic to $\mathbb{P}^1 \times \mathbb{P}^1$ and 3 components isomorphic to a blow-up of $\mathbb{A}^1 \times \mathbb{P}^1$. The markings $C_{ij}$ denote coordinates on one of the affine charts of each 
$\P^1$ along which the components are glued.

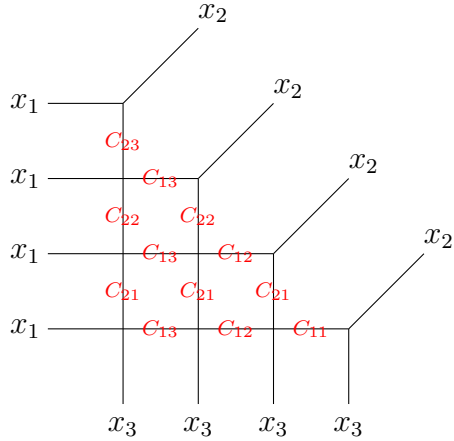
\begin{figure}[!h]
\begin{tikzpicture}

\draw (0,0) -- (4,0);
\draw (0,1) -- (3,1);
\draw (0,2) -- (2,2);
\draw (0,3) -- (1,3);

\draw (1,3) -- (1,-1);
\draw (2,2) -- (2,-1);
\draw (3,1) -- (3,-1);
\draw (4,0) -- (4,-1);

\draw (1,3) -- (2,4); 
\draw (2,2) -- (3,3); 
\draw (3,1) -- (4,2); 
\draw (4,0) -- (5,1);

\node at (-0.3,3) { $x_1$};
\node at (-0.3,2) { $x_1$};
\node at (-0.3,1) { $x_1$};
\node at (-0.3,0) { $x_1$};

\node at (1,-1.3) { $x_3$};
\node at (2,-1.3) { $x_3$};
\node at (3, -1.3) { $x_3$};
\node at (4, -1.3) { $x_3$};

\node at (2.2,4.2) { $x_2$};
\node at (3.2,3.2) { $x_2$};
\node at (4.2,2.2) { $x_2$};
\node at (5.2,1.2) { $x_2$};

        \node[red] at (1.5, 0) {\tiny $C_{13}$};
        \node[red] at (2.5, 0) {\tiny $C_{12}$};
        \node[red] at (3.5, 0) {\tiny $C_{11}$};
        \node[red] at (1.5, 1) {\tiny $C_{13}$};
        \node[red] at (2.5, 1) {\tiny $C_{12}$};
        \node[red] at (1.5, 2) {\tiny $C_{13}$};

        \node[red] at (1, 0.5) {\tiny $C_{21}$};
        \node[red] at (1, 1.5) {\tiny $C_{22}$};
        \node[red] at (1, 2.5) {\tiny $C_{23}$};
        \node[red] at (2, 0.5) {\tiny $C_{21}$};
        \node[red] at (2, 1.5) {\tiny $C_{22}$};
        \node[red] at (3, 0.5) {\tiny $C_{21}$};

\end{tikzpicture}

        \caption{Moment map picture of $\mathcal{Z}_{2,4}$ - the mirror to $V_{2,4}$}
\label{toricdim2}

\end{figure}

\subsection{Description of the toric variety}\label{toric-constr-sec}

We now explicitly describe the toric variety $\mathcal{Y}_{n,k}$. Let us set 
\[ \Sigma= \mathbb{P}^1_\mathbb{C} \setminus (\{x_1,x_2,\ldots x_{n+1}\} \cup \{v_1,v_2, \ldots, v_k\}), \] the $(n+k+1)$ punctured surface whose boundary components are marked by $x_1,\ldots,x_{n+1}$ and $v_1,\ldots, v_n$ as in Figure \ref{distinguishedcurves}.

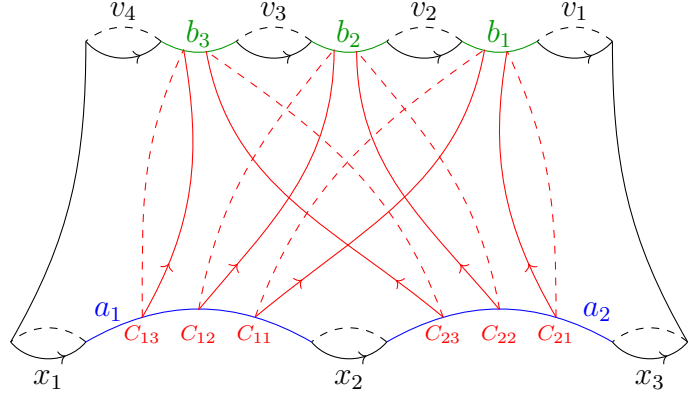
\begin{figure}[!h] \centering
\begin{tikzpicture}
\tikzset{
  with arrows/.style={
    decoration={ markings,
      mark=at position #1 with {\arrow{>}}
    }, postaction={decorate}
  }, with arrows/.default=8mm,
}

\draw[with arrows] (0,0) to [in=230, out=310]  (1,0);  
\draw[dashed] (1,0) to [in=40, out=140]  (0,0);  

\draw[with arrows] (2,0) to [in=230, out=310]  (3,0);  
\draw[dashed] (3,0) to [in=40, out=140]  (2,0);  

\draw[with arrows] (4,0) to [in=230, out=310]  (5,0);  
\draw[dashed] (5,0) to [in=40, out=140]  (4,0);  

\draw[with arrows] (6,0) to [in=230, out=310]  (7,0);  
\draw[dashed] (7,0) to [in=40, out=140]  (6,0);

\draw[with arrows] (-1,-4) to [in=230, out=310]  (0,-4);  
\draw[dashed] (0,-4) to [in=40, out=140]  (-1,-4);  

\draw[with arrows] (3,-4) to [in=230, out=310]  (4,-4);  
\draw[dashed] (4,-4) to [in=40, out=140]  (3,-4);  

\draw[with arrows] (7,-4) to [in=230, out=310]  (8,-4);  
\draw[dashed] (8,-4) to [in=40, out=140]  (7,-4);

\draw[]  (0,0) to [in=60, out=270] (-1,-4);

\draw[black!40!green]  (1,0) to [in=210, out=330] (2,0);
\draw[black!40!green]  (3,0) to [in=210, out=330] (4,0);
\draw[black!40!green]  (5,0) to [in=210, out=330] (6,0);

\draw[]  (7,0) to [in=120, out=270] (8,-4);

\draw[blue]  (0,-4) to [in=150, out=30] (3,-4);
\draw[blue]  (4,-4) to [in=150, out=30] (7,-4);

\draw[red, with arrows]  (0.75, -3.67) to [in=280, out=60] (1.3,-0.11);
\draw[red, dashed]  (0.75, -3.67) to [in=245, out=90] (1.3,-0.11);

\draw[red, with arrows]  (1.5, -3.57) to [in=270, out=50] (3.3,-0.11);
\draw[red, dashed]  (1.5, -3.57) to [in=230, out=80] (3.3,-0.11);

\draw[red, with arrows]  (2.25, -3.67) to [in=260, out=40] (5.3,-0.11);
\draw[red, dashed]  (2.25, -3.67) to [in=215, out=70] (5.3,-0.11);

\draw[red, with arrows]  (4.75, -3.67) to [in=280, out=140] (1.6,-0.13);
\draw[red, dashed]  (4.75, -3.67) to [in=325, out=110] (1.6,-0.13);

\draw[red, with arrows]  (5.5, -3.57) to [in=270, out=130] (3.6,-0.13);
\draw[red, dashed]  (5.5, -3.57) to [in=310, out=100] (3.6,-0.13);

\draw[red, with arrows]  (6.25, -3.67) to [in=260, out=120] (5.6,-0.13);
\draw[red, dashed]  (6.25, -3.67) to [in=295, out=90] (5.6,-0.13);

\node[red] at (0.75, -3.9) {\tiny $C_{13}$};
\node[red] at (1.5, -3.9) {\tiny $C_{12}$};
\node[red] at (2.25, -3.9) {\tiny $C_{11}$};

\node[red] at (4.75, -3.9) {\tiny $C_{23}$};
\node[red] at (5.5, -3.9) {\tiny $C_{22}$};
\node[red] at (6.25, -3.9) {\tiny $C_{21}$};

\node[blue] at (0.3, -3.6) {$a_1$};
\node[blue] at (6.8, -3.6) {$a_2$};

\node[black!40!green] at (1.5, 0.1) {$b_3$};
\node[black!40!green] at (3.5, 0.1) {$b_2$};
\node[black!40!green] at (5.5, 0.1) {$b_1$};

\node[] at (0.5, 0.4) {$v_4$};
\node[] at (2.5, 0.4) {$v_3$};
\node[] at (4.5, 0.4) {$v_2$};
\node[] at (6.5, 0.4) {$v_1$};

\node[] at (-0.5, -4.5) {$x_1$};
\node[] at (3.5, -4.5) {$x_2$};
\node[] at (7.5, -4.5) {$x_3$};

\end{tikzpicture}
\caption{The distinguished curves on the surface $\Sigma$; $n=2$, $k=4$.}
\label{distinguishedcurves} 
\end{figure}

Here we will describe the toric variety $\mathcal{Y}_{n,k}$ in terms of the dual cones to cones forming the toric fan. 
These dual cones live in $H \otimes \mathbb{R}$, where $H = H_1(\Sigma)$ is the homology lattice of the punctured surface $\Sigma$
and are governed by the combinatorics of curves on $\Sigma$. 
We denote by $T$ the algebraic torus with the character lattice $H$
($T$ will be the torus acting on $\YY_{n,k}$).

Let us equip each circle $x_i$ with the orientation induced by the orientation of the surface, and we equip each circle $v_j$ with the opposite orientation to the one induced by the orientation of the surface. We still denote by $x_i$ and $v_j$ the corresponding classes in $H$. 
As in Figure \ref{distinguishedcurves}, let us denote by $a_i$ the arc from $x_i$ to $x_{i+1}$ and $b_j$ the arc from $v_j$ to $v_{j+1}$, and consider the circles $c_{ij}$ which intersect $a_i$ and $b_j$ once. We equip $c_{ij}$ with the orientation induced by the subsurface of $\Sigma$ to the left of $c_{ij}$. It is also convenient to set 
$$c_{1k} = -x_1 \ \text{ and } c_{n,0}= x_{n+1}.$$  
Note that we have the following relation in $H$: 
\begin{equation}\label{cij-formula-eq}
c_{ij}=v_k+v_{k-1}+\ldots+v_{j+1}-x_1-x_2-\ldots-x_i.
\end{equation}
We denote by $X_i,V_j,C_{ij}$ the invertible functions on $T$ corresponding to the elements $x_i,v_j,c_{ij}$ of the lattice $H$
(some of these functions and their inverses will extend to regular functions forming coordinates on affine charts of $\YY_{n,k}$).

We will often use the following relations in our lattice and the corresponding relations between functions:
\begin{align}
        c_{i,j} &= c_{i+1,j} + x_{i+1}, \ \ C_{i,j}=C_{i+1,j}X_{i+1},\\
        c_{i,j} &= c_{i,j+1} + v_{j+1}, \ \ C_{i,j}=C_{i,j+1}V_{j+1}. 
\end{align}

Now, we can define our cones in $H$. We start with the set $S \subset H$ given by the classes $v_1,\ldots, v_k$. The cones are numbered by nonincreasing functions
\[ \phi : [1,n] \to [1,k]. \] 
For each such function $\phi$, we will define a set of vectors $S_\phi$ obtained by removing some elements from $S$ and adding some new elements. 
The vectors of $S_\phi$ will correspond to coordinates on an affine chart $\UU_\phi\sub\YY_{n,k}$.
Here is the recipe. 

Let $[1,n]=I_1\sqcup\ldots\sqcup I_r$ be the partition of $[1,n]$ into the level sets of the function $\phi$. For each interval $I=[i,p]\subset [1,n]$ and an index $j$, $1 \leq j \leq k$, we consider the subsurface $\Sigma(I,j)$ bounded by the circles
\[ -c_{i,j}, -v_j, c_{p,j-1}, x_{i+1},\ldots, x_p \] 
Let us denote by $\partial'\Sigma(I,j)$ the set of classes of all the boundary components of $\Sigma(I,j)$ except for $-v_j$, so that $v_j$ is equal to the sum of the classes in $\partial'\Sigma(I,j)$. 
Now, for our partition of the arcs $I_1,\ldots, I_r$, we consider the subsurfaces $\Sigma_1,\ldots, \Sigma_r$, where $\Sigma_m = \Sigma(I_m, \phi(I_m))$. To get the set $S_\phi$ from $S$, we replace each $v_{\phi(I_m)}$ for $m=1,\ldots, r$ with the classes in $\partial' \Sigma_m$. In other words,
\[ S_\phi = \{ v_j | j \notin \mathrm{im}(\phi)\} \sqcup \{x_{i+1}| \phi(i)=\phi(i+1)\} \sqcup \] \[ \{ - c_{i, \phi(i)}| \phi(i-1)>\phi(i) \text{ or } i =1\} \sqcup \{ c_{i,\phi(i)-1}| \phi(i+1) < \phi(i) \text { or } i=n \}\]
Let $\sigma_\phi^\vee \subset H \otimes \mathbb{R}$ denote the cones generated by $S_\phi$, and let $\sigma_\phi \subset H^\vee \otimes \mathbb{R}$ be the dual cones. One can check that the cones $\sigma_\phi$ form a smooth toric fan, and we define $\mathcal{Y}_{n,k}$ to be the corresponding smooth toric scheme 
over the base ring $\k$.
By definition, $\YY_{n,k}$ is covered by the open affine charts $\UU_\phi$, each isomorphic to the $(n+k)$-dimensional affine space with the coordinates given by the multiplicative
variables corresponding to elements of $S_\phi$. Instead of checking that $(\sigma_\phi)$ forms a fan we will construct a projective morphism
$\YY_{n,k}\to \mathbb{A}^{n+k+1}$, which proves that $\YY_{n,k}$ is separable (see Corollary \ref{toric-proj-emb-cor}).

Finally, we define \[ \mathcal{X}_{n,k} \coloneqq \mathcal{Y}_{n,k} \times \mathbb{A}^k_{U_1,\ldots, U_k} \] 
where $\mathbb{A}^k_{U_1,\ldots,U_k}$ is the $k$-dimensional affine space with coordinate functions $(U_1,\ldots, U_k)$.

The vectors $v_1,\ldots, v_k$ belong to each of the cones $\sigma_\phi^{\vee}$, so they give $n$ functions $V_1,\ldots, V_n$ on $\mathcal{Y}_{n,k}$. Now, we define the superpotential $\mathbf{w}_{n,k}$ on the smooth toric variety $\mathcal{X}_{n,k}$ by
\[ \mathbf{w}_{n,k} \coloneqq U_1 V_1  + U_2 V_2 + \ldots + U_k V_k \]

Furthermore, for each $\phi$ the vectors $v_1,\ldots v_k$ are obtained by taking sums of disjoint subsets of the bases $S_\phi$. Hence, the functions $V_1,\ldots, V_k$ form a regular sequence on $\mathcal{Y}_{n,k}$. 

We define the singular variety 
\[ \mathcal{Z}_{n,k} \coloneqq  \{V_1=V_2=\ldots, V_k=0 \} \subset \mathcal{Y}_{n,k} \] 

Let $T$ denote the torus (of dimension $n+k$) acting on $\ov{\YY}_{n,k}$, 
so that the character lattice of $T$ is the homology lattice $H$. 
With respect to this action the weight of the functions $X_i$, $C_{ij}$ and $V_i$ are the corresponding
homology classes $x_i$, $c_{ij}$ and $v_i$. 
The torus $T\times \G_m^k$ acts naturally on $\XX_{n,k}=\YY_{n,k}\times\A^k_{U_1,\ldots,U_k}$, where the second factor 
only rescales $U_i$'s. We restrict this action to the subtorus $\wt{T}\sub T\times \G_m^k$ given by
$$\wt{T}=\{(t,\la_1,\ldots,\la_k) \ |\ v_1(t)\la_1=\ldots=v_k(t)\la_k\}.$$
We can identify the character lattice of $\wt{T}$ with 
$$\wt{H}:=(H\oplus (\bigoplus_i \Z\cdot u_i))/(u_i+v_i-u_j-v_j \ |\ i<j),$$
so that $u_i$ is the weight of the function $U_i$ on $\XX_{n,k}$.

By definition, we have the character 
$$l_0:=u_1+v_1=\ldots=u_k+v_k\in \wt{H},$$
so that the function $\bw=\sum_i U_iV_i$ has weight $l_0$.
Let us consider the sublattice
$$\bbL_B:=2\wt{H}+\Z l_0\sub \wt{H}.$$

The main category of interest for us is the $\bbL_B$-graded category of $\wt{T}$-equivariant matrix factorizations of
$\bw_{n,k}$,
$$\BB_{n,k}:=\MF_{\wt{T}}(\XX_{n,k},\bw_{n,k})_{\bbL_B}$$ 
(see Sec.\ \ref{L-gr-G-mf} for the precise definition).
The grading group $\bbL_B$ can be naturally identified with the grading group $\bbL$ appearing on the A-side.

\begin{example} Consider the case $n=2$, $k=4$. The homology classes of our distinguished curves in this case are, $x_1,x_2,x_3,v_1,v_2,v_3,v_4,c_{11},c_{12},c_{13}, c_{21},c_{22},c_{23}$. We have the relations:
\begin{align*} v_4 &= x_1 + c_{13},\\
                v_3 &= c_{12} - c_{13} = c_{22} - c_{23},  \\
                v_2 &= c_{11} - c_{12} = c_{21} - c_{22},\\
                v_1 &= x_3 - c_{21}, \\ 
         c_{11} &= c_{21} + x_2, \, c_{12} = c_{22} + x_2, \, c_{13}=c_{23}+x_2, \end{align*}
        For each such curve, we write the corresponding toric monomial with a capital letter. The cones are numbered by $\phi:[1,2] \to [1,4]$. Let us write them as $\sigma_{\phi(1)\phi(2)}\coloneqq \sigma_\phi$. We have 10 cones and the corresponding monomials are given as follows:
\begin{align*}
   \sigma_{44}^\vee &= \langle X_1,X_2,C_{23},V_3,V_2,V_1 \rangle, & \sigma_{43}^\vee &= \langle X_1,C_{13},C^{-1}_{23},C_{22},V_2,V_1 \rangle, \\ 
   \sigma_{42}^\vee &= \langle X_1,C_{13},V_3,C_{22}^{-1},C_{21},V_1 \rangle, & \sigma_{41}^\vee &= \langle X_1,C_{13},V_3,V_2,C_{21}^{-1},X_3\rangle \\
        \sigma_{33}^\vee &= \langle V_4,C_{13}^{-1},X_2,C_{22},V_2,V_1 \rangle, & 
        \sigma_{32}^\vee &= \langle V_4,C_{13}^{-1},C_{12},C_{22}^{-1},C_{21},V_1 \rangle \\
        \sigma_{31}^\vee &= \langle V_4,C_{13}^{-1},C_{12},V_2,C_{21}^{-1},X_3\rangle, & \sigma_{22}^\vee &= \langle V_4,V_3,C_{12}^{-1},X_2,C_{21},V_1 \rangle \\
        \sigma_{21}^\vee &= \langle V_4,V_3,C_{12}^{-1},C_{11}, C_{21}^{-1},X_3 \rangle, & \sigma_{11}^\vee &= \langle V_4,V_3,V_2,C_{11}^{-1},X_2,X_3 \rangle
\end{align*}
        Setting $V_1=V_2=V_3=V_4=0$ gives $\k[\sigma_{ii}^\vee]/(V_\bullet) \simeq \mathbb{C}[X,Y,Z]$ for all $i$ and $\k[\sigma_{ij}^\vee]/(V_\bullet) \simeq \mathbb{C}[X,Y,Z,W]/(XY,ZW)$ for all $i>j$, which form the affine charts for $\mathcal{Z}_{2,4}$, corresponding to the vertices of Figure \ref{toricdim2}. Note that these charts are mirror to either the 2-dimensional pair of pants (i.e., the symmetric square of the 4-puncture sphere), or the product of two 1-dimensional pairs of pants. One can see the corresponding 4-punctured spheres or embeddings of two disjoint 3-punctured spheres in the surface $\Sigma$ bounded by the corresponding curves in Figure \ref{distinguishedcurves}. 

        Via homological mirror symmetry, skyscraper sheaves on the components of $\mathcal{Z}_{2,4}$ match with $\mathbb{Z}$-gradable Lagrangian 2-tori in $V_{2,4}$ equipped with unitary local systems. Such Lagrangian tori can be given as products of two disjoint circles on $\Sigma$. We describe this correspondence without proof in Figure \ref{supports}. Lagrangian tori appearing in the same chamber are Hamiltonian isotopic (\cite{perutz}). Any other products among our distinguished curves not appearing in Figure \ref{supports} are non-gradable.
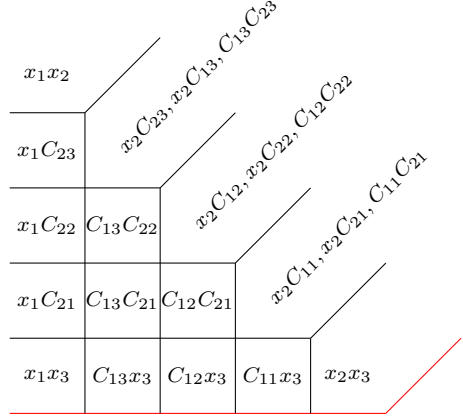
\begin{figure}[!h]
\begin{tikzpicture}

\draw (0,0) -- (4,0);
\draw (0,1) -- (3,1);
\draw (0,2) -- (2,2);
\draw (0,3) -- (1,3);
\draw[red] (0,-1) -- (5,-1);

\draw (1,3) -- (1,-1);
\draw (2,2) -- (2,-1);
\draw (3,1) -- (3,-1);
\draw (4,0) -- (4,-1);

\draw (1,3) -- (2,4); 
\draw (2,2) -- (3,3); 
\draw (3,1) -- (4,2); 
\draw (4,0) -- (5,1);

\draw[red] (5,-1) -- (6,0);

        \node at (1.5, -0.5) {\tiny $C_{13}x_3$};
        \node at (2.5, -0.5) {\tiny $C_{12}x_3$};
        \node at (3.5, -0.5) {\tiny $C_{11}x_3$};
        \node at (1.5, 0.5) {\tiny $C_{13}C_{21}$};
        \node at (2.5, 0.5) {\tiny $C_{12}C_{21}$};
        \node at (1.5, 1.5) {\tiny $C_{13}C_{22}$};

        \node at (0.5, 0.5) {\tiny $x_1 C_{21}$};
        \node at (0.5, 1.5) {\tiny $x_1 C_{22}$};
        \node at (0.5, 2.5) {\tiny $x_1 C_{23}$};
        
        \node[rotate=45] at (4.5, 1.5) {\tiny $x_2C_{11}, x_2 C_{21}, C_{11}C_{21}$};
        \node[rotate=45] at (2.5, 3.5) {\tiny $x_2 C_{23}, x_2 C_{13}, C_{13}C_{23}$};
        \node[rotate=45] at (3.5, 2.5) {\tiny $x_2 C_{12}, x_2 C_{22}, C_{12}C_{22}$};

        \node at (0.5,3.5) {\tiny $x_1 x_2$ };
        \node at (0.5,-0.5) {\tiny $x_1 x_3$ };
        \node at (4.5,-0.5) {\tiny $x_2 x_3$ };

\end{tikzpicture}

	\caption{Lagrangian tori corresponding to skyscraper sheaves on the irreducible components of $\mathcal{Z}_{2,4}$. The red line indicates the compactification divisor $\overline{\mathcal{Z}}_{2,4}\setminus \mathcal{Z}_{2,4}$  }
\label{supports}

\end{figure}

\end{example}

\subsection{Compactified LG model}

It is easy to see that there is a natural locally closed toric embedding
$$\YY_{n,k}\to \YY_{n,k+1}$$
whose image is the subscheme $V_1=0$ in the union of open charts of $\YY_{n,k+1}$ corresponding to $\phi:[1,n]\to [1,k+1]$
with $\im(\phi)\sub [2,k+1]$.
Under this embedding the functions $V_{i+1}$ on $\YY_{n,k+1}$ restrict to $V_i$ on $\YY_{n,k}$, and
function $X_{n+1}$ on $\YY_{n,k}$ gets identified with the restriction of the function $C_{n,1}$ on these open charts of $\YY_{n,k+1}$.

We define $\ov{\YY}_{n,k}$ (resp., $\ov{\ZZ}_{n,k}$) as the closure of $\YY_{n,k}$ (resp., $\ZZ_{n,k}$) under this embedding.
The functions $V_1,\ldots,V_k$ extend regularly to $\ov{\YY}_{n,k}$,
so the potential $\bw_{n,k}$ extends regularly to $\ov{\XX}_{n,k}=\ov{\YY}_{n,k}\times \A^k$. We denote by $\ov{\BB}_{n,k}$ the category of $\bbL_B$-graded equivariant
matrix factorizations of $\bw_{n,k}$ on $\ov{\XX}_{n,k}$ (recall that $\bbL_B$ is naturally isomorphic to $\bbL$). 

We have the following extensions of Conjecture \ref{mainconj} and Theorem \ref{main-thm}.
We equip $V_{n,k}$ with one stop $\Lambda$ corresponding to a point on the component marked by $x_k$ 
and consider the corresponding partially wrapped Fukaya category. 

\begin{conj} \label{mainconj-bis} There exists a quasi-equivalence of $\mathbb{Z}$-graded pre-triangulated categories 
        \begin{align} \label{limitcat-bis} \mathcal{W}(V_{n,k},\Lambda,\eta_0) \simeq D^b \mathrm{Coh} (\ov{\ZZ}_{n,k}),
       \end{align}
       and a quasi-equivalence of $\bbL$-graded pre-triangulated categories
       \begin{align} \label{LG-limitcat-bis} \mathcal{W}(V_{n,k},\Lambda)\simeq \ov{\BB}_{n,k},\end{align}
               where $\ov{\BB}_{n,k}$ is the category of $\bbL$-graded matrix factorizations of $\mathbf{w}_{n,k}$ on $\ov{\XX}_{n,k}$.
\end{conj}

\begin{thm}\label{main-thm-bis} Equivalences  \eqref{limitcat-bis} and \eqref{LG-limitcat-bis} hold for $n=1$ and $n=2$.
\end{thm}

Note that we derive Theorem \ref{main-thm} from Theorem \ref{main-thm-bis} by passing to localizations.
The Landau-Ginzburg version is also crucial for our approach, as it allows to prove formality of the $A_\infty$-endomorphism algebra
of the generators in $D^b\Coh(\ov{\ZZ}_{n,k})$ (for $n\le 2$). 
Namely, we observe that for a certain $\Z$-grading the endomorphism algebra of the corresponding matrix factorizations is concentrated in degree zero
and derive from this formality of the $\bbL$-graded endomorphism algebra.

\medskip
             
\noindent
{\it Conventions}. For our main results we work over a ground field $\k$. In some results $\k$ can be any regular Noetherian ring.
Both for Fukaya categories and for matrix factorization categories, we pass to a perfect derived category of the original dg-category.

\medskip
             
\noindent
{\it Acknowledgments}. Y.~L. would like to thank Denis Auroux and Nick Sheridan for helpful correspondences. Y.~L. was partially funded by the Royal Society
URF\textbackslash R\textbackslash 180024. A.~P. is partially supported by the NSF grant DMS-2001224, 
and within the framework of the HSE University Basic Research Program and by the Russian Academic Excellence Project `5-100'.

\section{Preliminaries}

\subsection{(Multi)gradings of $A_\infty$-algebras}\label{multigrading-sec}

Let $\bbL$ be an abelian group equipped with an element $l_0\in \bbL$ and a homomorphism
$\bbL\to \Z/2: x\mapsto |x|$, such that $|l_0|=1\mod 2$.
We will use this homomorphism to collapse an $\bbL$-grading to a $\Z/2$-grading.
The usual notion of an $A_\infty$-algebra which is a $\Z$-graded vector space with some operations
has a natural $\bbL$-graded version.

\begin{definition}
An {\it $\bbL$-graded $A_\infty$-algebra} is an $\bbL$-graded vector space $A=\bigoplus_{l\in \bbL} A_l$ equipped with operations
        $$\mathfrak{m}_n:A^{\ot n}\to A, \ n\ge 1,$$
        where the $\bbL$-degree of $\mathfrak{m}_n$ is $(2-n)l_0$. These operations satisfy the usual $A_\infty$-axioms 
       \[ \sum_{m,n} (-1)^{|a_1|+\ldots + |a_n|-n} \mathfrak{m}_{d-m+1}(a_d,\ldots, a_{n+m+1}, \mathfrak{m}_m (a_{n+m},\ldots, a_{n+1}), a_n,\ldots a_1) = 0. \] 
        where we use
the induced $\Z/2$-grading on $A$. Similarly, we can define the notion of an $A_\infty$-homomorphism between $\bbL$-graded $A_\infty$-algebras.
\end{definition}

We will call a homomorphism $f:\bbL\to \Z$ {\it admissible} if $f(l_0)=1$ and $f(x)\equiv |x|\mod(2)$.
Given an $\bbL$-graded $A_\infty$-algebra $A$ and an admissible homomorphism $f:\bbL\to \Z$,
we can collapse the $\bbL$-grading on $A$ to a $\Z$-grading and get a usual $A_\infty$-algebra $A/f$.
Note that an admissible homomorphism $f$ induces a decomposition
$$\bbL=\Z\cdot l_0\oplus \bbL_0,$$
where $\bbL_0=\ker(f)$. Then the $\bbL$-grading on $A$ can be viewed as an $\bbL_0$-grading on $A/f$ compatible with the $\Z$-grading,
such that all $\mathfrak{m}_n$ on $A/f$ preserve the $\bbL_0$-degree. The same is true about $A_\infty$-homomorphisms (in particular, gauge equivalences)
between $\bbL$-graded $A_\infty$-algebras $A$ and $B$: they can be viewed as usual $A_\infty$-homomorphisms between the usual $A_\infty$-algebras
$A/f$ and $B/f$, preserving the $\bbL_0$-degree.

Note that when $\bbL_0$ is finitely generated, an $\bbL_0$-graded vector space can be thought as an algebraic representation of the corresponding commutative algebraic group $G$, such that $\bbL_0$ is the group of characters of $G$.
Thus, in the presence of a splitting $\bbL=\Z\oplus \bbL_0$, the above notion of an $\bbL$-graded $A_\infty$-algebra is equivalent to that of a $G$-equivariant
$\Z$-graded $A_\infty$-algebra. Since for a reductive group $G$ we can choose all the projections involved in
the homological perturbation theory to be $G$-equivariant, we have the following statement.

\begin{lem} Let $G$ be a linearly reductive algebraic group over a field $\k$.
For a dg-algebra $A$ over $\k$ equipped with an algebraic $G$-action (respecting the dg-algebra structure) there is a $G$-equivariant version
of the homological perturbation producing a $G$-equivariant minimal $A_\infty$-structure on $H^*(A)$.
\end{lem}

We have the following obvious formality criterion for $\bbL$-graded $A_\infty$-algebras.

\begin{lem}\label{formality-lem} 
  Let $(A, \mathfrak{m}_\bullet)$ be a minimal $\bbL$-graded $A_\infty$-algebra such that for some admissible homomorphism $f:\bbL\to \Z$, the algebra $A/f$ is
        concentrated in degree $0$. Then $\mathfrak{m}_n=0$ for $n>2$.
\end{lem}

\subsection{Multigraded Fukaya categories} 

Multigradings appear naturally in the context of Fukaya categories: they have been used in \cite{seidelgenus2}, \cite{FOOO18} and studied extensively by Sheridan \cite{sheridan}, \cite{sheridan2}. Given a symplectic manifold $V$, we consider
the relative Lagrangian Grassmannian of the tangent bundle $TV$,
$$\mathcal{G}V \to V.$$
The exact sequence associated to this fibration gives the following exact sequence
\[ \mathbb{Z} = H_1(\mathcal{G}_p V) \to H_1(\mathcal{G}V) \to H_1(V) \to 0\] 
where $\GG_pV$ is the Lagrangian Grassmannian of the tangent space at a point $p\in V$, and
the identification $\mathbb{Z} = H_1(\mathcal{G}_pV)$ is given by the Maslov index. We set 
$$\bbL:= H_1(\mathcal{G}V)$$ 
and let $l_0$ be the image of $1 \in \mathbb{Z}$ under the homomorphism $H_1(\mathcal{G}_p V ) \to H_1(\mathcal{G}V)$. 

In addition, we choose a homomorphism 
\[ \bbL = H_1(\mathcal{G}V) \to \mathbb{Z}/2 \]
which sends $l_0$ to 1. 

In \cite{sheridan}\ \cite{sheridan2}, a canonical choice of such a homomorphism is given by the pairing with the first Stiefel-Whitney class $w_1(\mathcal{L}) \in H^1(\mathcal{G}V;\mathbb{Z}/2)$ of the universal bundle $\mathcal{L} \to \mathcal{G}V$ whose fiber over a point is the Lagrangian subspace at that point. This corresponds to the canonical $\mathbb{Z}/2$-grading structure on Fukaya categories corresponding to the fibrewise double cover of $\mathcal{G}V$ given by the oriented Lagrangian Grassmannian of the tangent bundle $TV$. 

To allow for other choices of $\mathbb{Z}/2$-gradings, we consider pairings with 
\[  w_1(\mathcal{L}) + p^*\sigma \in H^1(\mathcal{G}V; \mathbb{Z}/2) \]
for $\sigma \in H^1(V;\mathbb{Z}/2)$, where $p: \mathcal{G}V \to V$ is the natural projection. By a slight abuse of notation, we denote the corresponding homomorphism by
\[ \sigma: \mathbb{L} \to \mathbb{Z}/2 \] 
Note that as we obtained this by pulling back a class from $V$, this homomorphism still satisfies $\sigma(l_0)=1$. 

Sheridan explains in \cite[Section 3]{sheridan} that the gluing formulae for index theory of Cauchy-Riemann operators allow one to define an $\bbL$-graded Fukaya category and in \cite[Appendix B.3]{sheridan2} provides careful elaborations of the arguments by Seidel in \cite[Sections 11 \& 12 ]{seidelbook}. The key topological observation that makes this work is that if there is a holomorphic disk contributing to some $A_\infty$-product, then its boundary is nullhomologous. The proof given in \cite[Appendix B.3]{sheridan2} applies in the same way if we use the more general $\mathbb{Z}/2$-gradings corresponding to $\sigma: \mathbb{L} \to \mathbb{Z}/2$ after modifying the definition in \cite[Definition B.4]{sheridan2} so that $\mathbb{Z}/2$ degree of a generator is computed with respect to the $\mathbb{Z}/2$ grading associated to $\sigma$. 

The objects of this Fukaya category are $\bbL$-graded Lagrangian submanifolds of $V$. This means that, for a Lagrangian submanifold $L$, we require a lift of its Gauss map $L \to \mathcal{G}V$ to the universal abelian cover $\tilde{\mathcal{G}}V \to \mathcal{G}V$ and a Pin structure on $L$. If $H_1(L) \to H_1(V)$ is trivial, then such a lift exists (and is unique up to translation by $\bbL$ for connected $L$). A Pin structure on $L$ exists if and only if the second Stiefel-Whitney class vanishes, and the space of Pin structures on $L$ is a torsor over $H^1(L; \mathbb{Z}_2)$. In this paper, we work only with contractible Lagrangians, so they will always be gradable (uniquely up to translation by $\bbL$) and will have canonical Pin structures on them. 

In summary, for each $\mathbb{Z}/2$-grading structure $\sigma$ on $V$, which can be prescribed by exhibiting a homomorphism $\sigma: \mathbb{L} \to \mathbb{Z}_2$ with $\sigma(l_0)=1$, we can construct a $\mathbb{L}$-graded Fukaya category of $V$, where $\mathbb{L}= H_1(\mathcal{G}V)$. There is an effective $H^1(V;\mathbb{Z}_2)$ worth of choices for the $\mathbb{Z}/2$-grading structures, but once we fix such a $\mathbb{Z}/2$-grading, there is a canonical $\mathbb{L}$-graded Fukaya category.  

\begin{rmk} i) Vanishing of $2c_1(V)$ guarantees that the morphism $\mathbb{Z} = H_1(\mathcal{G}_p V) \to H_1(\mathcal{G}V)$ is injective, or equivalently $l_0 \in \bbL$ is non-torsion. 
 \ \\
        ii) Recall that the usual way of equipping Fukaya categories with a $\mathbb{Z}$-grading goes via similar construction (see \cite{seidelgraded}) however instead of using the universal abelian cover of $\mathcal{G}V$, one considers a fibrewise universal cover of $\mathcal{G}V$. Such a cover exists if and only if $2c_1(V)=0$. Moreover, if non-empty, the space of such covers is an $H^1(V)$-torsor.
        \end{rmk}

\subsection{Categories of graded matrix factorizations}

Let $X$ be a smooth space over $\k$ (possibly an algebraic stack) with a regular function $W$.
We are going to recall the definition of $\Z$-graded categories of matrix factorizations following
\cite{PV-mf-st} and explain how to generalize it to get multigraded versions of these categories.

\subsubsection{$\Z$-graded categories of $G$-equivariant matrix factorizations}\label{Z-gr-G-mf}

Assume that we have an action of a reductive group $G$ (over $\k$) on $X$, and the potential $W$ on $X$ satisfies 
$$W(gx)=l_0(g)W(x).$$
for some character $l_0:G\to \G_m$.
A {\it $(G,l_0)$-equivariant matrix factorization} of $W$ consists of 
$G$-equivariant vector bundles $E_0$ and $E_1$ and $G$-equivariant homomorphisms
$$d_1:E_1\to E_0, \ \ d_0:E_0\to E_1\ot l_0,$$
such that $d_1 d_0=d_0 d_1=W$. 
 
For a pair $(E,d_E)$, $(F,d_F)$ of $(G,l_0)$-equivariant matrix factorizations we define a $(G,l_0)$-equivariant matrix
factorization $\und{\Hom}=\und{\Hom}((E,d_E),(F,d_F))$ of $0$, by setting
\footnote{This construction can be seen as a combination of two operations on matrix factorizations, 
the tensor product and the duality (see \cite[Sec.\ 1]{PV-cohft}).}
$$\und{\Hom}_0:=E_0^\vee\ot F_0\oplus E_1^\vee\ot F_1, \ \
\und{\Hom}_1:=E_0^\vee\ot F_1\oplus E_1^\vee\ot F_0\ot l_0^{-1}.
$$

Now for any $(G,l_0)$-equivariant matrix factorization $(H,d_H)$ of $0$, we can consider a $\Z$-graded complex of
$G$-equivariant vector bundles
$$\com(H)=[\ldots H_1\to H_0\to H_1\ot l_0\to H_0\ot l_0\ldots]$$
with $H_0$ in degree $0$. Note that it is equipped with an isomorphism 
\begin{equation}\label{alpha-E-shift-relation}
\a_H:\com(H)[2]\to \com(H)\ot l_0.
\end{equation}

For a pair $E=(E,d_E)$, $F=(F,d_F)$ as above, the $\Z$-graded morphism space is defined by
$$\hom_{\Z}(E,F):=R\Ga(X,\com(\und{\Hom}(E,F)))^G,$$
where we use some functorial multiplicative resolutions to compute $R\Ga$.
\footnote{Equivalently, one can first consider the naive dg-category, without deriving the functor of global sections, and then
pass to the quotient by the subcategory of locally contractible objects (see \cite{PV-mf-st}).}
We denote by $\MF_{G,l_0}(X,W)_{\Z}$ the perfect derived category of the $\Z$-graded dg-category of $(G,l_0)$-equivariant
matrix factorizations of $W$. Note that if $W$ is not a zero then $l_0$ is uniquely determined by $W$, so we will sometimes
omit $l_0$ from the notation.

Note that due to isomorphism \eqref{alpha-E-shift-relation}, for any $G$-equivariant matrix factorization $E$ we have 
$$E\ot l_0\simeq E[2].$$

Assume that $W$ is not a zero divisor and consider the hypersurface $X_0\sub X$ given by $W=0$. 
We denote by $D_{\sing}([X_0/G])$
the singularity category of the stack $[X_0/G]$, i.e., the quotient of the bounded derived category of $G$-equivariant sheaves on $X_0$ by the subcategory perfect complexes. Assuming that the stack $[X/G]$ is sufficiently nice
(of finite cohomological dimension and with a resolution property) one has a natural equivalence
(see \cite[Thm.\ 3.14]{PV-mf-st})
\begin{equation}\label{MF-Dsing-G-equivalence}
\MF_G(X,W)_{\Z}\simeq D_{\sing}([X_0/G]).
\end{equation}

\subsubsection{$\G_m$-equivariant matrix factorizations of a potential of degree $1$}

An important case is when $G=\G_m$ and $l_0:\G_m\to \G_m$ is the identity character,
so the potential $W$ satisfies
$$W(\la\cdot x)=\la\cdot W(x),$$
where $\la\in \G_m$, $x\in X$. 
In this case the above definition gives a definition of the dg category $\MF_{\G_m}(X,W):=\MF_{\G_m,l_0}(X,W)_{\Z}$ 
of $\G_m$-equivariant matrix factorizations of $W$. 

\begin{rmk}
1. It can be shown that the above category $\MF_{\G_m}(X,W)$ is equivalent to the category of $B$-branes defined as in
Segal's paper \cite{Segal} for the {\it doubled action} of $\G_m$ on $X$.

\noindent
2. If $W$ is a potential of weight $d>0$ with respect to the $\G_m$-action, we can pass to the stack $\XX:=[X/\mu_d]$ and equip it
with the action of $\G_m/\mu_d\simeq \G_m$. Then $W$ descends to a function on $\XX$ of weight $1$ with respect to the $\G_m$-action.
Then the above definition can be applied to the stack $\XX$ with the descended potential.
\end{rmk}

In the case when $W$ is not a zero divisor and the stack $[X/\G_m]$ is sufficiently nice, we get
from \eqref{MF-Dsing-G-equivalence} an equivalence
\begin{equation}\label{MF-Dsing-Gm-equivalence}
\MF_{\G_m}(X,W)\simeq D_{\sing}([X_0/\G_m]).
\end{equation}

Now assume that $V_1,\ldots,V_k$ is a regular sequence of global functions on a smooth scheme $Y$, and let $Z\sub Y$
be the zero locus $Z(V_1,\ldots,V_k)$. Then we can consider the scheme $X=Y\times \A^k_{U_1,\ldots,U_k}$ with the
potential $W=U_1V_1+\ldots+U_kV_k$. We equip $X$ with the natural $\G_m$-action, trivial on $Y$ and such that each $U_i$
has weight $1$, and consider the category of matrix factorizations $\MF_{\G_m}(X,W)$. Then I\c{s}\i k's theorem \cite{Isik}
states that there is an equivalence
\begin{equation}\label{Isik-equivalence}
D^b\Coh(Z)\simeq \MF_{\G_m}(X,W).
\end{equation}

\subsubsection{Multigraded categories of matrix factorizations}\label{L-gr-G-mf}

Now we observe that in the context of Sec.\ \ref{Z-gr-G-mf}, we can also define a multigraded dg-category
of matrix factorizations.

Namely, let $L$ denote the group of characters of $G$. We assume that $l_0\not\in 2L$
and consider the sublattice
$$L(l_0):=\Z\cdot l_0+2L\sub L.$$
Note that we have a natural surjective homomorphism
$$L(l_0)\to \Z/2$$
sending $l_0$ to $1\mod 2$, and sending $2L$ to $0$.

For any $(G,l_0)$-equivariant matrix factorization $(H,d_H)$ of $0$ we can define an $L(l_0)$-graded complex
of $G$-equivariant bundles $\com_{L(l_0)}(H)$ by 
$$\com_{L(l_0)}(H)_{2l}:=H_0\ot l, \ \ \com_{L(l_0)}(H)_{2l+l_0}:=H_1\ot l\ot l_0,$$
with the differential of degree $l_0$ induced by $d_H$.

Now for a pair $E=(E,d_E)$, $F=(F,d_F)$ of $(G,l_0)$-equivariant matrix factorizations
we define the $L(l_0)$-graded space of morphisms
$$\hom_{L(l_0)}(E,F):=R\Ga(X,\com_{L(l_0)}(\und{\Hom}(E,F)))^G,$$
where on the right we take the totalization of the $\Z\times L(l_0)$-graded complex with
respect to the homomorphism $\Z\times L(l_0)\to L(l_0): (n,l)\mapsto nl_0+l$.
We denote by $\MF_{G,l_0}(X,W)_{L(l_0)}$ the corresponding $L(l_0)$-graded dg-category
(sometimes we will omit $l_0$ from the notation).

Note that we have a natural shift operation $(E,d_E)\mapsto (E[1],d_{E[1]})$, where
$$E[1]_0=E_1\ot l_0, \ E[1]_1=E_0\ot l_0,$$
and $d_{E[1]}$ is induced by $-d_E$. This shift operation is compatible with the shift of degree by $l_0$ on the spaces
of morphisms.

\begin{rmk}
It is clear from the definition that 
if a global function $f$ on $X$ has weight $l$ with respect to the $G$-action, then the multiplication by $f$
gives an endomorphism of degree $2l$ of any object of $\MF_{G,l_0}(X,W)_{L(l_0)}$.
Thus, the natural $G$-weights get doubled when passing to the endomorphisms in this category.
\end{rmk}

\begin{lem}\label{G-to-Gm-lem}
(i) There is an equivalence of $\Z$-graded dg-categories
$$\MF_{G,l_0}(X,W)_{\Z l_0}\simeq \MF_{G,l_0}(X,W)_{\Z},$$
where on the left we consider the subcategory corresponding to the 
subcomplexes of $\hom$-spaces that have grading in $\Z l_0\sub L(l_0)$ 
(and identify the grading group $\Z l_0$ with $\Z$).

\noindent
(ii) Assume that $G$ is a split algebraic torus: $G\simeq \G_m^n$.
Then for a subgroup $\G_m\sub G$, such that $l_0$ restricts to the identity character of $\G_m$, the
natural forgetful functor
$$\MF_{G,l_0}(X,W)_{L(l_0)}\to \MF_{\G_m}(X,W)$$
acts as identity on morphism spaces by collapsing the $L(l_0)$-grading to the $\Z$-grading via the homomorphism
$L(l_0)\hra L\to \Z$ dual to the embedding $\G_m\to G$.
\end{lem}

\begin{proof}
Part (i) follows immediately from the definitions. For part (ii), we first observe that the homomorphism $p_\Z:L\to \Z$ induces a splitting
$$L=\Z l_0\oplus K,$$
where $K=\ker(p_\Z)$. From this we get an induced splitting
$$L(l_0)=\Z l_0 \oplus 2K.$$
Next, let $(H,d_H)=\und{\Hom}(E,F)$, for some
$(G,l_0)$-equivariant matrix factorizations $E$ and $F$.
Then we observe that
$$\com_{L(l_0)}(H)=\bigoplus_{k\in K} \com(H)\ot k.$$
Due to our assumption on $G$, this implies that
$$R\Ga(X,\com_{L(l_0)}(H))^G\simeq R\Ga(X,\com(H))^{\G_m}$$
as claimed.
\end{proof}

From Lemma \ref{G-to-Gm-lem}(i) we see that one can recover all morphisms 
in $\MF_{G,l_0}(X,W)_{L(l_0)}$ from those in $\MF_{G,l_0}(X,W)_{\Z}$ using the isomorphism
of morphism spaces in $\MF_{G,l_0}(X,W)_{L(l_0)}$,
$$\Hom^{2l+ml_0}(E,F)\simeq \Hom^{ml_0}(E,F\ot l).$$

Assume now that $W$ is not a zero divisor.
Then using the equivalence \eqref{MF-Dsing-G-equivalence}
we get an equivalence 
$$\MF_{G,l_0}(X,W)_{\Z l_0}\simeq \MF_{G,l_0}(X,W)_{\Z}\simeq D_{\sing}([X_0/G]).$$
Thus, given a $G$-equivariant coherent sheaf $F$ on $X_0$, we get from the corresponding object of the $G$-equivariant singularity category an object of $\MF_{G,l_0}(X,W)_{L(l_0)}$.

\subsubsection{Grothendieck-Serre duality and a vanishing criterion for morphisms}\label{Serre-dual-sec}

Assume now that $G=\G_m$ and $l_0$ is the identity character.
Recall that for a pair of $\G_m$-equivariant matrix factorizations of $W$, $E$ and $F$, the cohomology of the morphism
space are given by
$$\Hom^*(E,F)\simeq H^*(X,\com(\und{\Hom}(E,F)))^{\G_m}.$$
Now assume that we have a $\G_m$-equivariant proper
morphism $\pi:X\to S$, where $S=\Spec(A)$ is a regular Noetherian affine scheme with a $\G_m$-action.
Then we can define a $\G_m$-equivariant sheaf on $S$ by
$$R\Hom(E,F)_S:=R\pi_*\com(\und{\Hom}(E,F)).$$

In addition, we will need the following construction from
\cite[Sec.\ 1.1]{PV-mf-st}.
For a bounded complex $(C^\bullet,\de)$ of $\G_m$-equivariant vector bundles on $X$, we define
a $\G_m$-equivariant matrix factorization $\mf(C)$ of $0$, by setting
$$\mf(C)_0=\bigoplus_n C^{2n}\ot l_0^{-n}, \ \ \mf(C)_1=\bigoplus_n C^{2n-1}\ot l_0^{-n},$$
with the differential induced by $\de$.
It is easy to see that there is a natural isomorphism of $\G_m$-equivariant complexes
$$\com(\mf(C))=\bigoplus_n C\ot l_0^n[-2n].$$

\begin{lem}\label{SD-mf-lem} 
(i) In the above situation,
for $\G_m$-equivariant matrix factorizations $E$ and $F$,
one has an isomorphism
$$R\Hom(E,F)^\vee_S\simeq R\Hom(F,E\ot \om_{X/S}[n])_S,$$
where $n$ is the relative dimension of $\pi$, and $M^{\vee}=R\Hom(M,\OO_S)$, for $M\in D_{\G_m}(S)$.

\noindent
(ii) In the same situation assume that $\und{\Hom}(F,E)\simeq \mf(i_*H)$ for a closed $\G_m$-invariant
subscheme $i:Z\hra X$ and $H\in D_{\G_m}(Z)$, such that
$i^*\om_{X/S}\simeq \OO_Z\ot l_0^m$ for some $m\in \Z$. 
Then the vanishing of $\Hom^*(E,F)$ implies the vanishing of $\Hom^*(F,E)$.
\end{lem}

\begin{proof} (i) Since $\com(\und{\Hom}(F,E\ot \om_{X/S}))\simeq \com(\und{\Hom}(E,F))^\vee\ot \om_{X/S}$, 
this follows immediately from the usual ($\G_m$-equivariant) Grothendieck-Serre duality applied to $\com(\und{\Hom}(E,F))$.

\noindent
(ii) 
Set $\HH:=R\Hom(E,F)_S$. By assumption, we have $R\Ga(S,\HH)^{\G_m}=0$. But we also have
$$\HH\ot l_0^n\simeq \HH[2n],$$
hence, we get $R\Ga(S,\HH\ot l_0^n)^{\G_m}=0$ for every $n\in \Z$. Since $S$ is affine, this implies that $\HH=0$.
Therefore, by part (i), we deduce that 
$$R\pi_*\com(\und{\Hom}(F,E\ot \om_{X/S}[n]))\simeq R\Hom(F,E\ot \om_{X/S}[n])_S=0.$$
By assumption, we have an isomorphism
$$\com(\und{\Hom}(F,E))\ot \om_{X/S}\simeq \bigoplus_n i_*H\ot l_0^n[-2n]\ot \om_{X/S}\simeq 
\bigoplus_n i_*H\ot l_0^{n+m}[-2n]\simeq \com(\und{\Hom}(F,E))\ot l_0^m.$$
Therefore, we deduce the vanishing of $R\pi_*\com(\und{\Hom}(F,E))$, and hence, of $\Hom^*(F,E)$.
\end{proof}

\subsubsection{Generation}

We will use the standard generation result for the singularity categories (see \cite[Thm.\ 3.5]{LinPom}).
However, we need to adapt this result to the multigraded categories we work with.

Let us assume that we are in the setup of Sec.\ \ref{L-gr-G-mf}. Assume in addition that $G$ is a split algebraic torus,
and let us fix a subgroup $\G_m\sub G$, such that $l_0|_{\G_m}$ is the identity character $\chi$.

\begin{lem}\label{L-graded-generation-lem} 
Let $Z\sub X_0$ be a $G$-invariant closed subscheme containing the singular locus of $X_0$, and let
$(E_i)_{i\in I}$ be a collection of $G$-equivariant coherent sheaves on $Z$, which we view as objects of $D_{\sing}([X_0/G])$ and as objects of $\MF_G(X,W)_{L(l_0)}$.
Assume that $(E_i\ot\chi^n)_{i\in I,n\in \Z}$ generate the $D^b\Coh(Z/\G_m)$. Then $(E_i)$ generate $\MF_G(X,W)_{L(l_0)}$.
\end{lem}

\begin{proof} The category $\MF:=\MF_G(X,W)_{L(l_0)}$ has a natural quasicoherent extension $\MF^{\qcoh}$, such that the objects of $\MF$ are compact.
Thus, to prove that $(E_i)$ generate $\MF$, it is enough to check that for $X\in \MF^{\qcoh}$, 
the vanishing $\Hom(E_i,X)=0$ for all $i$ implies that $X=0$.
It is enough to know the same assertion after collapsing the $L(l_0)$-grading to the $\Z$-grading via some homomorphism 
$f:L(l_0)\to\Z$ sending $l_0$ to $1$.
Hence, by Lemma \ref{G-to-Gm-lem}(ii),
we can replace $G$ with our fixed subgroup $\G_m$. 
Thus, we are reduced to showing that $(E_i)$ generate $\MF_{\G_m}(X,W)$. Note that 
$E_i\ot l_0^n\simeq E_i[2n]$
in this category. Hence, using equivalence with the singularity category \eqref{MF-Dsing-Gm-equivalence}, it is enough to check that
$(E_i\ot\chi^n)$ generate $D_{\sing}([X_0/\G_m])$. By the equivariant version of \cite[Thm.\ 3.5]{LinPom}, it is enough to check that 
the subcategory of $D^b \Coh(X_0/\G_m)$ generated by $(E_i\ot\chi^n)$ contains $i_*D^b\Coh(\Sing X_0)$, where $i:\Sing X_0\hra X_0$ is the singular locus. 
But this follows from our assumption on $(E_i)$.
\end{proof}

\subsubsection{Elementary computations with Koszul matrix factorizations}

Let us denote by $\{a,b\}$ the Koszul matrix factorization of $a\cdot b$, where $a$ and $b$ are possibly sections of
line bundles,
$$E_1=\OO(-b)\rTo{b} E_0=\OO\rTo{a}\OO(a)=E_1(ab)$$
where $a$ (resp., $b$) is a section of a line bundle $\OO(a)$ (resp., $\OO(b)$).
  In the computations below we view coherent sheaves as matrix factorizations of $0$ via the functor $\mf(\cdot)$
described in Sec.\ \ref{Serre-dual-sec}.

\begin{lem}\label{Koszul-mf-lem}
(i) Assume $(a,b)$ is a regular sequence. Then one has natural quasiisomorphisms
$$\und{\End}(\{a,b\})\simeq \{a,b\}^\vee|_{b=0}\simeq \OO/(a,b).$$

\noindent
(ii) Assume $(a,b,c)$ is a regular sequence. Then one has natural quasiisomorphisms
$$\und{\Hom}(\{bc,a\},\{ca,b\})\simeq \{bc,a\}^\vee|_{b=0}\simeq \OO(a)/(a,b)[-1].$$
The generating morphism comes from the exact triangle
$$\OO(-a)/(b)\rTo{a} \OO/(ab)\to \OO/(a)\to \OO(-a)/(b)[1].$$

\noindent
(iii) Assume $(a,b,c)$ is a regular sequence. Then one has natural quasiisomorphisms
$$\und{\Hom}(\{c,ab\},\{ca,b\})\simeq \{c,ab\}^\vee|_{b=0}\simeq \OO/(b,c),$$
$$\und{\Hom}(\{ca,b\},\{c,ab\})\simeq \{ca,b\}^\vee|_{ab=0}\simeq \OO(-a)/(b,c).$$
The generating morphism for the former is the natural projection $\OO/(ab)\to \OO/(b)$,
while the generating morphism for the latter corresponds to the natural embedding
$$\OO/(b)\rTo{a} \OO(a)/(ab).$$
\end{lem}

\begin{proof}
We use the standard quasiisomorphism of matrix factorizations of zero,
$$\und{\Hom}((E,d_E),\{a,b\})\simeq (E,d_E)^\vee\ot \{a,b\}\simeq (E,d_E)^\vee|_{b=0}$$  
(see \cite[Prop.\ 1.6.3(ii)]{PV-cohft}). 

In (i) this leads to
$$\OO(-a)/(b)\rTo{a} \OO/(b)\rTo{0}\ldots$$
which is quasiisomorphic to $\OO/(a,b)$.

In (ii) we get
$$\OO(-bc)/(b)\rTo{0} \OO/(b)\rTo{a}\OO(a)/(b)$$
which is quasiisomorphic to $\OO(a)/(a,b)[-1]$.

In (iii), for the first $\und{\Hom}$ we get
$$\OO(-c)/(b)\rTo{c}\OO/(b)\rTo{0}\ldots$$
which is quasiisomorphic to $\OO/(b,c)$.

Finally, for second $\und{\Hom}$ in (iii) we get
$$\OO(-ca)/(ab)\rTo{ca} \OO/(ab)\rTo{b}\OO(b)/(ab)\to\ldots$$
We observe that 
$$\ker(ca:\OO/(ab)\to \OO(ca)/(ab))=\im(b:\OO(-b)/(ab)\to \OO/(ab)),$$  
$$\ker(b:\OO/(ab)\to \OO(b)/ab)=\im(a:\OO(-a)/(ab)\to \OO/(ab)).$$
Hence, the above matrix factorization of $0$ is quasiisomorphic to
$$\coker(\OO(-ca)/(ca)\rTo{ca} a\OO(-a)/(ab))\simeq \coker(\OO(-ca)/(b)\rTo{c} \OO(-a)/(b))\simeq \OO(-a)/(b,c).$$
\end{proof}

\subsection{Toric geometry of $\YY_{n,k}$}\label{Toric-geom-sec}

Below we use the notation from the description of the dual cones of $\YY_{n,k}$ from Sec.\ \ref{toric-constr-sec}.
Let us denote by $\lan S_\phi\ran_{\Z_{\ge 0}}$ the linear combinations of vectors in $S_\phi$ with nonnegative integer coefficients. 

\begin{lem}\label{toric-cones-lem} 
Fix a nonincreasing map $\phi:[1,n]\to [1,k]$.

\noindent
(i) Each subset $S_\phi\sub H$ is a basis of $H$.

\noindent
(ii) One has $c_{i,j}\in \lan S_\phi\ran_{\Z_{\ge 0}}$ whenever $j<\phi(i)$ and $-c_{i,j}\in \lan S_\phi\ran_{\Z_{\ge 0}}$ whenever $j\ge \phi(i)$. 

\noindent
(iii) For each $i=1,\ldots,k$ and each $j=1,\ldots,n+1$, the vectors $v_i$ and $x_j$ are in $\lan S_\phi\ran_{\Z_{\ge 0}}$.
\end{lem}

\begin{proof} 
(i) This follows immediately from the fact that the natural map of lattices 
$$H_1(\Si_1)\oplus\ldots\oplus H_1(\Si_r)\oplus \bigoplus_{i\not\in\im(\phi)}\Z\cdot v_i\to H$$
is an isomorphism.

\noindent
(ii) The relations $c_{i,j-1}=c_{i,j}+v_j$ and $-c_{i,j+1}=-c_{i,j}+v_{j+1}$ reduce this to
$$c_{i,\phi(i)-1}\in \lan S_\phi\ran_{\Z_{\ge 0}}, \ \ -c_{i,\phi(i)}\in \lan S_\phi\ran_{\Z_{\ge 0}}.$$
Let us prove the first inclusion---the proof of the second is similar. We use descending induction on $i$.
If $\phi(i+1)<\phi(i)$ or $i=n$ then
$c_{i,\phi(i)-1}\in S_\phi$ by the definition. If $\phi(i+1)=\phi(i)$ then $x_{i+1}\in S_\phi$ and we can write 
$c_{i,\phi(i)-1}=c_{i+1,\phi(i+1)-1}+x_{i+1}$, and the assertion follows from the induction assumption.

\noindent
(iii) The fact that $v_i\in \lan S_\phi\ran_{\Z_{\ge 0}}$ 
follows from the first definition of $S_\phi$ (to get $S_\phi$ from $S$ we keep all $v_i$ for $i\not\in\im(\phi)$ and replace each $v_i$
for $i\in \im(\phi)$ with a collection of vectors with the sum $v_i$).

Next, let us prove that $x_i\in \lan S_\phi\ran_{\Z_{\ge 0}}$. For $i=1$ this follows from the equality $x_1=-c_{1,k}$.
For $i>1$ we use the equality $x_i=c_{i-1,j}-c_{i,j}$. 
If $\phi(i)<\phi(i-1)$ then we set $j=\phi(i)$ and note that $-c_{i,j}\in \lan S_\phi\ran_{\Z_{\ge 0}}$ and
$c_{i-1,j}\in \lan S_\phi\ran_{\Z_{\ge 0}}$ since $j<\phi(i-1)$. In the case $\phi(i)=\phi(i-1)$ we have $x_i\in S_\phi$ by the definition.
\end{proof}

Let $\SS_{n,k}$ denote the toric hypersurface $X_1\ldots X_{n+1}=V_1\ldots V_k$ in the affine space with coordinates 
$(X_1,\ldots,X_{n+1},V_1,\ldots,V_k)$.
We have a natural toric birational morphism
$$\pi:\YY_{n,k}\to \SS_{n,k},$$
which is a resolution of singularities since $\YY_{n,k}$ is smooth.

Note that $\SS_{n,k}$ is the affine toric variety associated with the cone
$C\sub \R H^\vee$ spanned by the vectors $\om_{ij}=e_i+f_j$, $i=1,\ldots,n+1$, $j=1,\ldots,k$.
Here we use the presentation
$$H=(\Z x_1\oplus \ldots\oplus \Z x_{n+1}\oplus \Z v_1\oplus\ldots\oplus \Z v_k)/(x_1+\ldots+x_{n+1}-v_1-\ldots-v_k),$$
which allows to identify $\R H^\vee$ with the hyperplane $\a_1+\ldots+\a_{n+1}=\b_1+\ldots+\b_k$ in $\R^{n+1}\times \R^k$, and we
denote by $(e_1,\ldots,e_{n+1})$ and $(f_1,\ldots,f_k)$ the standard bases in $\R^{n+1}$ and $\R^k$.

Note that the vectors $\om_{ij}$ are precisely the vertices of the polytope
$$\Pi:=\De_n\times \De_{k-1}\sub \R^{n+1}\times \R^k,$$
where $\De_m\sub \R^{m+1}$ is the standard simplex spanned by the basis vectors.
Thus, $C$ is the cone generated by this polytope. Now taking cones over simplices of a standard triangulation of $\Pi$ 
(see e.g., \cite[Ch.\ 7, Sec.\ 3D]{GKZ}) one gets a toric fan, and the corresponding toric variety is a small resolution of $\SS_{n,k}$
(since we do not add new $1$-dimensional cones). We claim that this precisely leads to our variety $\YY_{n,k}$. 

\begin{prop}\label{small-res-prop} 
The map $\pi$ is the small resolution, associated with the standard triangulation of $\Pi$ corresponding to the usual ordering of the vertices of
$\De_n$ and the reverse ordering of the vertices of $\De_{k-1}$.
\end{prop}

\begin{proof} We need to identify the dual simplices to $\si^\vee_\phi$, where $\phi$ runs through nonincreasing maps $[1,n]\to [1,k]$,
with the simplices of the standard triangulation of $\Pi$ described in \cite[Ch.\ 7, Sec.\ 3D]{GKZ}. 
Note that due to our ordering of the vertices of $\De_{k-1}$, the latter simplices
are numbered by lattice paths from $(1,k)$ to $(n+1,1)$, i.e., paths obtained by starting at the point $(1,k)$ and moving either to the right or down
along the grid until reaching the point $(n+1,1)$. Alternatively, they can be numbered by shuffles $w$ of the word 
$A_1A_2\ldots A_nB_1B_2\ldots B_{k-1}$ (preserving the order of $A_i$'s and of $B_j$'s), where $A_i$ stands for the $i$th horizontal move
and $B_j$ stands for the $j$th vertical move. 

The simplex $\De_L$ of the triangulation associated with the lattice path $L$ is given inside $\Pi$ by the collection of inequalities:
$\a_1+\ldots+\a_i\le \b_k+\ldots+\b_j$ whenever $A_i$ precedes $B_j$ in the corresponding shuffled word $w_L$, and
$\a_1+\ldots+\a_i\ge \b_k+\ldots+\b_j$ whenever $B_j$ precedes $A_i$.

We have a natural bijection between lattice paths and nonincreasing functions $\phi:[1,n]\to [1,k]$: for a lattice path $L$
we let $\phi_L$ to be the unique function such that
$(i+1/2,\phi_L(i))$ lies on the path $L$ for each $i=1,\ldots,n$. Now it is easy to see that $A_i$ precedes $B_j$ in the shuffled word $w_L$
if and only if $\phi_L(i)\ge j$. Thus, due to formula \eqref{cij-formula-eq}
the simplex $\De_L$ is given inside $\Pi$ by the inequalities $\lan v,c_{i,j-1}\ran\ge 0$ whenever $\phi_L(i)\ge j$, and
$\lan v,-c_{i,j-1}\ran\ge 0$ whenever $\phi_L(i)<j$. Hence, the same inequalities give the cone $\si_L$ generated by $\De_L$ inside
the cone $C$.

It remains to observe that the cone $C$ is dual to the cone generated by all vectors $x_i$ and $v_j$ in $H$. Hence, 
Lemma \ref{toric-cones-lem} implies that the cone $\si_L$ is dual to the cone $\si^\vee_{\phi_L}$. Thus, $\YY_{n,k}$ is precisely the
toric variety associated with the fan with the maximal cones $\si_L$.
\end{proof}

\begin{cor}\label{CY-cor}
The canonical line bundle $\om_{\YY_{n,k}}$ is trivial.
\end{cor}


\section{A-side}

\subsection{Fukaya categories}\label{Fuk-gen-sec}

\subsubsection{Generalities}

In \cite{LPsymhms}, we studied the partially wrapped Fukaya category of
$\Sym^{n}(\Sigma)$ equipped with a stop of the form $Z \times
\mathrm{Sym}^{n-1} (\Sigma) $, where $\Sigma$ is a genus 0 surface with boundary
and $Z$ is either one point or two points lying on a single component of
$\partial\Sigma$. Our main tool was Auroux's determination of
generators for this category \cite{auroux} and the combinatorial description of their endomorphism algebras via strand algebras from \cite{LOT}. In the case where $Z$ is two points we explicitly wrote out the endomorphism algebra of a particular generator of these categories. 

Here, we will concentrate in the case where $Z$ is a single point marked at the outer boundary component as in Figure \ref{Sigmagen}. We will use the same Lagrangians as in \cite{LPsymhms}. Namely, they are given by products of $n$ disjoint arcs among the Lagrangians drawn on $\Sigma$ in Figure \ref{Sigmagen} except that as we only consider the case $Z$ is a single point, the last arc $L_{n+k}$ is not needed. Thus, our generators are of the form
\[ L_S = L_{i_1} \times L_{i_2} \times \ldots \times L_{i_n} \]
where $S$ is a subset of $[0,n+k-1]$ of size $n$. The computation given in \cite[Thm 3.2.5]{LPsymhms} readily gives the endomorphism algebras we need, and the corresponding $A_\infty$-algebra is formal. 

In order to avoid repetition, we have chosen to skip over many of the details; the reader is referred to \cite[Sections 2 \& 3]{LPsymhms} for a complete background on these calculations. We also note that in \cite{LPsymhms}, $\Sigma$ was a $(k+1)$-punctured surface of genus 0. Here, we let $\Sigma$ to be the genus 0 surface with $(n+k+1)$-punctures. This slight modification in notation turns out to be more natural.  Finally, we note that in \cite{LPsymhms} we worked with $\mathbb{Z}$-graded categories (for all possible grading structures), whereas here we consider the $\bbL$-graded categories, which in a sense is equivalent to working with all the $\mathbb{Z}$-grading structures compatible with a given $\Z/2$-grading structure. The latter structure
is characterized by the fact that all endomorphisms of our generators have even grading (this choice is made in order to
get an equivalence with the B-side).

\begin{figure}[ht!]
\centering
\begin{tikzpicture}
\begin{scope}[scale=0.8]

\tikzset{
  with arrows/.style={
    decoration={ markings,
      mark=at position #1 with {\arrow{>}}
    }, postaction={decorate}
  }, with arrows/.default=2mm,
}

\tikzset{vertex/.style = {style=circle,draw, fill,  minimum size = 5pt,inner        sep=1pt}}
\def \radius {1.5cm}

\foreach \s in {2.5,3.5,4.5,5.5,6.5} {
   \draw[thick] ([shift=({360/10*(\s)}:\radius+1.5cm)]0,0) arc ({360/10 *(\s)}:{360/10*(\s+1)}:\radius+1.5cm);
}

\foreach \s in {7.5,8.5,9.5,10.5,11.5} {
   \draw[thick] ([shift=({360/10*(\s)}:\radius+1.5cm)]11,0) arc ({360/10 *(\s)}:{360/10*(\s+1)}:\radius+1.5cm);
}

\draw[thick] (0,3) -- (11,3);
\draw[thick] (0,-3) -- (11,-3);
\foreach \s in {0.5,1.5,2.5,3.5,4.5,5.5,6.5,7.5,8.5,9.5} {
    
\draw[thick] ([shift=({360/10*(\s)}:\radius-1cm)]1,0) arc ({360/10 *(\s)}:{360/10*(\s+1)}:\radius-1cm);
\draw[thick] ([shift=({360/10*(\s)}:\radius-1cm)]-1,0) arc ({360/10 *(\s)}:{360/10*(\s+1)}:\radius-1cm);
\draw[thick] ([shift=({360/10*(\s)}:\radius-1cm)]12,0) arc ({360/10 *(\s)}:{360/10*(\s+1)}:\radius-1cm);
\draw[thick] ([shift=({360/10*(\s)}:\radius-1cm)]10,0) arc ({360/10 *(\s)}:{360/10*(\s+1)}:\radius-1cm);
\draw[thick] ([shift=({360/10*(\s)}:\radius-1cm)]6,0) arc ({360/10 *(\s)}:{360/10*(\s+1)}:\radius-1cm);
\draw[thick] ([shift=({360/10*(\s)}:\radius-1cm)]4,0) arc ({360/10 *(\s)}:{360/10*(\s+1)}:\radius-1cm);

}

\draw[with arrows] ([shift=({360/10*(7)}:\radius-1cm)]1,0) arc ({360/10*(7)}:{360/10*(5)}:\radius-1cm);
\draw[with arrows] ([shift=({360/10*(7)}:\radius-1cm)]-1,0) arc ({360/10*(7)}:{360/10*(5)}:\radius-1cm);
\draw[with arrows] ([shift=({360/10*(7)}:\radius-1cm)]10,0) arc ({360/10*(7)}:{360/10*(5)}:\radius-1cm);
\draw[with arrows] ([shift=({360/10*(7)}:\radius-1cm)]12,0) arc ({360/10*(7)}:{360/10*(5)}:\radius-1cm);
\draw[with arrows] ([shift=({360/10*(7)}:\radius-1cm)]4,0) arc ({360/10*(7)}:{360/10*(5)}:\radius-1cm);
\draw[with arrows] ([shift=({360/10*(7)}:\radius-1cm)]6,0) arc ({360/10*(7)}:{360/10*(5)}:\radius-1cm);
\draw[with arrows]({360/10 * (3)}:\radius+1.5cm) arc ({360/10 *(3)}:{360/10*(4)}:\radius+1.5cm);
\draw [blue] ([shift=({360/10*(10)}:\radius-1cm)]4,0) --  ([shift=({360/10*(5)}:\radius-1cm)]6,0); 
\draw [blue] ([shift=({360/10*(10)}:\radius-1cm)]0,0) --  ([shift=({360/10*(5)}:\radius-1cm)]0.4,0); 
\draw [blue] ([shift=({360/10*(10)}:\radius-1cm)]-1,0) --  ([shift=({360/10*(5)}:\radius-1cm)]0.5,0); 
\draw [blue] ([shift=({360/10*(10)}:\radius-1cm)]-3.4,0) --  ([shift=({360/10*(5)}:\radius-1cm)]-1,0); 
\draw [blue] ([shift=({360/10*(10)}:\radius-1cm)]11,0) --  ([shift=({360/10*(5)}:\radius-1cm)]11.5,0); 
\draw [blue] ([shift=({360/10*(10)}:\radius-1cm)]10,0) --  ([shift=({360/10*(5)}:\radius-1cm)]11.5,0); 
\draw [black, dashed] ([shift=({360/10*(10)}:\radius-1cm)]1.2,0) --  ([shift=({360/10*(5)}:\radius-1cm)]3.7,0); 
\draw [black, dashed] ([shift=({360/10*(10)}:\radius-1cm)]6.2,0) --  ([shift=({360/10*(5)}:\radius-1cm)]9.9,0); 

\node[black] at (-1,0.7) {\small $\ell_1$};
\node[black] at (1,0.7) {\small $\ell_2$};
\node[black] at (4,0.7) {\small $\ell_{k}$};
\node[black] at (6,0.7) {\small $\ell_{k+1}$};
\node[black] at (10,0.7) {\small $\ell_{n+k-1}$}; 
\node[black] at (-1,-0.9) {\small $r_1$};
\node[black] at (1,-0.9) {\small $r_2$};
\node[black] at (4,-0.9) {\small $r_k$};
\node[black] at (6,-0.9) {\small $r_{k+1}$};
\node[black] at (10,-0.9) {\small $r_{n+k-1}$}; 
\node[black] at (-1,0) {\tiny $u_1$};
\node[black] at (1,0) {\tiny $u_2$};
\node[black] at (4,0) {\tiny $u_k$};
\node[black] at (6,0) {\tiny $x_{1}$};
\node[black] at (10,0) {\tiny $x_{n-1}$};
\node[black] at (12,0) {\tiny $x_n$};
\node[blue] at (-2,-0.5) {\small $L_{0}$};
\node[blue] at (0,-0.5) {\small $L_{1}$};
\node[blue] at (5,-0.6) {\small $L_{k}$};
\node[blue] at (11.2,-0.6) {\small $L_{n+k-1}$};
\draw[purple] (4.7, 3) to[in=180,out=270] (5,2.6);
\draw[purple] (5, 2.6) to[in=270,out=0] (5.3,3);
\node[purple] at ([shift=({360/10*(7)}:\radius+1.5cm)]6,5.1) {\small $T$}; 
\node[vertex] at  ([shift=({360/10*(2.5)}:\radius+1.5cm)]5,0) {};

\end{scope}

\end{tikzpicture}
\caption{Surface $\Sigma$ and generating Lagrangians}
\label{Sigmagen}
\end{figure}
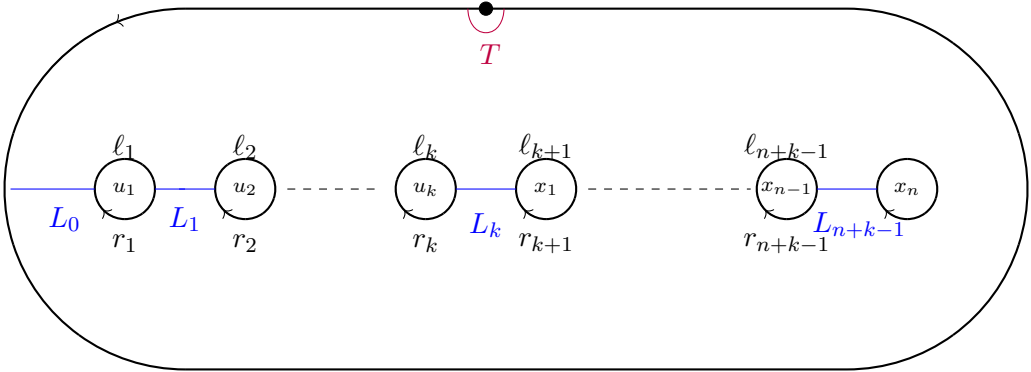

\subsubsection{The universal grading group and its splitting}

Let us set $V= \Sym^n(\Sigma)$.
Note that there is an isomorphism between $H_1(V)$ and $H_1(\Sigma) = \mathbb{Z}^{n+k}$ and $c_1(V)=0$. 
Thus, the grading group $\bbL = H_1(\mathcal{G}V)$, where $\GG V$ is the bundle of Lagrangian subspaces in
$TV$, fits into an exact sequence
\begin{equation}\label{A-side-grading-sequence}
0 \to \mathbb{Z} \to \bbL \to H_1(\Sigma) \to 0 
\end{equation}

In Lemma \ref{A-side-grading-splitting-lem} below, we will see that there is a canonical splitting of this sequence. Using this splitting, we can lift the boundary loops $u_1,\ldots, u_k, x_1, \ldots, x_n$ to homology classes in $\bbL = H_1(\mathcal{G}V)$ (denoted in the same way) and we choose the unique $\mathbb{Z}/2$-grading structure 
\[ \sigma : \mathbb{L} \to \mathbb{Z}/2 \]
which sends $u_1,\ldots u_k, x_1,\ldots, x_n$ to zero (and $\sigma(l_0)=1$).  
This particular choice of $\mathbb{Z}/2$-grading structure is made in order to match the $B$-side computations.

Following \cite{sheridan}, the grading of a morphism in a partially wrapped Fukaya category can be defined as follows. 
Suppose $L$ and $L'$ are two Lagrangians in $V$ equipped with grading structures, i.e., lifts of their Gauss maps $L, L' \to \mathcal{G}V$ to $\tilde{\mathcal{G}}V$, the universal abelian cover of $\GG V$. Note that the group $\bbL$ acts on $\tilde\GG V$ by deck transformations. This gives a
simply transitive action of $\bbL$, on the set of grading structures on a Lagrangian $L$, which we denote as
$L\mapsto L(\ell)$, where $\ell\in \bbL$. By definition, one has $\hom^{\ell_1}(L,L'(\ell_2))=\hom^{\ell_1+\ell_2}(L,L')$,
so it is enough to characterize which morphisms have degree $0\in \bbL$.

We have the following commutative diagram
\[ \begin{tikzcd} \tilde{\mathcal{G}}V \arrow[r]\arrow[d]& \tilde{V}\arrow[d] \\ \mathcal{G}V \arrow[r]& V
\end{tikzcd} \]
where $\tilde{V}$ is the universal abelian cover of $V$. Note that the top arrow $\tilde{\mathcal{G}}V \to \tilde{V}$ also can be identified as the (unique) fibre-wise universal cover of the Lagrangian Grassmannian $\mathcal{G}\tilde{V}$. 
Thus, the symplectic manifold $\tilde{V}$ is equipped with a (unique) $\Z$-grading structure.

Now assume that $\gamma \in \mathrm{hom}(L,L')$ is a Reeb chord.
The grading structures on $L,L'$ determine lifts $\tilde{L}$ and $\tilde{L}'$ of $L$ and $L'$ to $\tilde{V}$. Furthermore, $\tilde{L}$ and $\tilde{L}'$
come equipped with $\Z$-grading structures, so morphisms between $\tilde{L}$ and $\tilde{L}'$ in the partially wrapped Fukaya category of $\tilde{V}$
are equipped with $\Z$-grading.
The path $\gamma: L\to L'$ lifts to a unique path $\tilde{\gamma}: \tilde{L} \to \overline{L}'$ where $\overline{L}'$ is a lift of $L'$ to $\tilde{V}$,
possibly different from $\tilde{L}'$. Now we say that $\gamma$ has degree $0\in \bbL$ if $\overline{L}'=\tilde{L}'$ and the $\Z$-degree of
the morphism given by $\tilde{\gamma}$ is $0$.
\medskip
\medskip

Note that the classes of the boundary loops $v_1,\ldots,v_k,x_1,\ldots,x_n$ (see Figure \ref{Sigmagen}) form a basis of $H_1(\Sigma)$. Let us set $u_i$ the be the circle associated oriented with the opposite orientation to $v_i$ so that its boundary orientation agrees with the boundary orientation of the surface.
We are going to use these classes to define a splitting of the sequence \eqref{A-side-grading-sequence}.

\begin{lem}\label{A-side-grading-splitting-lem}
There exist canonical lifts of the classes of the boundary loops $u_1,\ldots,u_k$, and $x_1,\ldots,x_n$ to $\bbL$ (which we will denote by the same letters), such that
the $\bbL$-degree of any nonzero endomorphism $u_i\in \End(L_S)$ (resp., $x_i\in \End(L_S)$) in the partially wrapped Fukaya category of $V=\Sym^n(\Sigma)$
is given by this lift, i.e the $\bbL$-degree is independent of $S$.
\end{lem}

\begin{proof}
Recall that away from the big diagonal $\Delta \subset \Sym^n(\Sigma)$, the Reeb flow along the contact boundary is given by the product of the Reeb flow on the surface and this agrees with the circular flow around the boundary components of $\Sigma$ in the direction induced by the boundary orientations. Therefore, to an orbit $\gamma$ going once around a boundary component of $\Sigma$, we can associate a Reeb orbit of the form 
$$\tilde{\gamma} = p_1 \times \ldots \times p_{n-1} \times \gamma$$ 
in $\Sym^n(\Sigma) \setminus \Delta$, where $p_1,\ldots,p_{n-1}$ are fixed points. 

Given a boundary component $\gamma$ of $\Sigma$, other than the one marked with a stop, and an arc $L_i$ intersecting it, for any subset
$S=\{i,i_1,\ldots,i_{n-1}\}$,
we construct as above the Reeb orbit $\tilde{\gamma} =p_1 \times \ldots \times p_{n-1} \times \gamma$ with $\tilde{\gamma}(0)=\tilde{\gamma}(1) \in L_S$,
by fixing some points $p_1\in L_{i_1},\ldots,p_{n-1}\in L_{i_{n-1}}$.
We need to check that the grading in $\bbL$ of the corresponding morphism in $\mathrm{hom}(L_S,L_S)$ does not depend on $S$. 
Suppose we are given some other $S'$ containing $i$, $S'=\{i,i'_1,\ldots,i'_{n-1}\}$ and some fixed points $p'_1\in L_{i'_1},\ldots,p'_{n-1}\in L'_{i'_{n-1}}$,
so that we have the corresponding Reeb chord $\tilde{\gamma}' =p'_1 \times \ldots \times p'_{n-1} \times \gamma$.
Now we note that the definition of the grading of $\tilde{\gamma}$ does not change if we replace $L_S$ by its open subset
$U_S=L_i\times U_{i_1}\times\ldots\times U_{i_{n-1}}$, where $U_{i_r}$ is an open neighborhood of $p_{i_r}$ in $L_{i_r}$.
Similarly, we can replace $L_{S'}$ by an open subset $U_{S'}$. 
Now we can deform $(U_S,\tilde{\gamma})$ to $(U_S',\tilde{\gamma}')$ continuously by moving the fixed points $p_1,\ldots,p_{n-1}$ and their neighborhoods
in the arcs $L_{i_1},\ldots,L_{i_{n-1}}$ to the points $p'_1,\ldots,p'_{n-1}$ and their neighborhoods in the arcs $L_{i'_1},\ldots,L_{i'_{n-1}}$.
Clearly, such a continuous deformation does not change the calculation of the grading. 

        Next, given another arc $L_j$ with $j=i\pm 1$, intersecting $\gamma$, and indices $i_1,\ldots,i_{n-1}$ distinct from $i$ and $j$, we have
to check that the grading of the chord from $L_i\times L_{i_1}\times\ldots\times L_{i_{n-1}}$ to itself, induced by $\gamma$, is equal to the
grading of the similar chord from $L_j\times L_{i_1}\times\ldots L_{i_{n-1}}$. For this we use a similar trick: we can replace $L_i$ and $L_j$ by neighborhoods
of the point of intersection with $\gamma$ and deform one into another.

Finally, for any indices $i_1,\ldots,i_{n-2}$, we can check that the gradings of the two chords from $L_i\times L_j\times L_{i_1}\times\ldots\times L_{i_{n-2}}$
to itself, induced by $\gamma$ (which we can stick either in the first or in the second factor) are the same: we replace $L_i$ and $L_j$ by neighborhoods
of the points of intersection with $\gamma$ and move them along $\gamma$ until they swap.
\end{proof}

Lemma \ref{A-side-grading-splitting-lem} allows us to determine the $\bbL$-grading on the endomorphism algebra of our generators. Indeed, the $\bbL$-grading structure can be computed as in \cite[Thm.\ 3.2.5(iii)]{LPsymhms}, which shows that up to shifts of the generators, the grading is uniquely determined by the degrees of $u_r$ and $x_r$ in $\bbL$.

\subsection{Endomorphism algebra of the generating Lagrangians: case $n=1$}
\ \\
\ \\
The generating objects are 
\[ L_0, L_1, \ldots, L_k \] 
We can express the endomorphism algebra 
\[ \mathcal{A}_{1,k} = \bigoplus_{i,j=0}^k \mathrm{hom}(L_i, L_j)\]
via the following quiver 
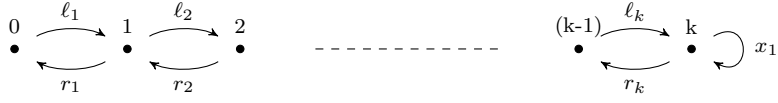
\begin{figure}[H]
\begin{tikzpicture}
\tikzset{vertex/.style = {style=circle,draw, fill,  minimum size = 2pt,inner    sep=1pt}} 
\tikzset{edge/.style = {->,-stealth',shorten >=8pt, shorten <=8pt  }}
\node[vertex] (a) at  (0,0) {};
\node[vertex] (a1) at (1.5,0) {};
\node[vertex] (a2) at (3,0) {};

\node[vertex] (a3) at (7.5,0) {};
\node[vertex] (a4) at (9,0) {};

\node at  (0,0.3) {\tiny 0};
\node at (1.5,0.3) {\tiny 1}; 
\node at (3,0.3) {\tiny 2};
        \node at (7.5,0.3) {\tiny (k-1)};
\node at (9,0.3) {\tiny k};


\draw[edge] (a) to[in=150,out=30] (a1);
\draw[edge] (a1) to[in=330,out=210] (a);
\draw[edge] (a1) to[in=150,out=30] (a2);
\draw[edge] (a2) to[in=330,out=210] (a1);
\draw[edge] (a3) to[in=150,out=30] (a4);
\draw[edge] (a4) to[in=330,out=210] (a3);
\draw[edge, looseness=40] (a4) to[in=330,out=30] (a4);

\draw[black, dashed] (4,0) to (6.5,0);

\node at (0.75,-0.5) {\tiny $r_1$};
\node at (0.75,0.5) {\tiny $\ell_1$}; 
\node at (2.25,-0.5) {\tiny $r_2$};
\node at (2.25,0.5) {\tiny $\ell_2$};
\node at (8.25,-0.5) {\tiny $r_k$};
\node at (8.25,0.5) {\tiny $\ell_k$};
\node at (10, 0) {\tiny $x_1$}; 
\end{tikzpicture}
\label{endn1}
\end{figure}
with relations
\begin{align*} 
        \ell_{i+1} \ell_i &= 0 = r_i r_{i+1} \\
        x_1 \ell_k &= 0 = r_k x_1
\end{align*}
Let us work with the $\mathbb{Z}$-graded category for which $\deg(\ell_i)=\deg(r_i)=1$ for $i\le k$ and $\deg(x_1)=0$. To remove the stop, we need to localize by the subcategory generated by an arc $T$ around the stop. It is easy to see that
$T$ can be represented by the twisted complex
\begin{equation}\label{n1-T-resolution} 
L_0[1] \xrightarrow{\ell_1} L_1[1] \xrightarrow{\ell_2} \ldots \xrightarrow{\ell_{k-1}} L_{k-1}[1] \xrightarrow{\ell_k} L_k[1] \xrightarrow{x_1} L_k \xrightarrow{r_k} L_{k-1} \ldots \xrightarrow{r_2} L_1 \xrightarrow{r_1} L_0. \end{equation}

\subsection{Endomorphism algebra of the generating Lagrangians: case $n=2$}
\ \\
\ \\
The generating objects are 
\[ L_{i,j} \coloneqq L_i \times L_j \text{ for } i,j \in \{0,1,\ldots, k+1\} \text{ with } i\neq j.\] 

The endomorphism algebra is defined by
\[ \mathcal{A}_{2,k} = \bigoplus_{S,S'} \mathrm{hom}(L_S, L_{S'})\]
where the sum is over distinct two element subsets of $\{0,1\ldots, k+1\}$. 
Let us write $e_{ij} \in \mathrm{end}(L_S)$ for the idempotent corresponding to $S = \{i,j\}$.

$\mathcal{A}_{2,k}$ can be described via the following quiver over $\mathbf{k}[u_1,u_2,\ldots,u_k,x_1,x_2]$ 
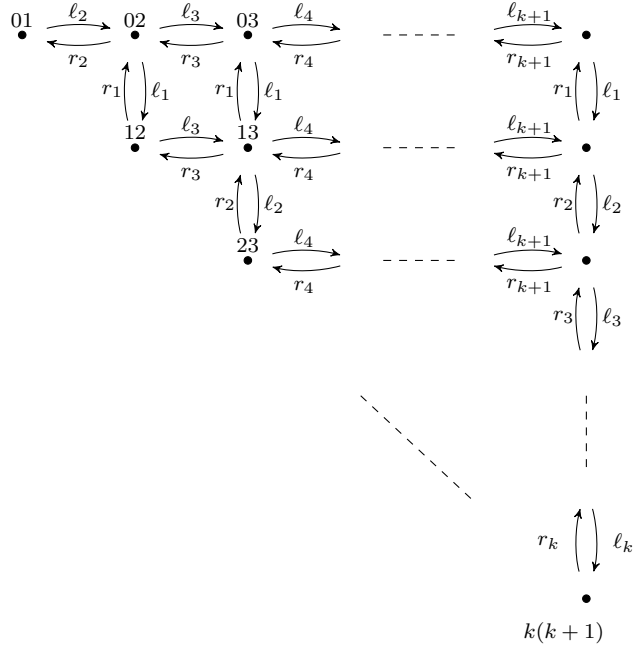
\begin{figure}[H]
\begin{tikzpicture}

\tikzset{vertex/.style = {style=circle,draw, fill,  minimum size = 2pt,inner    sep=1pt}} 
\tikzset{edge/.style = {->,-stealth',shorten >=8pt, shorten <=8pt  }}
\tikzset{vedge/.style = {->,-stealth',shorten >=9pt, shorten <=9pt  }}

\node[vertex] (a1) at (0,0) {};
\node[vertex] (a2) at (1.5,0) {};
\node[vertex] (a3) at (3,0) {};
\node[vertex] (a6) at (7.5,0) {};

\node[vertex] (b2) at (1.5,-1.5) {};
\node[vertex] (b3) at (3,-1.5) {};
\node[vertex] (b6) at (7.5,-1.5) {};

\node[vertex] (c3) at (3,-3) {};
\node[vertex] (c6) at (7.5,-3) {};

\node[vertex] (f6) at (7.5,-7.5) {};

\node at (0,0.2) {\tiny $01$};
\node at (1.5,0.2) {\tiny $02$};
\node at (3,0.2) {\tiny $03$};
\node at(1.5,-1.3) {\tiny $12$};
\node at(3,-1.3) {\tiny $13$};

\node at(3,-2.8) {\tiny $23$};

\node at (7.2,-7.95) {\tiny $k(k+1)$};


\draw[edge] (a1) to[in=165,out=15] (a2);
\draw[edge] (a2) to[in=345,out=195] (a1);
\draw[edge] (a2) to[in=165,out=15] (a3);
\draw[edge] (a3) to[in=345,out=195] (a2);
\draw[edge] (a3) to[in=165,out=15] (4.5,0);
\draw[edge] (4.5,0) to[in=345,out=195] (a3);
\draw[edge] (6,0) to[in=165,out=15] (a6);
\draw[edge] (a6) to[in=345,out=195] (6,0);

\draw[edge] (b2) to[in=165,out=15] (b3);
\draw[edge] (b3) to[in=345,out=195] (b2);
\draw[edge] (b3) to[in=165,out=15] (4.5,-1.5);
\draw[edge] (4.5,-1.5) to[in=345,out=195] (b3);
\draw[edge] (6,-1.5) to[in=165,out=15] (b6);
\draw[edge] (b6) to[in=345,out=195] (6,-1.5);

\draw[edge] (c3) to[in=165,out=15] (4.5,-3);
\draw[edge] (4.5,-3) to[in=345,out=195] (c3);
\draw[edge] (6,-3) to[in=165,out=15] (c6);
\draw[edge] (c6) to[in=345,out=195] (6,-3);

\draw[vedge] (a2) to[in=75,out=285] (b2);
\draw[vedge] (b2) to[in=255,out=105] (a2);

\draw[vedge] (a3) to[in=75,out=285] (b3);
\draw[vedge] (b3) to[in=255,out=105] (a3);

\draw[vedge] (b3) to[in=75,out=285] (c3);
\draw[vedge] (c3) to[in=255,out=105] (b3);

\draw[vedge] (a6) to[in=75,out=285] (b6);
\draw[vedge] (b6) to[in=255,out=105] (a6);

\draw[vedge] (b6) to[in=75,out=285] (c6);
\draw[vedge] (c6) to[in=255,out=105] (b6);

\draw[vedge] (c6) to[in=75,out=285] (7.5,-4.5);
\draw[vedge] (7.5,-4.5) to[in=255,out=105] (c6);

\draw[vedge] (7.5,-6) to[in=75,out=285] (f6);
\draw[vedge] (f6) to[in=255,out=105] (7.5,-6);




\draw[black, dashed] (4.8,0) to (5.8,0) ;
\draw[black, dashed] (4.8,-1.5) to (5.8,-1.5) ;
\draw[black, dashed] (4.8,-3) to (5.8,-3) ;

\draw[black, dashed] (7.5,-4.8) to (7.5,-5.8) ;

\draw[black, dashed] (4.5,-4.8) to (6,-6.2) ;


\node at (0.75,-0.35) {\tiny $r_2$};
\node at (0.75,0.3) {\tiny $\ell_2$}; 
\node at (2.25,-0.35) {\tiny $r_3$};
\node at (2.25,0.3) {\tiny $\ell_3$};
\node at (3.75,0.3) {\tiny $\ell_4$};
\node at (3.75,-0.35) {\tiny $r_4$};
\node at (6.75,-0.35) {\tiny $r_{k+1}$};
\node at (6.75,0.3) {\tiny $\ell_{k+1}$};

\node at (2.25,-1.85) {\tiny $r_3$};
\node at (2.25,-1.2) {\tiny $\ell_3$};
\node at (3.75,-1.85) {\tiny $r_4$};
\node at (3.75,-1.2) {\tiny $\ell_4$};
\node at (6.75,-1.85) {\tiny $r_{k+1}$};
\node at (6.75,-1.2) {\tiny $\ell_{k+1}$};


\node at (3.75,-3.35) {\tiny $r_4$};
\node at (3.75,-2.7) {\tiny $\ell_4$};
\node at (6.75,-3.35) {\tiny $r_{k+1}$};
\node at (6.75,-2.7) {\tiny $\ell_{k+1}$};


\node at (1.2,-0.75) {\tiny $r_1$};
\node at (2.7,-0.75) {\tiny $r_1$};
\node at (7.2,-0.75) {\tiny $r_1$};

\node at (1.85,-0.75) {\tiny $\ell_1$};
\node at (3.35,-0.75) {\tiny $\ell_1$};
\node at (7.85,-0.75) {\tiny $\ell_1$};

\node at (2.7,-2.25) {\tiny $r_2$};
\node at (7.2,-2.25) {\tiny $r_2$};
\node at (3.35,-2.25) {\tiny $\ell_2$};
\node at (7.85,-2.25) {\tiny $\ell_2$};

\node at (7.2,-3.75) {\tiny $r_3$};
\node at (7.85,-3.75) {\tiny $\ell_3$};

\node at (7,-6.75) {\tiny $r_k$};
\node at (8,-6.75) {\tiny $\ell_k$};

\end{tikzpicture}
\label{endn2}
\end{figure}
with relations, using the convention that $x_1= u_{k+1}, x_2=u_{k+2}$,
\begin{align*}
        &u_r e_{ij}=0 \text{ if }r\neq i,i+1,j,j+1, \\
        &\ell_{i+1} \ell_i = 0 = r_i r_{i+1} \text{ for all } i,\\
        &r_i r_{i+j} = r_{i+j} r_i \text{ for all } i \text{ and for all } j \geq 2, \\
        &\ell_i \ell_{i+j} = \ell_{i+j} \ell_i \text{ for all } i \text{ and for all } j \geq 2,  \\
        &\ell_i r_{i+j} = r_{i+j} \ell_i \text { for all } i \text{ and for all } j \geq 2,  \\
        &r_i \ell_{i+j} = \ell_{i+j} r_i \text { for all } i \text{ and for all } j \geq 2,  \\
        &r_i \ell_i = u_i = \ell_i r_i  \text{ for } i = 1,\ldots, k+1, \\
        &x_2 \ell_{k+1} = 0 = r_{k+1} x_2. 
\end{align*}

To remove the stop, we need to localize by the subcategory generated by
\[ T \times L_i , \text{ for } i=0,\ldots, k+1 \]
where $T$ is the arc around the stop. We will identify resolutions for these objects in terms of the generators $(L_{i,j})$.
Let us work in the $\Z$-graded Fukaya category,
where we use the gradings of $(L_{i,j})$ such that $\deg(\ell_i)=\deg(r_i)=1$ for $i\le k$ and $\deg(\ell_{k+1})=\deg(r_{k+1})=\deg(x_2)=0$.

\begin{prop} For $i\le k$, $T\times L_i$ is represented by the twisted complex 
\begin{equation}\label{n2-T-resolution} 
\begin{array}{l}
L_{0,i}[1]\xrightarrow{\ell_1}\ldots \xrightarrow{\ell_{i-1}} L_{i-1,i}[1] \xrightarrow{\ell_{i}\times\ell_{i+1}} L_{i,i+1}[2] \xrightarrow{\ell_{i+2}} \ldots 
\xrightarrow{\ell_k} L_{i,k}[2] \xrightarrow{\ell_{k+1}} L_{i,k+1}[1] \xrightarrow{x_2} \\
L_{i,k+1} \xrightarrow{r_{k+1}} L_{i,k}[-1]\xrightarrow{r_k}\ldots \xrightarrow{r_{i+2}} L_{i,i+1}[-1] \xrightarrow{r_i\times r_{i+1}}
L_{i-1,i}\xrightarrow{r_{i-1}}\ldots\xrightarrow{r_1}L_{0,i},
\end{array}
\end{equation}
while $T\times L_{k+1}$ is represented by
\begin{equation}\label{n2-T-resolution-bis} 
L_{0,k+1}[1]\xrightarrow{\ell_1}\ldots\xrightarrow{\ell_k}L_{k,k+1}[1]\xrightarrow{x_1}L_{k,k+1}\xrightarrow{r_k}\ldots
\xrightarrow{r_1}L_{0,k+1}.
\end{equation}
\end{prop}

\begin{proof} By considering the $n=1$ case, we already know that $T$ can be represented by the twisted complex
        \[ L_0[1] \xrightarrow{\ell_1} L_1[1] \xrightarrow{\ell_2} \ldots \xrightarrow{\ell_{k}} L_{k}[1] \xrightarrow{\ell_{k+1}} L_{k+1}[1] \xrightarrow{x_2} L_{k+1} \xrightarrow{r_{k+1}} L_{k} \ldots \xrightarrow{r_2} L_1 \xrightarrow{r_1} L_0 
        \]
        Since this description involves $L_i$, we cannot directly multiply this complex with $L_i$ to get a description of $T \times L_i$. We instead will argue in two steps by using the arcs $U$ and $D$ drawn in Figure \ref{auxlag}. This applies for $i \leq k$ and for $i=k+1$ we use the arc $E$ which plays a similar role. Using these, we see that $T$ can be represented by 
        \begin{equation}
                \begin{array}{l} L_0[1] \xrightarrow{\ell_1} \cdots L_{i-2}[1]  \to U\to L_{i+2}[1] \xrightarrow{\ell_{i+3}} \cdots \xrightarrow{\ell_{k+1}} L_{k+1}[1] \xrightarrow{x_2} L_{k+1} \xrightarrow{r_{k+1}}  \\ L_{i+2}[-1] \to D \to L_{i-2} \cdots \xrightarrow{r_1} L_0 
       \end{array}
        \end{equation}
for $i \leq k$ and 
        \begin{equation}
                \begin{array}{l} L_0[1] \xrightarrow{\ell_1} \cdots \xrightarrow{\ell_{k-1}} L_{k-1}[1] \to E \to L_{k-1} \xrightarrow{r_{k-1}} \cdots \xrightarrow{r_1} L_0 
       \end{array}
        \end{equation}
for $i=k+1$. 

Since none of these complexes involve $L_i$, we can multiply every term by $L_i$ to get a twisted complex describing $T \times L_i$. 

Thus, it remains to describe $U \times L_i$, $D \times L_i$ for $i\leq k$ and $E \times L_{k+1}$ in terms of our generators. We will prove that $U \times L_i$ is given by the cone 
        \[ L_{i-1} \times L_i \xrightarrow{\ell_i \times \ell_{i+1}} L_i \times L_{i+1}\] 
        The other cases are similar. 
        The existence of this triangle can be seen as follows: First slide $L_{i-1}$ over $L_i$ to $\tilde{U}$. As this is a Hamiltonian isotopy, this gives an equivalence $L_{i-1} \times L_i \simeq \tilde{U} \times L_i$. Next, observe that sliding $\tilde{U}$ over $L_{i+1}$ gives $U$, hence we have an exact $C$ is given as a cone $\tilde{U} \to L_{i+1}$. Multiplying this with $L_i$, we conclude that $L_{i-1} \times L_i \simeq \tilde{U} \times L_i \to L_{i} \times L_{i+1}$ describes the object $U \times L_i$. 
        
        Finally, since none of the slides involved lower part of the surface, the morphism in the cone description has to be $\ell_i \times \ell_{i+1}$ as any other morphism from $L_{i-1} \times L_i \to L_i \times L_{i+1}$ goes around the holes. Indeed, one can see that this exact triangle in $\bbL$-graded setting which necessitates that the connecting morphism to be $\ell_i \times \ell_{i+1}$. Similarly, note that there exists no morphism from $L_{i-2} \times L_i$ to $L_i \times L_{i+1}$ (resp. $L_{i-1} \times L_i$ to $L_i \times L_{i+2}$) and there exists a unique morphism from $L_{i-2} \times L_i$ to $L_{i-1} \times L_i$ (resp. $L_i \times L_{i+1} \to L_i \times L_{i+2}$), namely $\ell_{i-1} \times \id_{{L}_i}$ (resp. $\id_{{L}_i} \times \ell_{i+2}$), which does not go around the holes. Hence, for reasons of $\bbL$-grading, it follows that the morphisms in the twisted complex representation of the $T \times L_i$ in terms of our generators must be those given in Equations \ref{n2-T-resolution} and \ref{n2-T-resolution-bis}. 
\end{proof}

\begin{figure}[ht!]
\centering
\begin{tikzpicture}

\begin{scope}[scale=0.8]

\tikzset{
  with arrows/.style={
    decoration={ markings,
      mark=at position #1 with {\arrow{>}}
    }, postaction={decorate}
  }, with arrows/.default=2mm,
} 

\tikzset{vertex/.style = {style=circle,draw, fill,  minimum size = 5pt,inner        sep=1pt}}
\def \radius {1.5cm}

\foreach \s in {2.5,3.5,4.5,5.5,6.5} {
  
   \draw[thick] ([shift=({360/10*(\s)}:\radius+1.5cm)]0,0) arc ({360/10 *(\s)}:{360/10*(\s+1)}:\radius+1.5cm);
}

\foreach \s in {7.5,8.5,9.5,10.5,11.5} {
  
   \draw[thick] ([shift=({360/10*(\s)}:\radius+1.5cm)]11,0) arc ({360/10 *(\s)}:{360/10*(\s+1)}:\radius+1.5cm);
}

\draw[thick] (0,3) -- (11,3);
\draw[thick] (0,-3) -- (11,-3);

\foreach \s in {0.5,1.5,2.5,3.5,4.5,5.5,6.5,7.5,8.5,9.5} {
 
    \draw[thick] ([shift=({360/10*(\s)}:\radius-1.1cm)]-1,0) arc ({360/10 *(\s)}:{360/10*(\s+1)}:\radius-1.1cm);

        \draw[thick] ([shift=({360/10*(\s)}:\radius-1.1cm)]1,0) arc ({360/10 *(\s)}:{360/10*(\s+1)}:\radius-1.1cm);

    \draw[thick] ([shift=({360/10*(\s)}:\radius-1.1cm)]3,0) arc ({360/10 *(\s)}:{360/10*(\s+1)}:\radius-1.1cm);
    \draw[thick] ([shift=({360/10*(\s)}:\radius-1.1cm)]5,0) arc ({360/10 *(\s)}:{360/10*(\s+1)}:\radius-1.1cm);
  \draw[thick] ([shift=({360/10*(\s)}:\radius-1.1cm)]9,0) arc ({360/10 *(\s)}:{360/10*(\s+1)}:\radius-1.1cm);
  \draw[thick] ([shift=({360/10*(\s)}:\radius-1.1cm)]11,0) arc ({360/10 *(\s)}:{360/10*(\s+1)}:\radius-1.1cm);
  \draw[thick] ([shift=({360/10*(\s)}:\radius-1.1cm)]13,0) arc ({360/10 *(\s)}:{360/10*(\s+1)}:\radius-1.1cm);

} 

\draw[with arrows] ([shift=({360/10*(7)}:\radius-1.1cm)]-1,0) arc ({360/10*(7)}:{360/10*(5)}:\radius-1.1cm);

\draw[with arrows] ([shift=({360/10*(7)}:\radius-1.1cm)]1,0) arc ({360/10*(7)}:{360/10*(5)}:\radius-1.1cm);

\draw[with arrows] ([shift=({360/10*(7)}:\radius-1.1cm)]3,0) arc ({360/10*(7)}:{360/10*(5)}:\radius-1.1cm);

\draw[with arrows] ([shift=({360/10*(7)}:\radius-1.1cm)]5,0) arc ({360/10*(7)}:{360/10*(5)}:\radius-1.1cm);

\draw[with arrows] ([shift=({360/10*(7)}:\radius-1.1cm)]9,0) arc ({360/10*(7)}:{360/10*(5)}:\radius-1.1cm);

\draw[with arrows] ([shift=({360/10*(7)}:\radius-1.1cm)]11,0) arc ({360/10*(7)}:{360/10*(5)}:\radius-1.1cm);

\draw[with arrows]({360/10 * (3)}:\radius+1.5cm) arc ({360/10 *(3)}:{360/10*(4)}:\radius+1.5cm);

\draw [blue] ([shift=({360/10*(10)}:\radius-1.1cm)]-1,0) --  ([shift=({360/10*(5)}:\radius-1.1cm)]1,0);

\draw [blue] ([shift=({360/10*(10)}:\radius-1.1cm)]1,0) -- ([shift=({360/10*(10)}:\radius+1.1cm)]0,0);

\draw [blue] ([shift=({360/10*(10)}:\radius-1.1cm)]3,0) -- ([shift=({360/10*(10)}:\radius+1.1cm)]2,0);

\draw [blue, dashed] ([shift=({360/10*(10)}:\radius-1.1cm)]5,0) -- ([shift=({360/10*(10)}:\radius+1.1cm)]3.5,0);

\draw [blue, dashed] ([shift=({360/10*(10)}:\radius-1.1cm)]7.5,0) --  ([shift=({360/10*(5)}:\radius-1.1cm)]9,0);

\draw [blue, dashed] ([shift=({360/10*(10)}:\radius-1.1cm)]-2.5,0) --  ([shift=({360/10*(5)}:\radius-1.1cm)]-1,0);

\draw [blue] ([shift=({360/10*(10)}:\radius-1.1cm)]9,0) -- ([shift=({360/10*(10)}:\radius+1.1cm)]8,0);

\draw [blue] ([shift=({360/10*(10)}:\radius-1.1cm)]11,0) -- ([shift=({360/10*(10)}:\radius+1.1cm)]10,0);

\draw [red] ([shift=({360/10*(7.5)}:\radius-1.1cm)]-1,0) to[in=210, out=330] ([shift=({360/10*(7)}:\radius-1.1cm)]5,0);

\draw [red] ([shift=({360/10*(2.5)}:\radius-1.1cm)]-1,0) to[in=150, out=30] ([shift=({360/10*(3)}:\radius-1.1cm)]5,0);

\draw [red] ([shift=({360/10*(1.5)}:\radius-1.1cm)]-1,0) to[in=150, out=30] ([shift=({360/10*(4)}:\radius-1.1cm)]3,0);

\draw [red] ([shift=({360/10*(1.5)}:\radius-1.1cm)]9,0) to[in=90, out=40] ([shift=({360/10*(5)}:\radius-1.1cm)]14,0);
\draw [red] ([shift=({360/10*(7.5)}:\radius-1.1cm)]9,0) to[in=270, out=320] ([shift=({360/10*(5)}:\radius-1.1cm)]14,0);

\node[blue] at (0,0.3) {\small $L_{i-1}$};
\node[red] at (1,1.4) {\small $U$};
\node[red] at (1,-1.4) {\small $D$};

\node[blue] at (2,0.3) {\small $L_{i}$};

\node[red] at (2.5,0.7) {\small $\tilde{U}$};

\node[blue] at (4,0.3) {\small $L_{i+1}$};
\node[blue]at (10,0.3) {\small $L_{k}$};
\node[blue]at (12,0.3) {\small $L_{k+1}$};

\node[red]at (11,1.3) {\small $E$};

\node[vertex] at  ([shift=({360/10*(2.5)}:\radius+1.5cm)]5,0) {};

\end{scope}

\end{tikzpicture}
    \caption{Auxiliary Lagrangians}
    \label{auxlag}
\end{figure}
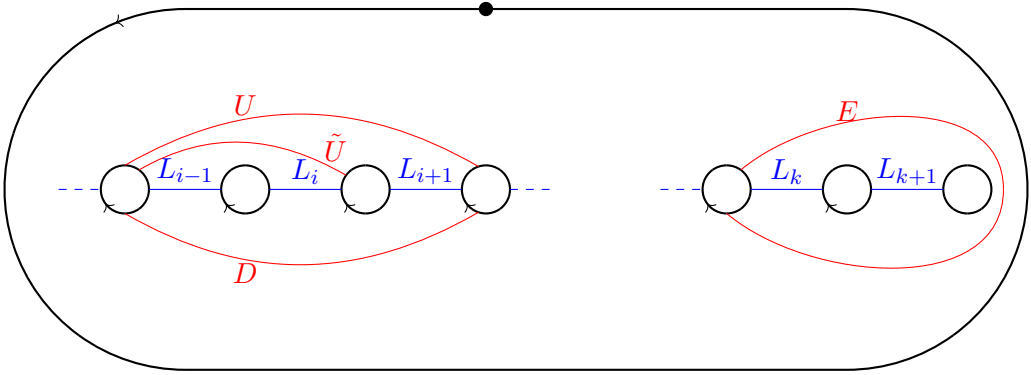

\section{B-side: case $n=1$}

\subsection{Compactified LG model}

The variety $\ov{\YY}_{1,k}$ is covered by $k+1$ affine charts $\UU_0,\ldots,\UU_k$, each isomorphic to the affine space,
where the coordinates on each chart are given as follows
\begin{equation}
\begin{array}{l}\nonumber
\UU_0: V_k,\ldots,V_2,V_1,X_2^{-1},\\
\UU_1: V_k,\ldots,V_2,C_1^{-1},X_2,\\
\cdots\\
\UU_{k-1}: V_k,C_{k-1}^{-1},C_{k-2},V_{k-2},\ldots,V_1,\\
\UU_k: X_1,C_{k-1},V_{k-1},\ldots,V_1.
\end{array}
\end{equation}
where we abbreviate $C_{1i}$ as $C_i$.
We have 
$$V_k=X_1C_{k-1}, \ V_{k-1}=C_{k-1}^{-1}C_{k-2}, \ \ldots, \ V_1=C_1^{-1}X_2,$$
and we consider $\ov{\XX}_{1,k}=\ov{\YY}_{1,k}\times \A^k$ with the potential
$$\bw_{1,k}=U_1V_1+\ldots+U_kV_k.$$
It is convenient to set $C_k=X_1^{-1}$, $C_0=X_2$.

The subvariety $\ov{\ZZ}_{1,k}$ given as the vanishing locus of $(V_1,\ldots,V_k)$ is the nodal chain of rational curves
$\ov{R}_0\cup R_1\cup\ldots R_{k-1}\cup R_k$,
where $R_k=\A^1$ with coordinate $X_1$; $R_i=\P^1$ for $0<i<k$ with $C_i$ and $C_i^{-1}$ being the coordinates on its affine charts; $\ov{R}_0=\P^1$ with $X_2$ and $X_2^{-1}$ being the coordinates on its affine charts. Note that the intersection point
$R_i\cap R_{i-1}$ lies on the nodal union of the affine lines $\A^1_{C_i^{-1}}\sub R_i$ and $\A^1_{C_{i-1}}\sub R_{i-1}$.

The critical locus $\crit(\bw_{1,k})\sub\ov{\XX}_{1,k}$ is the union of our nodal chain $\ov{\ZZ}_{1,k}$ with $k$ affine lines $L_1,\ldots,L_k$, such that $L_i$ intersects the nodal chain at
the $i$th node.

\begin{lem}\label{proj-P1-lem}
(i) For every $i=0,\ldots,k-1$, there is a morphism $p_i:\ov{\YY}_{1,k}\to \P^1$ such that
$p_i^*z=C_i$ for $i>0$ and $p_0^*z=X_2$, where $z$ is a coordinate on $\A^1\sub\P^1$.
Furthermore, the induced map
$$f:\ov{\YY}_{1,k}\rTo{X_1,V_1,\ldots,V_k,p_0,\ldots,p_{k-1}} \A^{k+1}_{X_1,V_1,\ldots,V_k}\times (\P^1)^k$$
is a closed embedding. 

\noindent
(ii) For $i=1,\ldots,k-1$, we have sections of $p_i^*\OO(-1)$, defined in an open neighborhood of $R_i$,
\begin{equation}\label{sections-of-O(-1)-eq}
s_1(i)=\frac{C_{i+1}^{-1}}{p_i^*z_0}=\frac{V_{i+1}}{p_i^*z_1}, \ \ s_2(i)=\frac{C_{i-1}}{p_i^*z_1}=\frac{V_i}{p_i^*z_0}
\end{equation}
(where $(z_0:z_1)$ are homogeneous coordinates on $\P^1$ such that $z=z_1/z_0$),
such that $R_i$ is the vanishing locus of 
$$V_1,\ldots,V_{i-1},s_1(i),s_2(i),V_{i+2},\ldots,V_k.$$
Similarly, we have a section of $p_0^*\OO(-1)$, defined in a neighborhood of $\ov{R}_0$
$$s_1(0)=\frac{C_{1}^{-1}}{p_0^*z_0}=\frac{V_1}{p_0^*z_1},$$
such that $\ov{R}_k$ is the vanishing locus $s_1(0),V_2,\ldots,V_k$. Finally, locally near $R_k$, it is given as the vanishing locus
of $V_1,\ldots,V_{k-1},C_{k-1}$. 
\end{lem}

\begin{proof}
(i) 
Let $\A_0=(z_0\neq 0)$ and $\A_1=(z_1\neq 0)$ be the standard open covering of $\P^1$ (where $(z_0:z_1)$ are
the homogeneous coordinates such that $z=z_1/z_0$). Then we define $p_i$ as the unique map
such that
$$p_i^{-1}(\A_0)=\UU_{i+1}\cup\ldots\cup \UU_k, \ \ p_i^{-1}(\A_1)=\UU_{0}\cup\ldots\cup\UU_{i},$$
with $p_i^*z=C_i$ on $\UU_{i+1}\cup\ldots\cup\UU_k$ and $p_i^*z^{-1}=C_i^{-1}$ on $\UU_0\cup\ldots\cup\UU_{i}$.
The fact that $C_i$ (resp., $C_i^{-1}$) is indeed well defined on $\UU_j$ with $j>i$ (resp., $j\le i$) 
follows from the identity $C_j=C_{j+1}V_{j+1}$.

For a sequence $(i_1,\ldots,i_k)$ of $0$'s and $1$'s let us set 
$$\A_{i_1,\ldots,i_k}:=\A_{i_1}\times\ldots\times \A_{i_k}\sub (\P^1)^k.$$
To prove the second assertion we note that $f^{-1}(\A^{k+1}\times \A_{i_1,\ldots,i_k})$ is nonempty
only when $(i_1,\ldots,i_k)=((0)^{i},(1)^{k-i})$ for some $0\le i\le k$, and in this case
$$f^{-1}(\A^{k+1}\times \A_{(0)^{i},(1)^{k-i}})=\UU_i.$$
It is easy to check that the induced map
$$\UU_i\rTo{X_1,V_1,\ldots,V_k,C_0,\ldots,C_{i-1},C_i^{-1},\ldots,C_{k-1}} \A^{k+1}\times \A_{(0)^{i},(1)^{k-i}}$$
is a closed embedding.

\noindent
(ii) The fact that $s_1(i)$ (resp., $s_2(i)$) is well defined follows from the identity $C_{i+1}^{-1}C_i=V_{i+1}$
(resp., $C_i^{-1}C_{i-1}=V_{i}$). The fact that $R_i$ coincides with the claimed vanishing locus can be checked locally.
\end{proof}

Let us set $\ov{R}_i=R_i$ for $i>0$, and
consider the objects $(P_i)_{i=0,\ldots,k}$ of the multigraded category $\ov{\BB}_{1,k}$ of $\wt{T}$-equivariant
matrix factorization of $\bw_{1,k}$ on $\ov{\XX}_{1,k}$, associated with the objects $\OO_{\ov{R}_i\times \A^k}$ of the corresponding singularity category. Note that the character group $\wt{H}$ of the torus $\wt{T}$
is a free abelian group with the basis $x_1,u_1,\ldots,u_k,l_0$. Hence,
the grading group $\bbL_B=2\wt{H}+\Z l_0$ is a free abelian group with the basis $2x_1,2u_1,\ldots,2u_k,l_0$, so that
the grading of $X_1$ is $2x_1$ and the grading of $U_i$ is $2u_i$.
We have $v_i=l_0-u_i$, so the relations $v_i=-c_i+c_{i-1}$ allow to compute the grading $c_i$ of $C_i$ (recall that $c_k=-x_1$).

\begin{prop} The only nonzero morphisms between $(P_i)$ in $\ov{\BB}_{1,k}$ are the following:
\begin{align*}
&\End(P_k)=k[X_1,U_k]/(X_1U_k), \ \End(P_{k-1})=k[U_k,U_{k-1}]/(U_kU_{k-1}), \ \ldots, \ 
	\End(P_{1})=k[U_{2},U_1]/(U_{2}U_1), \\
&\End(P_0)=k[U_1],
\end{align*}
$$\Hom(P_i,P_{i-1})=k[U_i]\cdot a_i, \ \ \Hom(P_{i-1},P_i)=k[U_i]\cdot b_i, \ \text{ for } i=1,\ldots,k,$$
where 
$$\deg(a_i)=l_0-2c_{i-1}, \ \ \deg(b_i)=l_0+2c_i,$$
$$a_ib_i=b_ia_i=U_i.$$
\end{prop}

The proof uses presentations of $\ov{R}_i$ as complete intersections and Lemma \ref{Koszul-mf-lem}. We will give the
details of similar calculations in the case $n=2$. 

There is a unique projection $|\cdot|:\bbL_B\to \Z$ sending $l_0$ to $1$ and $x_1$ and $u_i$ to $0$.
In the obtained $\Z$-graded category the $\Z$-gradings are given by 
$$|X_1|=|U_1|=\ldots=|U_k|=0.$$
With this choice of $\Z$-grading we have $|a_i|=1-2(k+1-j)$, and
the endomorphism algebra 
$$\End(P_k\oplus P_{k-1}[-1]\oplus P_{k-2}[-4] \oplus\ldots \oplus P_0[-k^2])$$
will be concentrated in degree $0$.

By Lemma \ref{formality-lem}, this implies formality of the $\bbL_B$-graded dg-endomorphism algebra of $\bigoplus P_i$.

The following simple general result will help us to prove generation.

\begin{lem}\label{components-generation-lem}
Suppose a scheme $X$ is the union of two closed subschemes: $X=Y\cup Z$. Set 
$$U=X\setminus Y=Z\setminus Y\cap Z.$$
Assume we have a collection of
coherent sheaves $(F_i)_{i\in I}$ on $Y$, generating $D^b\Coh(Y)$ and a collection of coherent sheaves 
$(G_j)_{j\in J}$ on $X$ such that $(G_j|_U)_{j\in J}$ generate $D^b\Coh(U)$. Then $(F_i)_{i\in I}$ and $(G_j)_{j\in J}$
together generate $D^b\Coh(X)$. The similar assertion holds for categories of equivariant sheaves.
\end{lem}

\begin{proof}
Indeed, in this case $(F_i)$ generate the subcategory in $D^b\Coh(X)$ consisting of objects with cohomology supported
on $Y$. But  latter subcategory is precisely the kernel of
the restriction functor
$$D^b\Coh(X)\to D^b\Coh(U).$$
The assertion follows from this.
\end{proof}

\begin{lem}\label{generation-lem-k=1} 
Let us work over a regular base ring $\k$. For any subgroup $\G_m\sub T$, the twists by $\chi^n$ of
the equivariant sheaves $\OO_{\ov{R}_0},\ldots,\OO_{R_k}$ (resp., $\OO_{R_1},\ldots,\OO_{R_k},\OO_{\ov{\ZZ}_{1,k}}$)
generate $D^b\Coh(\ov{\ZZ}_{1,k}/\G_m)$.
\end{lem}

\begin{proof} This can be proved by induction on $k$. By the induction assumption, the objects $\OO_{R_1},\ldots,\OO_{R_k}$ and their twists generate
the subcategory of equivariant sheaves supported on $R_1\cup\ldots\cup R_k$. 
Hence, by Lemma \ref{components-generation-lem},
it remains to check that the twists of the restriction of $\OO_{R_0}$ (resp., $\OO_{\ov{\ZZ}_{1,k}}$) to 
$\ov{R}_0\setminus R_{1}$ generate $D^b\Coh(\ov{R}_0\setminus R_{1})$. But this restriction is
 the structure sheaf, and $\ov{R}_0\setminus R_{1}$ is the affine line, so this is true.
\end{proof}

\begin{prop}\label{n1-generation-prop} 
The objects $P_0,\ldots,P_k$ generate the category $\ov{\BB}_{1,k}$.
\end{prop}

\begin{proof} 
By Lemma \ref{generation-lem-k=1} with the base ring being the polynomial ring in $U_1,\ldots,U_k$, we get that
the twists by $\chi^n$ of the equivariant coherent sheaves $(\OO_{\ov{R}_i\times\A^k})$ generate $D^b\Coh (\ov{\ZZ}_{1,k}/\G_m\times \A^k)$.
        It is easy to see that the critical locus $\crit(\bw_{1,k})$ is contained in $\ov{\ZZ}_{1,k}\times \A^k$, so the assertion follows from
Lemma \ref{L-graded-generation-lem}. 
\end{proof}

\subsection{Localization for $n=1$}\label{n1-loc-sec}

The complement $\ov{\ZZ}_{1,k}\setminus \ZZ_{1,k}$ consists of one point $p$, namely the infinite point of $\ov{R}_0\simeq \P^1$.
Furthermore, the critical locus of $\bw_{1,k}$ intersects the complement $\ov{\XX}_{1,k}\setminus \XX_{1,k}$ exactly at the point $p$.

Let us denote by $\BB_{1,k}^\Z$ (resp., $\ov{\BB}_{1,k}^\Z$) 
the category of $\Z$-graded matrix factorizations on $\bw_{1,k}$ on $\XX_{1,k}$ (resp., $\ov{\XX}_{1,k}$),
corresponding to the grading 
$$|U_i|=2, \ |X_i|=|C_i|=|V_i|=0.$$
By Isik's theorem (see \eqref{Isik-equivalence}), we have equivalences 
$$\BB_{1,k}^\Z\simeq D^b\Coh(\ZZ_{1,k}), \ \ \ov{\BB}_{1,k}^\Z\simeq D^b\Coh(\ov{\ZZ}_{1,k}),$$
compatible with the obvious restriction functors.

It is well known that the restriction functor induces an equivalence of $D^b\Coh(\ZZ_{1,k})$ with the quotient of
$D^b\Coh(\ov{\ZZ}_{1,k})$ by the subcategory of sheaves supported at $\ov{\ZZ}_{1,k}\setminus \ZZ_{1,k}=\{p\}$.
Since the restriction functor has an obvious dg-enhancement, this is also true at the level of dg-categories.
Hence, we can identify $\BB_{1,k}^\Z\simeq D^b\Coh(\ZZ_{1,k})$ with the quotient of
$\ov{\BB}_{1,k}^\Z$ by the subcategory of matrix factorizations supported at $p$.
The latter category is generated by the matrix factorization $E$ associated with the object $\OO_{p\times \A^k}$ of the singularity category.

We have the following resolution of $E$ in terms of the generating objects $P_i$.
For an interval $[i,j]$, where $j<k$ and for integers $m_i,\ldots,m_j$, let us consider
the object 
$$\OO_{R_{[i,j]}}(m_i,\ldots,m_j)=\bigotimes_{r=i}^j p_r^*\OO(m_r)|_{R_{[i,j]}},$$
where $R_{[i,j]}=\ov{R}_i\cup\ldots\cup R_j$.
Similarly, for $i<k$ we set
$$\OO_{R_{[i,k]}}(m_i,\ldots,m_{k-1})=\bigotimes_{r=i}^{k-1} p_r^*\OO(m_r)|_{R_{[i,k]}}.$$

\begin{lem}\label{n1-resolution-lem}
One has the following exact triangles
$$\OO_{R_{[0,i]}}(1,\ldots,1,0)\rTo{p_i^*z_1} \OO_{R_{[0,i+1]}}(1,\ldots,1,0)\to \OO_{R_{i+1}}\to\ldots, \ \text{ for }i<k-1,$$
$$\OO_{R_{[0,k-1]}}(1,\ldots,1,0)\rTo{p_{k-1}^*z_1} \OO_{R_{[0,k]}}(1,1,\ldots,1)\to \OO_{R_k}\to\ldots,$$
$$\OO_{R_{i+1}}\rTo{p_{i+1}^*z_0} \OO_{R_{[0,i+1]}}(0,1,\ldots,1)\to \OO_{R_{[0,i]}}(0,1,\ldots,1)\to\ldots, \text{ for }0<i<k-1,$$
$$\OO_{R_{k}}\rTo{X_1} \OO_{R_{[0,k]}}(0,1,\ldots,1)\to \OO_{R_{[0,k-1]}}(0,1,\ldots,1)\to\ldots,$$
$$\OO_{R_{1}}\rTo{p_1^*z_0} \OO_{R_{[0,1]}}(0,1)\to \OO_{R_0}\to\ldots,$$
which express $\OO_{R_{[0,k]}}(1,\ldots,1)$ and $\OO_{R_{[0,k]}}(0,1,\ldots,1)$ in terms of the generators.
Finally, there is an exact triangle
$$\OO_{R_{[0,k]}}(0,1,\ldots,1)\rTo{X_2^{-1}} \OO_{R_{[0,k]}}(1,\ldots,1)\to \OO_p\to \ldots$$
\end{lem}

\begin{proof} These exact triangles come from the natural exact sequences of coherent sheaves extending the surjections of the form $L|_C\to L|_{C'}$,
where $L$ is a line bundle, $C$ is a union of components of the nodal curve, $C'$ is either a smaller union of components, or the point $p$.
\end{proof}

\subsection{Toric maps $\YY_{n,k}\to \YY_{n,1}$}

Here we will prove that $\YY_{n,k}$ is projective over $\A^{n+k+1}$, and hence it is a toric variety over $\k$.
We reduce this to the case $n=1$ using the following observation.

\begin{prop}\label{toric-closed-emb-prop}
For every $a=1,\ldots,n$, there is a morphism 
$$f_a:\YY_{n,k}\to \YY_{1,k}$$
such that $f_a^{-1}\UU_i=\cup_{\phi:\phi(a)=i}\UU_\phi$, and
$$f_a^*V_i=V_i, \ f_a^*C_i=C_{ai}, \ f_a^*X_1=X_1\ldots X_a,\ f_a^*X_2=X_{a+1}\ldots X_{n+1}.$$
Furthermore, the morphism
$$\YY_{n,k}\rTo{(f_1,\ldots,f_n),X_2,\ldots,X_n} 
(\YY_{1,k}\times_{A^k_{V_1,\ldots,V_k}}\ldots\times_{\A^k_{V_1,\ldots,V_k}}\YY_{1,k})\times \A^{n-1}_{X_2,\ldots,X_n}$$
is a closed embedding.
\end{prop}

\begin{proof} The proof is a straightforward calculation in the affine charts.
\end{proof}

\begin{cor}\label{toric-proj-emb-cor}
The morphism
$$\YY_{n,k}\rTo{V_1,\ldots,V_k,X_1,\ldots,X_{n+1}}\A^{n+k+1}$$
is projective.
\end{cor}

\begin{proof} Combining Proposition \ref{toric-closed-emb-prop} with Lemma \ref{proj-P1-lem} we see that the composition
$$\YY_{n,k}\to \A^k_{V_1,\ldots,V_k}\times\A^{n+1}_{X_1,\ldots,X_{n+1}}\rTo{\id\times\phi} \A^k_{V_1,\ldots,V_k}\times\A^{3n-1}$$
is projective, where 
$$\phi(X_1,\ldots,X_{n+1})=(X_1,X_2\ldots X_{n+1};X_1X_2,X_3\ldots X_{n+1};\ldots;X_1\ldots X_n,X_{n+1};X_2,\ldots,X_{n-1}).$$
Since $\phi$ is a closed embedding, the morphism $\YY_{n,k}\to \A^{k+n+1}$ is also projective.
\end{proof}

\section{B-side: case $n=2$}

\subsection{Varieties $\ZZ_{2,k}$ and $\ov{\ZZ}_{2,k}$}

The affine charts $\UU_{ij}$ covering $\YY_{2,k}$ are numbered by pairs $i\ge j$ in $[1,k]$, where $i=\phi(1)$, $j=\phi(2)$.
For $i>j$ the coordinates on $\UU_{ij}$ are
$$V_k,\ldots,V_{i+1},C_{1,i}^{-1},C_{1,i-1},V_{i-1},\ldots,V_{j+1},C_{2,j}^{-1},C_{2,j-1},V_{j-1},\ldots,V_1,$$
while for $i=j$, the coordinates on $\UU_{ii}$ are
$$V_k,\ldots,V_{i+1},C_{1,i}^{-1},X_2,C_{2,i-1},V_{i-1},\ldots,V_0$$
(recall that $C_{1k}^{-1}=X_1$, $C_{2,0}=X_3$).

The variety $\ZZ_{2,k}$ (given as the vanishing locus of $V_1,\ldots,V_k$)
is $2$-dimensional and is a union of smooth irreducible components $R_{ij}$, with $k\ge i\ge j\ge 0$.
Here for $i<k$, $R_{ii}$ is the closure of the locus $C_{1,i}^{-1}=0$ in $\UU_{ii}\cap \ZZ_{2,k}$,
while $R_{kk}$ is the locus $C_{2,k-1}=0$ in $\UU_{kk}\cap\ZZ_{2,k}$.
For $k>i>j$, $R_{ij}$ is the closure of the locus $C_{1,i}^{-1}=C_{2,j}^{-1}=0$ in $\UU_{ij}\cap \ZZ_{2,k}$,
while for $j<k-1$, $R_{kj}$ is the closure of the locus $C_{1,k-1}=C_{2,j}^{-1}=0$ in $\UU_{kj}\cap \ZZ_{2,k}$,
and $R_{k,k-1}$ is the closure of the locus $X_2=0$ in $\UU_{kk}\cap \ZZ_{2,k}$.

Note that $R_{kk}$ is the affine plane with coordinates $X_1,X_2$; $R_{k0}$ is the affine plane with coordinates $X_1,X_3$; and
$R_{00}$ is the affine plane with coordinates $X_2,X_3$. The components $R_{kj}$ are isomorphic to $\P^1\times \A^1_{X_1}$, the components
$R_{i0}$ are isomorphic to $\P^1\times \A^1_{X_3}$,
the components $R_{ii}$ are isomorphic to the blow up of the point $(0,0)$ in $\P^1\times \A^1_{X_2}$, and the remaining components are
isomorphic to $\P^1\times\P^1$.
In the case $k=4$, the variety $\ZZ_{2,4}$ is schematically depicted in Figure \ref{toricdim2}, with
$R_{44}$ being the component at the top corner, $R_{40}$ at the bottom left corner, and $R_{00}$ at the bottom right corner.

In the compactification $\ov{\YY}_{2,k}$ we have extra charts $\UU_{i,0}$, where $i=0,\ldots,k$.
For $i>1$ the coordinates on $\UU_{i,0}$ are
$$V_k,\ldots,V_{i+1},C_{1,i}^{-1},C_{1,i-1},V_{i-1},\ldots,V_1,X_3^{-1},$$
while the coordinates on $\UU_{1,0}$ are
$$V_k,\ldots,V_{2},C_{1,1}^{-1},X_2X_3,X_3^{-1},$$
and the coordinates on $\UU_{0,0}$ are
$$V_k,\ldots,V_1,X_2^{-1}X_3^{-1},X_2.$$

The compactified variety $\ov{\ZZ}_{2,k}$ is the union of irreducible components $\ov{R}_{ij}$, $k\ge i\ge j\ge 0$,
where $\ov{R}_{ij}=R_{ij}$ for $j>0$, while $R_{i0}$ get compactified as follows.
The affine plane $R_{k0}$ gets compactified to $\P^1\times \A^1_{X_1}$.
For $0<i<k$, the component $R_{i0}=\P^1\times \A^1_{X_3}$ gets compactified to $\P^1\times\P^1$.
The component $R_{00}=\A^2_{X_2,X_3}$ gets compactified to the blow up of a point in $\P^1\times \A^1_{X_2}$.

Thus, the complement of $\ZZ_{2,k}$ in $\ov{\ZZ}_{2,k}$ is the nodal chain 
$$D=D'_0\cup D_0\cup D_1\cup\ldots\cup D_k,$$
where $D'_0=D_{k}=\A^1$, $D_i=\P^1$ for $i=0,\ldots,k-1$; $D_i=\ov{R}_{i0}\setminus R_{i0}$ for $i>0$,
$D'_0\cup D_0=\ov{R}_{00}\setminus R_{00}$.

We have the following extension of Proposition \ref{toric-closed-emb-prop} with a straightforward proof.

\begin{lem}
(i) There are two natural morphisms 
$$f_1:\ov{\YY}_{2,k}\to \YY_{1,k}, \ \ f_2:\ov{\YY}_{2,k}\to \ov{\YY}_{1,k}$$
such that $f_1^{-1}\UU_i=\cup_j \UU_{ij}$, $f_2^{-1}\UU_j=\cup_i \UU_{ij}$, and
$$f_1^*X_1=X_1, \ f_1^*X_2=X_2X_3, \ f_1^*C_i=C_{1i}, \ f_1^*V_i=V_i.$$
$$f_2^*X_1=X_1X_2, \ f_2^*X_2=X_3, \ f_2^*C_i=C_{2i}, \ f_2^*V_i=V_i,$$
Furthermore, the morphism
$$\ov{\YY}_{2,k} \rTo{(f_1,f_2),X_2} (\YY_{1,k}\times_{\A^{k}_{V_1,\ldots,V_k}} \ov{\YY}_{1,k})\times \A^1_{X_2}$$
is a closed embedding.

\noindent
(ii) The schematic preimage $f_2^{-1}(\ov{R}_j)$ is equal to the union (with the reduced scheme structure)
$$\ov{R}_{*j}:=\cup_i \ov{R}_{ij}$$
(these are ``rows" in Figure \ref{toricdim2}). 
\end{lem}


In the following result we determine equations that present $\ov{R}_{ij}$ as a complete intersection in its open neighborhood in $\ov{\YY}_{2,k}$.
Note that for $\ov{R}_{*j}=f_2^{-1}(\ov{R}_j)$ we have such a description due to Lemma \ref{proj-P1-lem}.

\begin{lem}\label{k=2-complete-int-lem}
(i) For $k>i>j>0$, such that $i-j>1$, $R_{ij}$ is the vanishing locus of 
$$V_k,\ldots,V_{i+2},f_1^*s_1(i),f_1^*s_2(i),V_{i-1},\ldots,V_{j+2},f_2^*s_1(j),f_2^*s_2(j),V_{j-1},\ldots,V_1,$$
where $s_1(i)$ and $s_2(i)$ are defined by \eqref{sections-of-O(-1)-eq}.

\noindent
(ii) For $1\le i<k$, there exists a section $s_i$ of
$f_2^*p_{i-1}^*\OO(-1)\ot f_1^*p_{i}^*\OO(-1)$, defined in a neighborhood of $\ov{R}_{i,i-1}$, given by
$$s_i:=\frac{f_1^*s_2(i)}{f_2^*p_{i-1}^*z_1}=\frac{f_2^*s_1(i-1)}{f_1^*p_{i}^*z_0}=
\frac{X_2}{f_2^*p_{i-1}^*z_0\cdot f_1^*p_{i}^*z_1}.$$
For $1<i<k$, $R_{i,i-1}$ is given as the vanishing locus of
$$V_k,\ldots,V_{i+2},f_1^*s_1(i),s_i,f_2^*s_2(i-1),V_{i-2},\ldots,V_1,$$
while 
$\ov{R}_{1,0}$ is the vanishing locus of
$$V_k,\ldots,V_{3},f_1^*s_1(1),s_1.$$ 

On the other hand, there exists a section $s_k$ of $f_2^*p_1^*\OO(-1)$ defined locally near $R_{k,k-1}$, given by
$$s_k=\frac{X_2}{f_2^*p_{k-1}^*z_0}=\frac{C_{1,k-1}}{f_2^*p_{k-1}^*z_1},$$
such that $R_{k,k-1}$ is the vanishing locus of
$$s_k,f_2^*s_2(k-1),V_{k-2},\ldots,V_1.$$

\noindent
(iii) For $0<j<k-1$, $\ov{R}_{kj}$ is cut out in its neighborhood as the vanishing locus of 
$$C_{1,k-1},V_{k-1},\ldots,V_{j+2},f_2^*s_1(j),f_2^*s_2(j),V_{j-1},\ldots,V_1.$$
For $1<i<k$, $\ov{R}_{i0}$ is the vanishing locus of
$$V_k,\ldots,V_{i+2},f_1^*s_1(i),f_1^*s_2(i),V_{i-1},\ldots,V_{2},f_2^*s_1(0).$$

\noindent
(iv) $\ov{R}_{k0}$ is cut out in its neighborhood as the vanishing locus of 
$$C_{1,k-1},V_{k-1},\ldots,V_2,f_2^*s_1(0).$$
\end{lem}

\begin{proof} (i)
Note that for $i>j$, we have
$$\ov{R}_{ij}=f_1^{-1}R_i\cap f_2^{-1}\ov{R_j}.$$
Thus, our assertion follows from Lemma \ref{proj-P1-lem}(ii), where we use the fact that
$$V_{i+1}=s_1(i)\cdot p_i^*z_1, \ V_i=s_2(i)\cdot p_i^*z_0,$$
so that $V_i$ and $V_{i+1}$ vanish on the zero locus of $s_1(i)$ and $s_2(i)$.

\noindent
(ii) The existence of $s_j$ follows from the claimed identities, which themselves follow from
$$f_1^*s_2(j)\cdot f_1^*p_{j}^*z_0=V_j=f_2^*s_1(j-1)\cdot f_2^*p_{j-1}^*z_1,$$
which holds by the definition of $s_1(j-1)$ and $s_2(j)$, and from
$$f_2^*s_1(j-1)\cdot f_2^*p_{j-1}^*z_0\cdot f_1^*p_{j}^*(\frac{z_1}{z_0})=C_{2,j}^{-1}\cdot C_{1,j}=X_2.$$
For the definition of $s_k$ we use the identity
$$C_{1,k-1}\cdot \frac{f_2^*p_{k-1}^*z_0}{f_2^*p_{k-1}^*z_1}=C_{1,k-1}\cdot C_{2,k-1}^{-1}=X_2.$$

It is clear that the claimed $k$ sections vanish on $R_{j,j-1}$. To prove that the subscheme they cut out is exactly
$R_{j,j-1}$, we recall that
by Lemma \ref{proj-P1-lem}(ii), we know that for $1<j<k$, $\ov{R}_{j-1,j}$ is
the vanishing locus of
$$V_1,\ldots,V_{j+2},f_1^*s_1(j), \ f_1^*s_2(j), \ f_2^*s_1(j-1), \ f_2^*s_2(j-1),V_{j-2},\ldots,V_k.$$
It remains to observe that $f_1^*s_2(j)$ and $f_2^*s_1(j-1)$ both vanish on the zero locus of $s_j$.

In the case of $R_{k,k-1}$, from Lemma \ref{proj-P1-lem}(ii) we get the description of $R_{k,k-1}$ as
the vanishing locus of
$$C_{1,k-1},f_2^*s_1(k-1),f_2^*s_2(k-1),V_{k-2},\ldots,V_1.$$
The fact that $C_{1,k-1}$ and $f_2^*s_1(k-1)$ vanish on the zero locus of $s_k$ follows from the factorizations 
$$C_{1,k-1}=s_k\cdot f_2^*p_{k-1}^*z_1, \ \ 
f_2^*s_1(k-1)=X_1\cdot s_k.$$

\noindent
(iii), (iv) The proofs are analogous, so we omit them. 
\end{proof}

\subsection{Generating matrix factorizations}

For each $k\ge i>j\ge 0$, let $P_{ij}$ denote the matrix factorization of $\bw_{2,k}$ on $\ov{\XX}_{2,k}$ corresponding to
the object (denoted in the same way), 
$$P_{ij}:=\OO_{\ov{R}_{ij}\times\A^k}\ot f_1^*p_i^*\OO(-1)$$ 
of the singularity category, where our convention is that $p_k^*\OO(-1)$ is trivial.
Here we equip $f_1^*p_i^*\OO(-1)$ with the $T$-equivariant structure by letting $f_1^*p_i^*z_0$ to have weight $0$, 
and $f_1^*p_i^*z_1$ to have weight $c_{1i}$.


We also 
define $Q_j$ to be the matrix factorization corresponding to the object of the singularity category
$$Q_j:=\OO_{\ov{R}_{*j}\times\A^k}.$$

\begin{lem}\label{k2-generation-lem}
Let us work over a regular base ring.
Then for any subgroup $\G_m\sub T$, the twists by $\chi^n$ of the equivariant sheaves
$$(\OO_{\ov{R}_{ij}}\ot f_1^*p_i^*\OO(-1))_{0\ge i>j\ge k} \ \text{ and }\ (\OO_{\ov{R}_{*j}})_{0\le j\le k}$$ 
generate $D^b\Coh(\ov{\ZZ}_{2,k}/\G_m)$.
\end{lem}

\begin{proof}
Let  us prove by descending induction on $j_0\ge 0$ that the twists of the sheaves
$(\OO_{\ov{R}_{ij}}\ot f_1^*p_i^*\OO(-1))_{j\ge j_0}$, $(\OO_{\ov{R}_{*j}})_{j\ge j_0}$ generate
the subcategory of equivariant sheaves supported on 
$$\ov{R}_{*\ge j_0}:=\ov{R}_{*k}\cup\ldots\cup \ov{R}_{*,j_0}.$$
The base $j_0=k$ is clear as $\ov{R}_{*k}=R_{kk}\simeq \A^2$, so the twists of $\OO_{R_{kk}}$ generates the needed subcategory. 
Assume that $j_0<k$ and the assertion holds for $j_0+1$. 
By Lemma \ref{components-generation-lem}, it is enough to check generation after restricting to 
$\ov{R}_{*j_0}\setminus\ov{R}_{*,j_0+1}$.
Now we use the further restriction
$$D^b\Coh(\ov{R}_{*j_0}\setminus\ov{R}_{*,j_0+1})\to D^b\Coh(\ov{R}_{j_0j_0}\setminus (\ov{R}_{j_0+1,j_0+1}\cup\ov{R}_{j_0+1,j_0})).$$
We note that
$$\ov{R}_{j_0j_0}\setminus (\ov{R}_{j_0+1,j_0+1}\cup\ov{R}_{j_0+1,j_0})\simeq \A^2,$$
so its derived category is generated by the twists of the structure sheaf obtained as the restriction of $\OO_{\ov{R}_{*k_0}}$.
Hence, it remains to show that the subcategory of 
$D^b\Coh(\ov{R}_{*j_0}\setminus\ov{R}_{*,j_0+1})$
consisting of sheaves supported on $(\ov{R}_{kj_0}\cup\ldots\ov{R}_{j_0+1,j_0})\setminus\ov{R}_{*,j_0+1}$ 
is generated by the restrictions of our objects.
But we have
$$(\ov{R}_{kj_0}\cup\ldots\ov{R}_{j_0+1,j_0})\setminus\ov{R}_{*,j_0+1}\simeq (R_k\cup\ldots\cup R_{j_0+1})\times \A^1,$$
and our assertion follows from Lemma \ref{generation-lem-k=1} (after tensoring with the generators with the line bundle 
$p_{k-1}^*\OO(-1)\ot\ldots\ot p_{j_0+1}^*\OO(-1)$).
\end{proof}

\begin{prop}\label{n2-generation-prop}
The objects $(P_{ij})_{i>j}$ and $(Q_j)$ generate $\MF(\ov{\XX}_{2,k},\bw_{2,k})$.
\end{prop}

\begin{proof} It is easy to see that the critical locus $\crit(\bw_{2,k})$ is contained in $\ov{\ZZ}_{2,k}\times\A^k$.
Thus, as in Proposition \ref{n1-generation-prop}, the statement follows from Lemma \ref{k2-generation-lem} and Lemma \ref{L-graded-generation-lem}. 
\end{proof}

\subsection{Computation of morphisms}\label{n=2-morphisms-sec}

For the computation of morphisms, we consider first the $\Z$-graded category
$\MF_{\G_m}(\ov{\XX}_{2,k},\bw_{2,k})$, where the subgroup $\G_m\sub \wt{T}$ corresponds
to the homomorphism $\wt{H}\to \Z$ sending $x_1$,$x_2$,$u_1,\ldots,u_k$ to zero and $l_0$ to $1$.  
This corresponds to the following $\Z$-grading of the variables:
\begin{equation}\label{comp-grading}
|X_1|=|X_2|=|U_i|=0, \ \ |V_i|=2, \ \ |C_{1i}|=|C_{2i}|=2(k-i), \ \ |X_3|=2k.
\end{equation}
We will consider the case of an arbitrary grading later in Sec.\ \ref{Recovering-multigrading}.

\subsubsection{Endomorphisms}\label{End-sec}

Let us consider the critical locus
$\crit(\bw_{2,k})\sub \ov{\XX}_{2,k}$.
It is easy to see that its projection to $\ov{\YY}_{2,k}$ factors through $\ov{\ZZ}_{2,k}$,
so we have a natural projection
$$p:\crit(\bw_{2,k})\to \ov{\ZZ}_{2,k}.$$

Recall that $\ov{\ZZ}_{2,k}$ is the union of the irreducible components $\ov{R}_{ij}$.
Let us introduce the notation for some other strata in $\ov{\ZZ}_{2,k}$.
For $k\ge i\ge j\ge 1$, set 
$$L^h_{ij}:=R_{i,j-1}\cap R_{i,j},$$
Similarly, for $k\ge i\ge j\ge 1$, we set
$$L^v_{ij}:=\ov{R}_{i,j-1}\cap \ov{R}_{i-1,j-1}.$$
Also, for $i=1,\ldots,k$, set
$$L^d_i:=R_{i-1,i-1}\cap R_{ii}.$$
Finally, for $k\ge i>j\ge 1$, we set $p_{ij}:=L^h_{ij}\cap L^v_{ij}$.
A straightforward calculation gives the following description of the critical locus of $\bw_{2,k}$.

\begin{lem} 
The locus $\crit(\bw_{2,k})$ is the union (with the reduced scheme structure) of the following smooth two dimensional
components:
\begin{itemize}
\item $\ov{R}_{ij}$ for $k\ge i\ge j\ge 0$;
\item $L^h_{ij}\times \A^1_{U_j}$ for $k\ge i\ge j\ge 1$;
\item $L^v_{ij}\times \A^1_{U_i}$ for $k\ge i\ge j\ge 1$;
\item $L^d_i\times \A^1_{U_i}$ for $1\le i\le k$;
\item $p_{ij}\times \A^2_{U_i,U_j}$ for $k\ge i>j\ge 1$.
\end{itemize}
\end{lem}

The crucial step in calculating the morphism spaces is the following local result.

\begin{lem}\label{local-End-lem}
For every $k\ge i>j\ge 0$, one has a natural quasiisomorphism
$$\und{\End}(P_{ij})\simeq \OO_{p^{-1}\ov{R}_{ij}}.$$
Let us set $L_{*j}=\cup_i L^h_{ij}\cup L^d_j$.
Then for $0<j<k$ we have
$$\und{\End}(Q_j)=\OO_{\ov{R}_{*j}\cup L_{*j}\times \A^1_{U_j}\cup L_{*j+1}\times \A^1_{U_{j+1}}},$$
while
$$\und{\End}(Q_k)=\OO_{p^{-1}R_{00}}=\OO_{R_{00}\cup L_{*k}\times \A^1_{U_k}},$$
$$\und{\End}(Q_0)=\OO_{\ov{R}_{*0}\cup L_{*1}\times \A^1_{U_1}}.$$
\end{lem}

\begin{proof}
We will combine the description of $\ov{R}_{ij}$ as complete intersections from Lemma \ref{k=2-complete-int-lem} with 
the computation of Lemma \ref{Koszul-mf-lem}(i). In the case of $Q_j$ we use the description of 
$\ov{R}_{*j}=f_2^{-1}(\ov{R}_j)$
as a complete intersection. Below we provide details of the calculation (in later calculations we will skip over some
similar steps).

\noindent
{\bf Case of $\und{\End}(P_{ij})$, for $k>i>j>0$.}

Assume first that $i-j>1$. Then we can decompose $W$ as
\begin{align}\label{W-decomp-eq}
&W=\sum_{r\neq i,i+1,j,j+1}V_rU_r+  \nonumber\\
&f_1^*s_1(i)\cdot f_1^*p_i^*z_1\cdot U_{i+1}+f_1^*s_2(i)\cdot f_1^*p_i^*z_0\cdot U_{i}+
f_2^*s_1(j)\cdot f_2^*p_j^*z_1\cdot U_{j+1}+f_2^*s_2(j)\cdot f_2^*p_j^*z_0\cdot U_{j}.
\end{align}
The description of $\ov{R}_{ij}$ as a complete intersection implies that $P_{ij}\ot f_1^*p_i^*\OO(1)$ can be
presented as the tensor product of the following Koszul matrix factorizations:
\begin{itemize}
\item for $r\neq i,i+1,j,j+1$, $\{U_r,V_r\}$;
\item $\{f_1^*p_i^*z_1\cdot U_{i+1},f_1^*s_1(i)\}, \ \{ f_1^*p_i^*z_0\cdot U_{i},f_1^*s_2(i)\}$;
\item $\{f_2^*p_j^*z_1\cdot U_{j+1},f_2^*s_1(j)\}, \ \{f_2^*p_j^*z_0\cdot U_{j}, f_2^*s_2(j)\}$.
\end{itemize}
Hence, $\und{\End}(P_{ij})$ is the tensor product of endomorphism sheaves of these
Koszul matrix factorizations.
Applying Lemma \ref{Koszul-mf-lem}(i), we get on $\ov{R}_{ij}\times \A^k$
\begin{align*}
&\und{\End}(P_{ij})\simeq \\
&\OO/((U_r \ |\ r\neq i,i+1,j,j+1),f_1^*p_i^*z_1\cdot U_{i+1},f_1^*p_i^*z_0\cdot U_{i},f_2^*p_j^*z_1\cdot U_{j+1},f_2^*p_j^*z_0\cdot U_{j})=\OO_{p^{-1}\ov{R}_{ij}}.\end{align*}

In the case $j=i-1$, 
$P_{i,i-1}\ot f_1^*p_{i}^*\OO(1)$ is the tensor product of the Koszul matrix factorizations,
\begin{itemize}
\item $\{U_r,V_r\}$, for $r\neq i-1,i,i+1$;
\item $\{f_1^*p_{i}^*z_1\cdot U_{i+1},f_1^*s_1(i)\}, \ 
\{f_1^*p_{i}^*z_0\cdot f_2^*p_{i-1}^*z_1\cdot U_i,s_i\}, \ 
\{f_2^*p_{i-1}^*z_0\cdot U_{i-1},f_2^*s_2(i-1)\}.$
\end{itemize}
Hence, we get an isomorphism on $\ov{R}_{i,i-1}\times \A^k$,
\begin{align*}
&\und{\End}(P_{i,i-1})\simeq \\
&\OO/((U_r \ |\ r\neq i-1,i,i+1),f_1^*p_{i}^*z_1\cdot U_{i+1},
f_1^*p_{i}^*z_0\cdot f_2^*p_{i-1}^*z_1\cdot U_i,f_2^*p_{i-1}^*z_0\cdot U_{i-1})\simeq \OO_{p^{-1}\ov{R}_{i,i-1}}.
\end{align*}

\noindent
{\bf Case of $\und{\End}(P_{i0})$, for $0<i<k$.}

Assuming that $i>1$ 
we use the presentation of $P_{i0}$ as the tensor product of
\begin{itemize}
\item $\{U_r,V_r\}$, for $r\neq 0,i,i+1$;
\item $\{f_1^*p_i^*z_1\cdot U_{i+1},f_1^*s_1(i)\}, \ 
\{f_1^*p_i^*z_0\cdot U_{i},f_1^*s_2(i)\}, \ \{f_2^*p_0^*z_1\cdot U_1,f_2^*s_1(0)\}.$
\end{itemize}
As before, from this we deduce an isomorphism on $\ov{R}_{i0}\times \A^k$,
$$\und{\End}(P_{i0})\simeq \OO/((U_r \ |\ r\neq 0,i,i+1),f_1^*p_i^*z_1\cdot U_{i+1},f_1^*p_i^*z_0\cdot U_{i},
f_2^*p_0^*z_1\cdot U_1)\simeq \OO_{p^{-1}\ov{R}_{i0}}.$$

In the case $i=1$, we 
use the presentation of $P_{1,0}$ as the tensor product of 
\begin{itemize}
\item $\{U_r,V_r\}$, for $r>1$;
\item $\{f_1^*p_{1}^*z_1\cdot U_{2},f_1^*s_1(1)\}, \ \{f_1^*p_{1}^*z_0\cdot f_2^*p_0^*z_1\cdot U_1,s_1\}.$
\end{itemize}
From this we get an isomorphism on $\ov{R}_{1,0}\times \A^k$,
$$\und{\End}(P_{1,0})\simeq \OO/((U_r \ |\ r>1),f_1^*p_{1}^*z_1\cdot U_{2},
f_1^*p_{1}^*z_0\cdot f_2^*p_0^*z_1\cdot U_1)\simeq \OO_{p^{-1}\ov{R}_{1,0}}.$$

\noindent
{\bf Case of $\und{\End}(P_{kj})$.}

Assume first that $0<j<k-1$. Then in a neighborhood of $R_{kj}\times \A^k$,
$P_{kj}$ is obtained as the tensor product of
\begin{itemize}
\item $\{U_r,V_r\}$, for $r\neq j,j+1,k$;
\item $\{X_1U_k,C_{1,k-1}\}$;
\item $\{f_2^*p_j^*z_1\cdot U_{j+1},f_2^*s_1(j)\}, \ \{f_2^*p_j^*z_0\cdot U_{j},f_2^*s_2(j)\}$.
\end{itemize}
This leads to an isomorphism on $R_{kj}\times\A^k$,
$$\und{\End}(P_{kj})\simeq \OO/((U_r\ |\ r\neq j,j+1,k),X_1U_k,f_2^*p_j^*z_1\cdot U_{j+1},f_2^*p_j^*z_0\cdot U_{j})\simeq
\OO_{p^{-1}R_{kj}}.$$

In the case $j=k-1$, we present $P_{k,k-1}$ as the tensor product of
\begin{itemize}
\item $\{U_r,V_r\}$, for $r<k-1$;
\item $\{X_1U_k\cdot f_2^*p_{k-1}^*z_1,s_k\}$;
\item $\{f_2^*p_{k-1}^*z_0\cdot U_{k-1},f_2^*s_2(k-1)\}$,
\end{itemize}
which leads to an isomorphism on $R_{k,k-1}\times \A^k$,
$$\und{\End}(P_{k,k-1})\simeq \OO/((U_r\ |\ r<k-1),X_1U_k\cdot f_2^*p_{k-1}^*z_1,f_2^*p_{k-1}^*z_0\cdot U_{k-1})\simeq
\OO_{p^{-1}R_{k,k-1}}.$$

In the case $j=0$, in a neighborhood of $\ov{R}_{k0}\times \A^k$ we can present $P_{k0}$ as the tensor product of
\begin{itemize}
\item $\{U_r,V_r\}$, for $r\neq 1,k$;
\item $\{X_1U_k,C_{1,k-1}\}$;
\item $\{f_2^*p_0^*z_1\cdot U_1,f_2^*s_1(0)\}$,
\end{itemize}
so we get an isomorphism on $\ov{R}_{k0}\times \A^k$,
$$\und{\End}(P_{k0})\simeq \OO/((U_r\ |\ r\neq 1,k),X_1U_k,f_2^*p_0^*z_1\cdot U_1)\simeq
\OO_{p^{-1}\ov{R}_{k0}}.$$

\noindent
{\bf Case of $\und{\End}(Q_j)$.}

For $0<j<k$, we present $Q_j$ as the tensor product of
\begin{itemize}
\item $\{U_r,V_r\}$, for $r\neq j,j+1$;
\item $\{f_2^*p_j^*z_1\cdot U_{j+1},f_2^*s_1(j)\}, \ \{f_2^*p_j^*z_0\cdot U_{j},f_2^*s_2(j)\}$.
\end{itemize}
From this we get an isomorphism on $R_{*j}\times \A^k$,
$$\und{\End}(Q_j)\simeq \OO/((U_r \ |\ r\neq j,j+1),f_2^*p_j^*z_1\cdot U_{j+1},f_2^*p_j^*z_0\cdot U_{j})\simeq
\OO_Z,$$
where $Z=\ov{R}_{*j}\cup L_{*j}\times \A^1_{U_j}\cup L_{*j+1}\times \A^1_{U_{j+1}}$.

For $j=k$, in a neighborhood of $R_{kk}\times\A^k$, we can present $Q_k$ as the tensor product of
\begin{itemize}
\item $\{U_r,V_r\}$, for $r<k$;
\item $\{X_1X_2\cdot U_k,C_{2,k-1}\}$,
\end{itemize}
which leads to an isomorphism on $R_{kk}\times \A^k$,
$$\und{\End}(Q_k)\simeq \OO/((U_r \ |\ r<k),X_1X_2\cdot U_k)\simeq \OO_{R_{kk}\cup L_{*1}\times \A^1_{U_1}}.$$

For $j=0$, we present $Q_0$ as the tensor product of
\begin{itemize}
\item $\{U_r,V_r\}$, for $r\neq 1$;
\item $\{f_2^*p_0^*z_1\cdot U_1,f_2^*s_1(0)\}$,
\end{itemize}
which leads to an isomorphism on $\ov{R}_{*0}\times \A^k$,
$$\und{\End}(Q_0)\simeq \OO/((U_r \ |\ r\neq 1),f_2^*p_0^*z_1\cdot U_1)\simeq
\OO_{\ov{R}_{*0}\cup L_{*1}\times \A^1_{U_1}}.$$
\end{proof}

\begin{lem}\label{Pij-End-lem} 
Let us set $U_{k+1}:=X_1$.

\noindent
(i) For $k\ge i>j>0$, such that $i-j>1$, the natural embedding $k[U_i,U_{i+1},U_j,U_{j+1}]\sub \OO(\ov{\XX}_{2,k})$ 
induces an isomorphism 
$$\End(P_{ij})\simeq k[U_i,U_{i+1},U_j,U_{j+1}]/(U_iU_{i+1},U_jU_{j+1}).$$

\noindent
(ii) For $k\ge i>1$, there is a similar isomorphism
$$\End(P_{i0})\simeq k[U_1,U_i,U_{i+1}]/(U_iU_{i+1}).$$

\noindent
(iii) For $i>1$, there is a similar isomorphism
$$\End(P_{i,i-1})\simeq k[U_{i-1},U_i,U_{i+1}]/(U_{i-1}U_iU_{i+1}).$$ 

\noindent
(iv) There is a similar isomorphism
$$\End(P_{1,0})\simeq k[U_1,U_2].$$
\end{lem}

\begin{proof}
By Lemma \ref{local-End-lem}, we have to calculate $H^*(\OO_{p^{-1}\ov{R}_{ij}})$.
Let us consider the natural projection
$$\pi:\ov{\XX}_{2,k}\to \A^{k+1}_{U_1,\ldots,U_k,U_{k+1}}.$$
For part (i), it is enough to show the natural isomorphism 
$$R\pi_*\OO_{p^{-1}\ov{R}_{ij}}\simeq \OO/((U_r)_{r\neq i,i+1,j,j+1},U_iU_{i+1},U_jU_{j+1}),$$
and similarly for parts (ii)--(iv).

Suppose that $k>i$, $j>0$ and $i-j>1$. Then 
$$p^{-1}R_{ij}\simeq C_1(U_i,U_{i+1})\times C_2(U_j,U_{j+1}),$$
where $C_1(U_i,U_{i+1})$ and $C_2(U_j,U_{j+1})$ are nodal chains with $3$ components:
$$C_1(U_i,U_{i+1})=\A^1_{U_i}\cup \P^1\cup \A^1_{U_{i+1}}, \ \ C_2(U_j,U_{j+1})=
\A^1_{U_j}\cup \P^1\cup \A^1_{U_{j+1}},$$
The morphism $C_1(U_i,U_{i+1})\times C_2(U_j,U_{j+1})\to \A^k$ 
is the product of the natural maps (contracting the $\P^1$ components) 
$$C_1(U_i,U_{i+1})\to A^2_{U_i,U_{i+1}}, \ \ C_2(U_j,U_{j+1})\to \A^2_{U_j,U_{j+1}}$$
composed with the coordinate embedding $\A^4_{U_i,U_{i+1},U_j,U_{j+1}}\hra \A^{k+1}$.
This immediately implies the assertion.

In the case $i=k$ and $0<j<k-1$, we use a similar decomposition of 
$p^{-1}R_{kj}$, where $C_1(U_k,U_{k+1})$ is a nodal chain with the $2$ components:
$C_1(U_k,U_{k+1})=\A^1_{U_k}\cup \A^1_{U_{k+1}}$. In the case $1<i<k$ and $j=0$,
we use a similar decomposition 
$$p^{-1}\ov{R}_{i0}\simeq C_1(U_i,U_{i+1})\times C_2(U_1),$$
where $C_2(U_1)$ is a nodal chain with $2$ components:
$C_2(U_1)=\A^1_{U_1}\cup \P^1$.
In the case $i=k$ and $j=0$ we use
$$p^{-1}\ov{R}_{k0}\simeq C_1(U_1,U_{k+1})\times C_2(U_1).$$

On the other hand, for $1<i<k$, 
$$p^{-1}R_{i,i-1}=Z\cup Z', \ \ \text{ where }$$
$$Z\simeq C_1(U_{i-1})\times C_2(U_{i},U_{i+1}), \ \ Z'\simeq C_1(U_{i-1},U_{i})\times C'_2(U_{i+1}),$$
$$Z\cap Z'\simeq C_1(U_{i-1})\times C'_2(U_{i+1}),$$
where $C_1(U_{i-1})=\A^1_{U_{i-1}}\cup \P^1$ and $C'_2(U_{i+1})=\P^1\cup \A^1_{U_{i+1}}$.
Now the exact triangle
$$R\pi_*\OO_{p^{-1}R_{i,i-1}}\to R\pi_*\OO_Z\oplus R\pi_*\OO_{Z'}\to R\pi_*\OO_{Z\cap Z'}\to\ldots$$
easily gives the result.

Similarly for $i=1$ we use
$$p^{-1}R_{1,0}=Z\cup Z', \ \ \text{ where }$$
$$Z\simeq C_1(U_{1})\times C_2(U_0), \ \ Z'\simeq C_1(U_0,U_1)\times \P^1,$$
$$Z\cap Z'\simeq C_1(U_{1})\times \P^1.$$
For $i=k$ we use
$$p^{-1}R_{k,k-1}=Z\cup Z', \ \ \text{ where }$$
$$Z\simeq \A^1_{U_{k+1}}\times C_2(U_{k-1},U_k), \ \ Z'\simeq C_1(U_k,U_{k+1})\times C'_2(U_{k-1}),$$
$$Z\cap Z'\simeq \A^1_{U_{k+1}}\times C'_2(U_{k-1}).$$
\end{proof}

Similarly we compute endomorphisms of $Q_j$.

\begin{lem}\label{Qj-End-lem}
For $0<j<k$, the embedding $k[X_1,X_2,U_j,U_{j+1}]\sub \OO(\ov{\XX}_{2,k})$ induces an isomorphism
$$\End(Q_j)\simeq k[X_1,X_2,U_j,U_{j+1}]/(X_1X_2,U_jU_{j+1}).$$
Similarly,
$$\End(Q_k)\simeq k[X_1,X_2,U_k]/(U_kX_1X_2),$$
$$\End(Q_0)\simeq k[U_1,X_1,X_2]/(X_1X_2).$$
\end{lem}

It follows from Lemmas \ref{Pij-End-lem} and \ref{Qj-End-lem}
that for of our generators $G=P_{ij}$ or $G=Q_j$ the algebra 
$\End(G)$ is a quotient of the algebra $A=\k[U_1,\ldots,U_k,X_1,X_2]$.
Let us set for any pair of our generators $G,G'$,
$$\End(G,G'):=\End(G)\ot_A \End(G').$$
The morphism spaces $\Hom(G,G')$ and $\Hom(G',G)$ are $\End(G,G')$-modules,
and we will show that they are either $0$ or free of rank $1$ over $\End(G,G')$.

\subsubsection{Morphisms between $P_{ij}$'s}\label{P-morphisms-sec}

Looking at the supports we see that $\Hom(P_{ij},P_{i'j'})=0$ unless $|i'-i|\le 1$ and $|j'-j|\le 1$.
The nonzero morphisms are computed in the Lemma below. 

\begin{lem}\label{Pij-morphisms-lem}
(i) For $i>j$, we have generators
$$\a_{ij}\in \Hom^{-1-2(k-j)}(P_{i,j+1},P_{i,j}), \ \b_{ij}\in \Hom^{1+2(k-j)}(P_{i,j},P_{i,j+1})$$
such that $\a_{ij}\b_{ij}=\b_{ij}\a_{ij}=U_{j+1}$, and such that the spaces
$$\Hom^*(P_{i,j+1},P_{i,j})=\End(P_{i,j},P_{i,j+1})\cdot \a_{ij}, \ \ \Hom^*(P_{i,j},P_{i,j+1})=\End(P_{i,j},P_{i,j+1})\cdot\b_{ij}$$ 
are free $\End(P_{i,j},P_{i,j+1})$-modules.
The generator $\a_{ij}$ (resp., $\b_{ij}$) comes from the exact triangle 
$$\OO_{R_{i,j}}(-2-2(k-j))\rTo{c_{2,j}} \OO_{R_{i,j}\cup R_{i,j+1}}\to \OO_{R_{i,j+1}}\rTo{\a_{ij}} \OO_{R_{i,j}}[1](-2-2(k-j))$$
(resp., from the triangle
$$\OO_{R_{i,j+1}}(2(k-j))\rTo{c_{2,j+1}^{-1}} \OO_{R_{i,j}\cup R_{i,j+1}}\to \OO_{R_{i,j}}\rTo{\b_{ij}} \OO_{R_{i,j+1}}[1](2(k-j))\ ),$$
well defined near $R_{ij}\cap R_{i,j+1}$.

(ii) Similarly, for $i>j$, we have generators
$$\ga_{ij}\in \Hom^{-1}(P_{i+1,j},P_{i,j}), \ \de_{ij}\in \Hom^1(P_{i,j},P_{i+1,j})$$
such that $\ga_{ij}\de_{ij}=\de_{ij}\ga_{ij}=U_{i+1}$, and such that the spaces
$$\Hom^*(P_{i+1,j},P_{i,j})=\End(P_{i,j},P_{i+1,j})\cdot \ga_{ij}, \ \ \Hom^*(P_{i,j},P_{i+1,j})=\End(P_{i,j},P_{i+1,j})\cdot\de_{ij}$$ 
are free $\End(P_{i,j},P_{i+1,j})$-modules.
For $i<k$, the generator $\ga_{ij}$ (resp., $\de_{ij}$) has the form 
$$\ga_{ij}=\ga'_{ij}\cdot \frac{f_1^*p_{i+1}^*z_1}{f_1^*p_{i}^*z_0},$$ 
        $$(\text{ resp., }\ \de_{ij}=\de'_{ij}\cdot \frac{f_1^*p_{i}^*z_0}{f_1^*p_{i+1}^*z_1}),$$
where $\ga'_{ij}$ (resp., $\de'_{ij}$) comes from the exact triangle
$$\OO_{R_{i,j}}(-2-2(k-i))\rTo{c_{1,i}} \OO_{R_{i,j}\cup R_{i+1,j}}\to \OO_{R_{i+1,j}}\rTo{\ga'_{ij}} \OO_{R_{i,j}}[1](-2-2(k-i))$$
(resp., from the triangle
$$\OO_{R_{i+1,j}}(2(k-i))\rTo{c_{1,i+1}^{-1}} \OO_{R_{i,j}\cup R_{i+1,j}}\to \OO_{R_{i,j}}\rTo{\de'_{ij}} \OO_{R_{i+1,j}}[1](2(k-i))\ ),$$
well defined near $R_{i+1,j}\cap R_{i+1,j}$.
For $i=k$, we have
$$\ga_{kj}=\ga'_{kj}\cdot \frac{1}{f_1^*p_{k-1}^*z_0}, \ \ \de_{kj}=\de'_{0j}\cdot f_1^*p_{k-1}^*z_0,$$
where $\ga'_{kj}$ and $\de'_{kj}$ come from the above triangles with $i=k$.

\noindent
(iii) For $i-j>1$, the space $\Hom(P_{i+1,j+1},P_{ij})$ (resp., $\Hom(P_{ij},P_{i+1,j+1})$) is a free $\End(P_{ij},P_{i+1,j+1})$-module
generated by $\ga_{i,j+1}\a_{ij}=\a_{i+1,j}\ga_{ij}$ (resp., $\de_{ij}\b_{i+1,j}=\b_{ij}\de_{i,j+1}$).
The similar result holds for $\Hom(P_{i+1,j-1},P_{ij})$ (resp., $\Hom(P_{ij},P_{i+1,j-1})$).
The space $\Hom(P_{i+1,i+2},P_{i,i+1})$ (resp., $\Hom(P_{i,i+1},P_{i+1,i+2})$) is a free $\End(P_{i,i+1},P_{i+1,i+2})$-module
generated by $\ga_{i,i+2}\a_{i,i+1}$ (resp., $\b_{i,i+1}\de_{i,i+2}$).
\end{lem}

\begin{proof} As before, we first compute the sheaves of homomorphisms using the complete intersection
equations for $\ov{R}_{ij}$ and Lemma \ref{Koszul-mf-lem}.
Furthermore, when calculating $\Hom(P_{ij},P_{i'j'})$
we only need to study the situation in a neighborhood of $(\ov{R}_{ij}\cap \ov{R}_{i'j'})\times \A^k$.

For example, let us compute $\und{\Hom}(P_{i,j+1},P_{i,j})$, where $i<k$.
We can work near $(R_{i,j}\cap R_{i,j+1})\times \A^k$, so we can replace
$f_2^*s_1(j+1)$ with $V_{j+2}$ and $f_2^*s_2(j+1)$ with $C_{2,j}$ in the equations of $R_{i,j+1}$.
Similarly, 
we can replace $f_2^*s_1(j)$ with $C_{2,j+1}^{-1}$ and $f_2^*s_2(j)$ with $V_{j}$
in the equations of $R_{i,j}$.
Thus, we split $W$ as
$$W=\sum_{r\neq i,i+1,j+1}V_rU_r+f_1^*s_1(i)\cdot f_1^*p_i^*z_1\cdot U_{i+1}+f_1^*s_2(i)\cdot f_1^*p_i^*z_0\cdot U_{i}+ 
C_{2,j+1}^{-1}C_{2,j} \cdot U_{j+1},$$
and then take the tensor product of the following sheaves
of homomorphisms between Koszul matrix factorizations:
\begin{itemize}
\item for $r\neq i,i+1,j+1$, $\und{\End}(\{U_r,V_r\})\simeq \OO/(U_r,V_r)$;
\item $\und{\End}(\{f_1^*p_i^*z_1\cdot U_{i+1},f_1^*s_1(i)\})\simeq \OO/(f_1^*p_i^*z_1\cdot U_{i+1},f_1^*s_1(i))$;
\item $\und{\End}(\{f_1^*p_i^*z_0\cdot U_{i},f_1^*s_2(i)\})\simeq \OO/(f_1^*p_i^*z_0\cdot U_{i},f_1^*s_2(i));$
\item $\und{\Hom}(\{C_{2,j+1}^{-1}\cdot U_{j+1},C_{2,j}\},\{C_{2,j}\cdot U_{j+1},C_{2,j+1}^{-1}\})\simeq
\OO/(C_{2,j+1}^{-1},C_{2,j})[-1]$.
\end{itemize}
Thus, we get an isomorphism on $\ov{R}_{ij}\cap\ov{R}_{i,j+1}\times \A^k$,
$$\und{\Hom}(P_{i,j+1},P_{i,j})\simeq 
\OO/((U_r \ |\ r\neq i,i+1,j+1),f_1^*p_i^*z_1\cdot U_{i+1},f_1^*p_i^*z_0\cdot U_{i})[-1]\simeq
\OO_{C(U_i,U_{i+1})\times \A^1_{U_{j+1}}}[-1],$$
where $C(U_i,U_{i+1})$ is the nodal chain with the components $\A^1_{U_i}$, $\P^1$,
and $\A^1_{U_{i+1}}$. 
Thus, we get
$$\Hom(P_{i,j+1},P_{i,j})\simeq \k[U_i,U_{i+1},U_{j+1}]/(U_iU_{i+1})[-1]=\End(P_{i,j+1},P_{i,j})[-1].$$
Chasing the generators from Lemma \ref{Koszul-mf-lem} we get the statement in this case.

In the case of $\und{\Hom}(P_{i+1,j},P_{i,j})$
we take the tensor product of the sheaves of homomorphisms,
\begin{itemize}
\item for $r\neq i+1,j,j+1$, $\und{\End}(\{U_r,V_r\})\simeq \OO/(U_r,V_r)$;
\item $\und{\End}(\{f_2^*p_j^*z_1\cdot U_{j+1},f_2^*s_1(j)\})\simeq \OO/(f_2^*p_j^*z_1\cdot U_{j+1},f_2^*s_1(j))$;
\item $\und{\End}(\{f_2^*p_j^*z_0\cdot U_{j},f_2^*s_2(j)\})\simeq \OO/(f_2^*p_j^*z_0\cdot U_{j},f_2^*s_2(j))$;
\item $\und{\Hom}(\{C_{1,i+1}^{-1}\cdot U_{i+1},C_{1,i}\},\{C_{1,i}\cdot U_{i+1},C_{1,i+1}^{-1}\})\simeq
\OO/(C_{1,i+1}^{-1},C_{1,i})[-1]$;
\end{itemize}
(we can omit the twists by line bundles in the definition of $P_{ij}$ and $P_{i+1,j}$ since we are working
near $\ov{R}_{ij}\cap \ov{R}_{i+1,j}$).
Using trivializations of the relevant line bundles over $\ov{R}_{ij}\cap \ov{R}_{i+1,j}\times \A^k$, we deduce the result
in this case.

Similarly we consider other $\Hom$-spaces in (i) and (ii).
Let us show how a similar computation goes in (iii). For example, for $i-j>1$ we compute
$\und{\Hom}(P_{i+1,j+1},P_{i,j})$ in the neighborhood of $(R_{i,j}\cap R_{i+1,j+1})\times\A^k$
as the tensor product of
\begin{itemize}
\item for $r\neq i+1,j+1$, $\und{\End}(\{U_r,V_r\})\simeq \OO/(U_r,V_r)$;
\item $\und{\Hom}(\{C_{1,i+1}^{-1}\cdot U_{i+1},C_{1,i}\},\{C_{1,i}\cdot U_{i+1},C_{1,i+1}^{-1}\})\simeq
\OO/(C_{1,i+1}^{-1},C_{1,i})[-1]$;
\item $\und{\Hom}(\{C_{2,j+1}^{-1}\cdot U_{j+1},C_{2,j}\},\{C_{2,j}\cdot U_{j+1},C_{2,j+1}^{-1}\})\simeq
\OO/(C_{2,j+1}^{-1},C_{2,j})[-1]$.
\end{itemize}
Recall that $R_{ij}\cap R_{i+1,j+1}$ is a single point $p_{i+1,j+1}$.
Thus, we get an isomorphism on $p_{i+1,j+1})\times\A^k$,
$$\und{\Hom}(P_{i+1,j+1},P_{i,j})\simeq \OO_{p_{i+1,j+1}\times\A^2_{U_{i+1},U_{j+1}}}[-2],$$
so 
$$\Hom(P_{i+1,j+1},P_{ij})\simeq \k[U_{i+1},U_{j+1}][-2]\simeq \End(P_{ij},P_{i+1,j+1})[-2]$$
and we can chase the generator in the above calculation.
\end{proof}

\subsubsection{Morphisms between $Q_i$'s}\label{Q-morphisms-sec}

We have the following analog of Lemma \ref{Pij-morphisms-lem}.
The proof is analogous (but simpler), so we omit it.

\begin{lem}\label{Qj-morphisms-lem}
(i) For $j\ge 0$, we have generators
$$\a_{j}\in \Hom^{-1-2(k-j)}(Q_{j+1},Q_j), \ \b_{j}\in \Hom^{1+2(k-j)}(Q_{j},Q_{j+1})$$
such that $\a_{j}\b_{j}=\b_{j}\a_{j}=U_{j+1}$, and such that
$$\Hom^*(Q_{j+1},Q_j)=\End(Q_{j},Q_{j+1})\cdot \a_j, \ \ \Hom^*(Q_{j},Q_{j+1})=\End(Q_{j},Q_{j+1})\cdot\b_{j}$$ 
in a way compatible with the $\End(Q_{j})-\End(Q_{j+1})$-bimodule structure.
The generator $\a_{j}$ (resp., $\b_{j}$) comes from the exact triangle 
$$\OO_{R_{*j}}(-2-2(k-j))\rTo{c_{2,j}} \OO_{R_{*j}\cup R_{*j+1}}\to \OO_{R_{*j+1}}\rTo{\a_{j}} \OO_{R_{*j+1}}[1](-2-2(k-j))$$
(resp., from the triangle
$$\OO_{R_{*j+1}}(2(k-j))\rTo{c_{2,j+1}^{-1}} \OO_{R_{*j}\cup R_{*j+1}}\to \OO_{R_{*j}}\rTo{\b_{j}} \OO_{R_{*j+1}}[1](2(k-j))\ ),$$
well defined near $R_{*j}\cap R_{*j+1}$.
\end{lem}

\subsubsection{Morphisms between $Q_j$ and $P_{ij}$}\label{QP-morphisms-sec}

\begin{lem}\label{Hom-P-Q-van-lem}
For $k>i>j$ and any $r$, one has $\Hom(P_{ij},Q_r)=\Hom(Q_r,P_{ij})=0$. 
\end{lem}

\begin{proof} We only have to consider the cases $r=j-1$, $j$ or $j+1$.
In addition, we can use the vanishing criterion for
the proper morphism $\ov{\XX}_{2,k}\to \A^{2k+2}_{U_1,\ldots,U_k,V_1,\ldots,V_k,X_1,X_2}$ as in 
Lemma \ref{SD-mf-lem}(ii). Namely, by this Lemma, assuming that either $j>0$ or $r>0$, we obtain
that the vanishing of $\Hom(Q_r,P_{ij})$ implies the vanishing of $\Hom(P_{ij},Q_r)$.
Indeed, the local computation shows that $\und{\Hom}(P_{ij},Q_r)$ is a push-forward of a coherent sheaf
from $R_{ij}\cap R_{*r}$, so we can use the fact that for
$j>0$ (resp., $r>0$), the relative canonical bundle has trivial restriction to $R_{ij}$ (resp., $R_{*j}$)
(which follows from Corollary \ref{CY-cor}).

However, in the case $r=j=0$, we have to deal with both $\Hom$-spaces.

\noindent
{\bf Case $r=j$, $\Hom(Q_j,P_{ij})$}.
As before, we use the equations of $\ov{R}_{*j}$ and Lemma \ref{Koszul-mf-lem}. 
For $i-1>j>0$, we use the splitting \eqref{W-decomp-eq} of $W$, and consider the tensor product of sheaf
homomorphisms between the following Koszul matrix factorizations, twisted by $f_1^*p_i^*\OO(-1)$ (due to
the definition of $P_{ij}$),
\begin{itemize}
\item for $r\neq i,i+1,j,j+1$, $\und{\End}(\{U_r,V_r\})\simeq \OO/(U_r,V_r)$;
\item $\und{\Hom}(\{U_{i+1},V_{i+1}\},\{f_1^*p_i^*z_1\cdot U_{i+1},f_1^*s_1(i)\})\simeq \OO/(U_{i+1},f_1^*s_1(i))$; 
\item $\und{\Hom}(\{U_{i},V_{i}\},\{ f_1^*p_i^*z_0\cdot U_{i},f_1^*s_2(i)\})\simeq 
\OO/(U_{i},f_1^*s_2(i))$;
\item $\und{\End}(\{f_2^*p_j^*z_1\cdot U_{j+1},f_2^*s_1(j)\})\simeq \OO/(f_2^*p_j^*z_1\cdot U_{j+1},f_2^*s_1(j))$;
\item $\und{\End}(\{f_2^*p_j^*z_0\cdot U_{j}, f_2^*s_2(j)\})\simeq 
\OO/(f_2^*p_j^*z_0\cdot U_{j}, f_2^*s_2(j))$.
\end{itemize}
Thus, we get an isomorphism on $\ov{R}_{ij}\times \A^k$, 
\begin{align*}
&\und{\Hom}(Q_j,P_{ij})\simeq \OO/((U_r \ |\ r\neq j,j+1),U_{j+1}\cdot f_2^*p_j^*z_1,f_2^*p_j^*z_0\cdot U_{j})\ot f_1^*p_i^*\OO(-1)=\\
&\OO_Z(-1),
\end{align*}
where $Z\simeq \P^1\times C$, with $\OO(-1)$ coming from the factor $\P^1$. Here, for $j>0$,
$C$ is the nodal curve $C_2(U_j,U_{j+1})$, as in the proof of Lemma \ref{Pij-End-lem}.
For $j=0$, we have $C=C_2(U_0)$.
Calculating the cohomology of $\OO(-1)$ on this product we get zero.

\noindent
{\bf Case $r=j=0$, $\Hom(P_{i0},Q_0)$}. 
The sheaf $\und{\Hom}(P_{i0},Q_0)\ot f_1^*p_i^*\OO(-1)$ is given by the tensor product of
\begin{itemize}
\item $\und{\End}(\{U_r,V_r\})\simeq \OO/(U_r,V_r)$, for $r\neq 0,i,i+1$;
\item $\und{\Hom}(\{f_1^*p_i^*z_1\cdot U_{i+1},f_1^*s_1(i)\},\{U_{i+1},V_{i+1}\})\simeq \OO/(f_1^*s_1(i),U_{i+1})\ot f_1^*p_{i}^*\OO(-1)$; 
\item $\und{\Hom}(\{f_1^*p_i^*z_0\cdot U_{i},f_1^*s_2(i)\},\{(U_{i},V_{i}\})\simeq 
\OO/(f_1^*s_2(i),U_{i})\ot f_1^*p_i^*\OO(-1)$; 
\item $\und{\End}(\{f_2^*p_0^*z_1\cdot U_1,f_2^*s_1(0)\})\simeq \OO/(f_2^*p_0^*z_1\cdot U_1,f_2^*s_1(0))$,
\end{itemize}
where for $i=1$ we replace the last two factors with
\begin{align*}
&\und{\Hom}(\{f_1^*p_{1}^*z_0\cdot f_2^*p_0^*z_1\cdot U_1,s_1\},\{f_2^*p_0^*z_1\cdot U_1,f_2^*s_1(0)\})\simeq\\
&\OO/(f_2^*p_0^*z_1\cdot U_1,s_1)\ot f_1^*p_{1}^*\OO(-1).
\end{align*}
This leads to an isomorphism over $\ov{R}_{i0}\times \A^k$,
$$\und{\Hom}(P_{i0},Q_0)\simeq \OO/((U_r\ |\ r\neq 0),f_2^*p_0^*z_1\cdot U_1)\ot f_1^*p_i^*\OO(-1)\simeq \OO_Z(-1),$$
where $Z=\P^1\times C_2(U_0)$, with $\OO(-1)$ coming from the factor $\P^1$. 
Hence, the cohomology of this sheaf vanish.

\noindent
{\bf Case $r=j+1$, $\Hom(Q_{j+1},P_{ij})$}.
Assume first that $i>j+1$.
We can work locally near $(R_{ij}\cap R_{*,j+1})\times \A^k$, so we can replace $f_2^*s_1(j)$ with $C_{2,j+1}^{-1}$ 
and $f_2^*s_2(j)$ with $V_j$ (resp., (resp., $f_2^*s_1(j+1)$ with $V_{j+2}$ and $f_2^*s_2(j+1)$ with $C_{2,j}$)
in the equations of $R_{ij}$ (resp., $R_{*,j+1}$). Thus, 
$\und{\Hom}(Q_{j+1},P_{ij})\ot f_1^*p_i\OO(1)$ is the tensor product of
\begin{itemize}
\item $\und{\End}(\{U_r,V_r\})\simeq \OO/(U_r,V_r)$, for $r\neq i,i+1,j+1$;
\item $\und{\Hom}(\{U_{i+1},V_{i+1}\},\{f_1^*p_i^*z_1\cdot U_{i+1},f_1^*s_1(i)\})\simeq \OO/(U_{i+1},f_1^*s_1(i));$
\item $\und{\Hom}(\{U_{i},V_{i}\},\{f_1^*p_i^*z_0\cdot U_{i},f_1^*s_2(i)\})\simeq \OO/(U_{i},f_1^*s_2(i));$
\item $\und{\Hom}(\{C_{2,j+1}^{-1}\cdot U_{j+1},C_{2,j}\},\{C_{2,j}\cdot U_{j+1},C_{2,j+1}^{-1}\})\simeq\OO/(C_{2,j+1}^{-1},C_{2,j})[-1].$
\end{itemize}
Hence, we get an isomorphism on $(R_{ij}\cap R_{*,j+1})\times \A^k$,
\begin{equation}\label{sh-Hom-Qj-1Pij-isom}
\und{\Hom}(Q_{j+1},P_{ij})\simeq \OO/(U_r \ |\ r\neq j+1)\ot f_1^*p_i^*\OO(-1)[-1]=
\OO_{(R_{ij}\cap R_{*,j+1})\times \A^1_{U_{j+1}}} \ot f_1^*p_i^*\OO(-1)[1],
\end{equation}
which has vanishing cohomology.

In the case $i=j+1$ we can use 
$$\ldots,V_{j+3},f_1^*s_1(j+1),s_{j+1},V_{j},\ldots$$
as equations of $R_{j+1,j}$ and
$$\ldots,V_{j+3},V_{j+2},f_2^*p_j^*z_1,V_{j},\ldots$$
as equations of $Q_{*,j+1}$, so 
$\und{\Hom}(Q_{j+1},P_{j+1,j})\ot f_1^*p_i\OO(1)$ is the tensor product of
\begin{itemize}
\item $\und{\End}(\{U_r,V_r\})\simeq \OO/(U_r,V_r)$, for $r\neq j+1,j+2$;
\item $\und{\Hom}(\{U_{j+2},V_{j+2}\},\{f_1^*p_{j+1}^*z_1\cdot U_{j+2},f_1^*s_1(j+1)\})\simeq \OO/(U_{j+2},f_1^*s_1(j+1));$
\item $\und{\Hom}(\{s_{j+1} f_1^*p_{j+1}^*z_0\cdot U_{j+1},f_2^*p_{j}^*z_1\},
\{f_1^*p_{j+1}^*z_0 f_2^*p_{j}^*z_1\cdot U_{j+1},s_{j+1}\})\simeq\OO/(f_2^*p_{j}^*z_1,s_{j+1})\ot f_2^*p_{j}^*\OO(1)[-1].$
\end{itemize}
Taking into account the trivialization of $f_2^*p_{j}^*\OO(1)$, we see that isomorphism \eqref{sh-Hom-Qj-1Pij-isom} still holds,
and so the $\Hom$-space vanishes.

\noindent
{\bf Case $r=j-1$, $\Hom(Q_{j-1},P_{ij})$}.
We work locally near $(R_{ij}\cap R_{*,j-1})\times\A^k$, so we can use
$$\ldots,V_{i+2},f_1^*s_1(i),f_1^*s_2(i),V_{i-1},\ldots,V_{j+1},C_{2,j-1},V_{j-1},\ldots$$
as equations of $R_{ij}$ (even in the case $i=j+1$), and
$$\ldots,V_{j+1},C_{2,j}^{-1},V_{j-1},\ldots$$
as equations of $R_{*,j-1}$. Thus, we can compute
$\und{\Hom}(Q_{j-1},P_{ij})\ot f_1^*p_i\OO(1)$ as the tensor product of 
\begin{itemize}
\item $\und{\End}(\{U_r,V_r\})\simeq \OO/(U_r,V_r)$, for $r\neq i,i+1,j$;
\item $\und{\Hom}(\{U_{i+1},V_{i+1}\},\{f_1^*p_i^*z_1\cdot U_{i+1},f_1^*s_1(i)\})\simeq \OO/(U_{i+1},f_1^*s_1(i))$; 
\item $\und{\Hom}(\{U_{i},V_{i}\},\{f_1^*p_i^*z_0\cdot U_{i},f_1^*s_2(i)\})\simeq \OO/(U_{i},f_1^*s_2(i));$
\item $\und{\Hom}(\{C_{2,j-1}\cdot U_{j},C_{2,j}^{-1}\},\{C_{2,j}^{-1}\cdot U_{j},C_{2,j-1}\})\simeq
\OO/(C_{2,j}^{-1},C_{2,j-1})[-1].$
\end{itemize}
This leads to
$$\und{\Hom}(Q_{j-1},P_{ij})\simeq \OO_{(R_{ij}\cap R_{*,j-1})\times \A^1_{U_{j}}} \ot f_1^*p_i^*\OO(-1)[1],$$
which has vanishing cohomology.
\end{proof}

\begin{lem}\label{PQ-morphisms-lem} 
(i) For $j<k$, the space $\Hom(Q_j,P_{kj})$ is freely generated over $\End(Q_j,P_{kj})$ by
the generator $\de_{j}$ of degree $0$ that corresponds to the morphism
$$\OO_{R_{*j}}\to \OO_{R_{kj}}.$$
On the other hand, the space $\Hom(P_{kj},Q_j)$ is freely generated over $\End(Q_j,P_{kj})$ by the generator
$\ga_{j}$ of degree $0$ the corresponds to the morphism
$$\OO_{R_{kj}}\rTo{X_1\cdot?} \OO_{R_{*j}}.$$
We have $\ga_{j}\de_{j}=\de_{j}\ga_{j}=X_1$.

\noindent
(ii) For $j\le k$, the space $\Hom(Q_j,P_{k,j-1})$ (resp., $\Hom(P_{k,j-1},Q_j)$)
is freely generated over $\End(Q_j,P_{k,j-1})$ by
$\de_{j-1}\a_{j-1}$ (resp., $\b_{j-1}\ga_{j-1}$),
where we use generators from Lemma \ref{Qj-morphisms-lem} and from part (i).
For $j<k-1$ the space $\Hom(Q_j,P_{k,j+1})$ (resp., $\Hom(P_{k,j+1},Q_j)$)
is freely generated over $\End(Q_j,P_{0,j+1}$) by 
$\de_{j+1}\b_{j}$ (resp., $\a_{j}\ga_{j+1}$).

\noindent
(iii) For $j<k$, we have $\de_{j-1}\a_{j-1}=\a_{k,j-1}\de_j$ in $\Hom(Q_j,P_{k,j-1})$
(resp., $\b_{j-1}\ga_{j-1}=\ga_j\b_{k,j-1}$ in $\Hom(P_{k,j-1},Q_j)$). We also have
$\de_j\b_{j-1}=\b_{k,j-1}\de_{j-1}$ in $\Hom(Q_{j-1},P_{kj})$ (resp., 
$\a_{j-1}\ga_j=\ga_{j-1}\a_{k,j-1}$ in $\Hom(P_{kj},Q_{j-1})$).
\end{lem}

\begin{proof} 
(i) As in Lemma \ref{Hom-P-Q-van-lem}, we get an isomorphism on $\ov{R}_{kj}\times\A^k$,
$$\und{\Hom}(Q_j,P_{kj})\simeq \OO_{p^{-1}\ov{R}_{kj}}/(U_k).$$
Now the locus $U_k=0$ in $p^{-1}\ov{R}_{kj}$ is the product $\A^1_{X_1}\times C_2(U_j,U_{j+1})$ for $j>0$, and
$\A^1_{X_1}\times C_2(U_1)$ for $j=0$.
The space of global functions on this locus coincides with $\End(Q_j,P_{kj})$.

Assume now that $j<k-1$.
Then as before, we get an isomorphism on $\ov{R}_{*j}\times \A^k$,
\begin{align*}
&\und{\Hom}(P_{kj},Q_j)\simeq \OO/(U_r \ |\ r\neq k,j,j+1)\ot \OO/(U_k,C_{1,k-1}) \ot \\
&\OO/(f_2^*p_j^*z_1\cdot U_{j+1},f_2^*p_j^*z_0\cdot U_{j})\simeq
\OO_{p^{-1}\ov{R}_{kj}}/(U_k),
\end{align*}
which is the same sheaf as before.

In the case $j=k-1$, the sheaf $\und{\Hom}(P_{k,k-1},Q_{k-1})$ is the tensor product of
\begin{itemize}
\item $\und{\End}(\{U_r,V_r\})\simeq \OO/(U_r,V_r)$, for $r<k-1$;
\item $\und{\Hom}(\{X_1f_2^*p_{k-1}^*z_1\cdot U_k,s_k\},\{f_2^*p_{k-1}^*z_1\cdot U_k,f_2^*s_{k-1}(1)\})\simeq \OO/(f_2^*p_{k-1}^*z_1\cdot U_k,s_1)$;
\item $\und{\End}(\{f_2^*p_{k-1}^*z_0\cdot U_{k-1},f_2^*s_2(k-1)\})\simeq \OO/(f_2^*p_{k-1}^*z_0\cdot U_{k-1},f_2^*s_2(k-1))$,
\end{itemize}
so we get an isomorphism on $R_{k,k-1}\times\A^k$,
$$\und{\Hom}(P_{k,k-1},Q_{k-1})\simeq \OO/((U_r \ |\ r<k-1),f_2^*p_{k-1}^*z_1\cdot U_k,f_2^*p_{k-1}^*z_0\cdot U_{k-1})=\OO_Z,$$
where $Z\simeq \A^1_{X_1}\times C_2(U_{k-1},U_k)$.
Hence,
$$\Hom(P_{k,k-1},Q_{k-1})\simeq \k[X_1,U_{k-1},U_k]/(U_{k-1}U_k)=\End(P_{k,k-1},Q_{k-1}).$$

\noindent
(ii) Similarly to Lemma \ref{Hom-P-Q-van-lem} we get the following local isomorphisms
$$\und{\Hom}(Q_j,P_{k,j-1})\simeq \und{\Hom}(P_{k,j-1},Q_j)\simeq \OO_{(R_{k,j-1}\cap R_{*j})\times\A^1_{U_{j}}}[-1],$$
$$\und{\Hom}(Q_j,P_{k,j+1})\simeq \und{\Hom}(P_{k,j+1},Q_j)\simeq \OO_{(R_{k,j+1}\cap R_{*j})\times\A^1_{U_{j+1}}}[-1].$$
Chasing the generators from Lemma \ref{Koszul-mf-lem} we get the assertion.

\noindent
(iii) The relation $\de_{j-1}\a_{j-1}=\a_{k,j-1}\de_j$ comes from the morphism of exact triangles
\begin{diagram}
\OO_{R_{*j-1}}&\rTo{c_{2,{j-1}}}&\OO_{R_{*j-1}\cup R_{*j}}&\rTo{}& \OO_{R_{*j}} &\rTo{\a_{j-1}}& \OO_{R_{*j-1}}[1]\\
\dTo{\de_{j-1}}&&\dTo{}&&\dTo{\de_j}&&\dTo{\de_{j-1}}\\
\OO_{R_{0,j-1}}&\rTo{c_{2,j-1}}&\OO_{R_{k,j-1}\cup R_{k,j}}&\rTo{}& \OO_{R_{k,j}}&\rTo{\a_{k,j-1}}&\OO_{R_{k,j-1}}[1]
\end{diagram}
The other relations are proved similarly.
\end{proof}

\subsubsection{Recovering the multigrading; shifted generators}\label{Recovering-multigrading}

Recall that we are interested in the $\bbL_B$-graded category of $\wt{T}$-equivariant matrix factorizations
$$\BB_{2,k}:=\MF_{\wt{T}}(\ov{\XX}_{2,k},\bw_{2,k}),$$
and our generating matrix factorizations $P_{ij}$ and $Q_j$ have natural lifts
to this category (recall that we equip $f_1^*p_i^*\OO(1)$ with the $T$-equivariant structure, so that $f_1^*p_i^*z_0$ has
weight $0$).

The group $\wt{H}$ of characters of the torus $\wt{T}$ is a free abelian group with generators $x_1,x_2,u_1,\ldots,u_k,l_0$.
Hence, the grading group $\bbL_B=2\wt{H}+\Z l_0$ is free on generators $2x_1,2x_2,2u_1,\ldots,2u_k,l_0$.
So far, we computed morphisms in the $\Z$-graded category obtained 
from $\BB_{2,k}$ by
collapsing the $\bbL_B$-grading via the homomorphism
$f:\bbL_B\to \Z$ given by 
$$f(2x_1)=f(2x_2)=f(2u_i)=0, f(l_0)=1.$$
for which we have
$f(2v_i)=2$, $f(2c_{1i})=f(2c_{2i})=2(k-i)$, and $f(2x_3)=2k$.

Now for all morphisms between the generators $(P_{ij}, Q_j)$ in this $\Z$-graded category we can identify
morphisms in $\BB_{2,k}$ that they come from, and in this way we will
determine completely morphisms between our generators in $\BB_{2,k}$.

We start with endomorphisms of $P_{ij}$ and $Q_j$. Recall that all of them are generated by 
the multiplication with functions $X_1,X_2$ and $U_i$. Hence, we lift them to the same endomorphisms
in $\BB_{2,k}$, and they acquire the natural (doubled) grading, so $X_i$ will have the grading $2x_i$ and $U_i$ will have
the grading $2u_i$.

The exact triangles in Lemmas \ref{Pij-morphisms-lem}, \ref{PQ-morphisms-lem} and \ref{Qj-morphisms-lem}
show that the generators of morphisms between $(P_{ij})$ have the following lifts:
$$\a_{ij}\in\Hom^{l_0-2c_{2,j}}(P_{i,j+1},P_{i,j}), \ \a_j\in\Hom^{l_0-2c_{2,j}}(Q_{j+1},Q_{j}),$$
$$\ga_{ij}\in\Hom^{l_0-2c_{1,i}+2c_{1,i+1}}(P_{i+1,j},P_{i,j}) \ \text{ for } i<k, \ \ga_{kj}\in\Hom^{l_0-2c_{1,k-1}}(P_{k,j},P_{k-1,j})$$
$$\de_j\in\Hom^0(Q_j,P_{kj}).$$
The gradings of the remaining generators are determined by the relations and by the gradings of the endomorphisms.

Now let us define the following shifted generators in $\BB_{2,k}$:
$$\wt{Q}_j:=Q_j[k-j]\ot (-2c_{2,j}-2c_{2,j+1}-\ldots-2c_{2,k-1}),$$
$$\wt{P}_{ij}:=P_{ij}[2k-i-j]\ot (-2c_{1,i}-2c_{2,j}-2c_{2,j+1}-\ldots-2c_{2,k-1}).$$
Then the morphisms $\a_{ij}$, $\a_j$, $\ga_{ij}$ and $\de_j$ induce morphisms of degree $0$ between these shifted
generators.

\begin{thm}\label{B-side-formality-thm}
The $\bbL_B$-graded $A_\infty$-algebra of endomorphisms of $\bigoplus \wt{Q}_j\oplus \bigoplus \wt{P}_{ij}$ is formal.
The $\bbL_B$-grading on its cohomology is uniquely determined by the condition that the morphisms $\a_{ij}$, $\a_j$, $\ga_{ij}$ and $\de_j$ have degree $0$,
while $X_r$ (resp., $U_r$), viewed as an endomorphism of $\wt{Q}_j$ or $\wt{P}_{ij}$, has degree $2x_r\in \bbL_B$
(resp., $2u_r\in \bbL_B$).
\end{thm}

\begin{proof} Let us denote by $S$ a multiplicatively closed subset of arrows in $\End(\bigoplus \wt{Q}_j\oplus \bigoplus \wt{P}_{ij})$ (viewed as a quiver with relations), which includes all arrows $X_r$ and $U_r$ (forming a loop at some vertex), as well as the arrows $\a_{ij}$, $\a_j$, $\ga_{ij}$ and $\de_j$. 
The computations in Sec.\ \ref{End-sec}, \ref{P-morphisms-sec}, \ref{QP-morphisms-sec}, \ref{Q-morphisms-sec} show
that for every other generator $t$ of the $\End$-algebra, there exists an element $s\in S$ such that $st$ is a nonzero element of $S$.
Hence, the grading of the entire $\End$-algebra is uniquely determined by the grading of the elements of $S$.
 
The formality statement follows from Lemma \ref{formality-lem}, due to the fact that with the respect to the $\Z$-grading induced by $f:\bbL_B\to \Z$,
the subset $S$ has degree $0$.
\end{proof}

\subsection{Localization for $n=2$}

As in Section \ref{n1-loc-sec}, we denote by $\BB_{2,k}^\Z$
(resp., $\ov{\BB}_{2,k}^\Z$) 
the category of $\Z$-graded matrix factorizations on $\bw_{2,k}$ on $\XX_{2,k}$ (resp., $\ov{\XX}_{2,k}$),
corresponding to the grading 
$$|U_i|=2, \ |X_i|=|C_{ij}|=|V_i|=0.$$

We have equivalences 
$$\BB_{2,k}^\Z\simeq D^b\Coh(\ZZ_{2,k}), \ \ \ov{\BB}_{2,k}^\Z\simeq D^b\Coh(\ov{\ZZ}_{2,k}),$$
and we can identify $D^b\Coh(\ZZ_{2,k})$ with the quotient of $\ov{\BB}_{2,j}^\Z$ by 
the subcategory of matrix factorizations supported on
the boundary of $\ov{\ZZ}_{2,k}$, 
$$D:=f_2^{-1}(p)=\crit(w)\cap (\ov{\YY}_{2,k}\setminus \YY_{2,k}),$$
where $p$ is the infinite point on $\ov{R}_0=\P^1\sub \ov{\YY}_{1,k}$.

Recall that $D$ is a nodal chain,
$$D=D'_0\cup D_0\cup D_1\cup\ldots \cup D_k,$$
where $D_i=\ov{R}_{i0}\setminus R_{i0}$ for $i>0$ and
$D'_0\cup D_0=\ov{R}_{00}\setminus R_{00}$.
Thus, by Lemma \ref{generation-lem-k=1},
our subcategory is generated by the matrix factorizations corresponding to
the structure sheaves of the components $D_0,\ldots,D_k$, as well as the structure sheaf of $D$.
Actually, the same proof shows that the objects 
$$\OO_D,\OO_{D_0}(-1),\ldots,\OO_{D_k}(-1)$$
generate $D^b\Coh(D)$.

Similarly to Lemma \ref{n1-resolution-lem} we can find the resolutions for these objects.
Let us set for brevity $\OO(H_i^v):=f_1^*p_i^*\OO(1)$, $\OO(H_i^h):=f_2^*p_i^*\OO(1)$
(here $v$ stands for ``vertical", $h$ stands for ``horizontal").
Let us set $\ov{R}_{i,\le j}:=\cup_{j'\le j}\ov{R}_{ij'}$,  $\ov{R}_{i*}:=\cup_j \ov{R}_{ij}$
$\ov{R}_{\le i,j}:=\cup_{i'\le i}\ov{R}_{i'j}$, and
consider the objects, for $i<j$,
$$P_{i,\le j}:=\OO_{\ov{R}_{i,\le j}}(H_{j-1}^h+\ldots+H_0^h-H_i^v),$$
$$P_{i,\le i}:=\OO_{\ov{R}_{i*}}(H_i^h+\ldots+H_0^h-H_i^v),$$
$$P'_{i,\le j}:=\OO_{\ov{R}_{i,\le j}}(H_j^h+\ldots+H_{1}^h-H_i^v),$$
$$P'_{i,\le i}:=\OO_{\ov{R}_{i*}}(H_i^h+\ldots+H_{1}^h-H_i^v),$$
$$P_{\le i,j}:=\OO_{\ov{R}_{\le i,j}}(-H_i^v), \ \ P_{ii}=P_{\le i,i}:=\OO_{\ov{R}_{ii}}(H_i^h-H_i^v),$$
$$P'_{\le i,j}:=\OO_{\ov{R}_{\le i,j}}, \ \ P'_{ii}=P'_{\le i,i}:=\OO_{\ov{R}_{ii}}.$$
Note that for $i>0$,
$$P_{i,\le 0}=P'_{i,\le 0}=P_{i0}, \ \  P_{\le k,j}=P'_{\le k,j}=Q_j.$$
It is also convenient to set $P'_{ij}=P_{ij}$ for $i>j$.

\begin{lem}\label{n2-resolution-lem}
(i) We have the following exact triangles, for $i>j$,
$$P_{ij}[-1]\to P_{\le i-1,j}\to P_{\le i,j}\to\ldots,$$
$$P'_{\le i,j}\to P'_{\le i-1,j}\to P_{ij}[1]\to\ldots,$$
and for $i\ge j$,
$$P_{i,\le j-1}\to P_{i,\le j}\to P_{ij}\to\ldots,$$
$$P'_{ij}\to P'_{i,\le j}\to P'_{i,\le j-1}\to\ldots,$$
which recursively express the objects $(P_{\le i,j})$, $(P_{i,\le j})$ in terms of the generators.
Finally, we have the exact triangles, for $i>0$,
$$P'_{i,\le i}\to P_{i,\le i}\to \OO_{D_i}(-1),$$
and
$$P'_{00}\to P_{00}\to \OO_{D_0}(-1).$$

\noindent
(ii) One has $\OO_D\simeq f_2^*\OO_p$, $Q_j=f_2^*\OO_{R_j}$, so the pull-back of the exact triangles
of Lemma \ref{n1-resolution-lem} under $f_2$ give a resolution for $\OO_D$.
\end{lem}

\begin{proof} These are pretty straightforward. The only nontrivial point is to use the following relations between divisors
on $R_{ii}$ (and $\ov{R}_{00}$), which is isomorphic to the blow up of a point on $\A^1\times\P^1$.
Let us set $E_i^v:=R_{ii}\cap R_{i+1,i}$, $E_i^h:=R_{ii}\cap R_{i,i-1}$. Then one has
$$\OO_{R_{ii}}(E_i^v)\simeq \OO_{R_{ii}}(-E_i^h)\simeq \OO_{R_{ii}}(H_i^v-H_i^h).$$
\end{proof}
  
\section{Matching the A-side with the B-side}

\subsection{Equivalence of the Fukaya category with one stop with the compactified LG-model}

Here we prove Theorem \ref{main-thm-bis}.

First, we observe that we can match the grading groups. Namely, for this we use canonical decompositions
$$\bbL=\Z\oplus H,$$
$$\bbL_B=\Z\cdot l_0\oplus 2H,$$
where $H=H_1(\Sigma)$ (see Sec.\ \ref{Fuk-gen-sec}). The isomorphism $\bbL\to \bbL_B$ sends
$(n,h)$ to $nl_0+2h$. Recall that by definition the homomorphism $\si:\bbL\to \Z/2$ sends $H$ to $0$,
hence our isomorphism is compatible with the homomorphisms to $\Z/2$.

In the case $n=1$, we consider the correspondence between the generating objects
$$P_i\leftrightarrow L_i, \ i=0,\ldots,k.$$
It is easy to see that it extends to the isomorphism between the corresponding (cohomology) endomorphism algebras,
so that endomorphisms $U_i$ and $X_1$ on the B-side correspond to $u_i$ and $x_1$ on the $A$-side, 
and $a_i:P_i\to P_{i-1}$ (resp., $b_i:P_{i-1}\to P_i$)  corresponds to $r_i:L_i\to L_{i-1}$ (resp. $\ell_i:L_{i-1}\to L_i$).
It is easy to see that there is a unique (up to common shift by an element of $\bbL$) choice of gradings of $L_i$, so that
the gradings match. Since both endomorphism algebras are formal, the equivalence follows.

In the case $n=2$, we consider the following correspondence between the generating objects
\begin{equation}\label{objects-correspondence-eq}
\begin{array}{l}
P_{ij}\leftrightarrow L_{i}\times L_{j}, \ \text{ for } 0\le j<i\le k, \nonumber\\
Q_j\leftrightarrow L_{j}\times L_{k+1}, \ \text{ for } 0\le j\le k.
\end{array}
\end{equation}

A straightforward check shows that the morphisms on both sides match, where the generators 
$u_i$ (resp., $x_1,x_2$) on the A-side correspond to $U_{i}$ (resp., $X_1,X_2$) on the B-side.

By \cite[Thm.\ 3.2.5(iii)]{LPsymhms}, up to shifts of the generators, the grading on the A-side is uniquely determined by the degrees of $u_r$ and $x_r$.
Hence, by choosing appropriate shifts, we can make it match with the grading on the B-side. Again both endomorphism algebras are formal, so
the equivalence follows.

\subsection{Matching the localizations}  

Now we can deduce Theorem \ref{main-thm} by localization.

\subsubsection{Case $n=1$}

We have to match the resolution \eqref{n1-T-resolution} for the object $T$ in the Fukaya category with the resolution for $\OO_p$ obtained 
in Lemma \ref{n1-resolution-lem}. The exact triangles of that Lemma show recursively the following correspondence between the A-side and the B-side:
$$\OO_{R_{[0,i]}}(1,\ldots,1,0)\leftrightarrow [L_i\rTo{r_i}\ldots\rTo{r_1} L_0], \ i<k,$$
$$\OO_{R_{[0,k]}}(1,1,\ldots,1)\leftrightarrow [L_k\rTo{r_k}\ldots\rTo{r_1} L_0],$$
$$\OO_{R_{[0,i]}}(0,1,\ldots,1)\leftrightarrow [L_0\rTo{\ell_1}\ldots\rTo{\ell_i} L_i], \ i\le k.$$
Finally, we can identify the morphism 
$$\OO_{R_{0,k}}(0,1,\ldots,1)\to \OO_{R_{[0,k]}}(1,\ldots,1)$$
in the resolution for $\OO_p$ with the degree $0$ morphism of twisted complexes
$$[L_0\to \ldots \to L_k]\to [L_k\to \ldots \to L_0]$$
given by $x_1:L_k\to L_k$.
The cone of this morphism is exactly the resolution of the arc $T$, which therefore corresponds to $\OO_p$.


\subsubsection{Case $n=2$}

We will match the resolutions of $L_i\times T$ obtained from \eqref{n2-T-resolution} with the resolutions in Lemma \ref{n2-resolution-lem}.
The exact triangles of that Lemma show recursively the following correspondences for $i\le k$:
\begin{itemize}
\item for $j=0,\ldots,i-1$,
$P_{i,\le j}\leftrightarrow [L_{i,j}\to\ldots \to L_{i,0}]$,
\item for $r=k-1,\ldots,i$,
$P_{\le r,i}\leftrightarrow [L_{k+1,i}\to L_{k,i}[-1]\to\ldots\to L_{r+1,i}[-1]]$,
\item for $j=0,\ldots,i-1$,
$P'_{i,\le j}\leftrightarrow [L_{i,0}\to \ldots \to L_{i,j}]$,
\item for $r=k-1,\ldots,i$,
$P'_{\le r,i}\leftrightarrow [L_{r+1,i}[1]\to\ldots\to L_{k,i}[1]\to L_{k+1,i}]$.
\end{itemize}
Next, using the exact triangles for $P_{i,\le i}$ and $P'_{i,\le i}$ we sew together the above complexes and get the correspondences
$$P_{i,\le i}\leftrightarrow [L_{k+1,i}\to L_{k,i}[-1]\to\ldots\to L_{i+1,i}[-1]\to L_{i,i-1}\to\ldots \to L_0],$$
$$P'_{i,\le i}\leftrightarrow [L_{i,0}\to\ldots\to L_{i,i-1}\to L_{i+1,i}[1]\to\ldots\to L_{k,i}[1]\to L_{k+1,i}].$$  
Finally, using the exact triangle for $\OO_{D_i}(-1)$, we see that it corresponds to the complex \eqref{n2-T-resolution} representing
$T\times L_i$.

For $i=k+1$, similarly to the case $n=1$, we check that the complex \eqref{n2-T-resolution-bis} corresponds to the resolution for $\OO_D$
in Lemma \ref{n2-resolution-lem}(ii).

\end{document}